\documentclass[12pt,a4paper]{amsart}%
\usepackage{imakeidx}
\usepackage{amsmath}
\usepackage{fullpage}
\usepackage{comment}
\usepackage{MyCommA}
\usepackage[final]{pdfpages}
\usepackage{tikz-cd}

\makeindex
\usetikzlibrary{calc}

\newcommand{\Dima}[1]{{{#1}}}
\newcommand{\DimaA}[1]{{{#1}}}
\newcommand{\DimaB}[1]{{{#1}}}
\newcommand{\DimaC}[1]{{{#1}}}
\newcommand{\Rami}[1]{{{#1}}}
\newcommand{\RamiA}[1]{{{#1}}}
\newcommand{\RamiC}[1]{{{#1}}}
\newcommand{\Eitan}[1]{{{#1}}}

\newcommand{\NextVer}[1]{}
\newcommand{\sub}{\subset}
\newcommand{\noleft}{\left.\kern-\nulldelimiterspace}

\newcommand{\alp}{\alpha}

\newtheorem{innercustomthm}{Theorem}

\newenvironment{customthm}[2][]{%
  \renewcommand{\theinnercustomthm}{#2}
  \begin{innercustomthm}[#1]
}{%
  \end{innercustomthm}%
}

\global\long\def\lct{\operatorname{lct}}%

\begin{document}
	
	\author[Aizenbud]{Avraham Aizenbud}
	\address{Avraham Aizenbud,
		Faculty of Mathematical Sciences,
		Weizmann Institute of Science,
		76100
		Rehovot, Israel}
	\email{aizenr@gmail.com}
	\urladdr{https://www.wisdom.weizmann.ac.il/~aizenr/}
	
	\author[Gourevitch]{Dmitry Gourevitch}
	\address{Dmitry Gourevitch,
		Faculty of Mathematical Sciences,
		Weizmann Institute of Science,
		76100
		Rehovot, Israel}
	\email{dimagur@weizmann.ac.il}
	\urladdr{https://www.wisdom.weizmann.ac.il/~dimagur/}

\author[Kazhdan]{David Kazhdan}
	\address{David Kazhdan,
    Einstein Institute of Mathematics, Edmond J. Safra Campus, Givaat Ram The
Hebrew University of Jerusalem, Jerusalem, 91904, Israel}
	\email{david.kazhdan@mail.huji.ac.il}
	\urladdr{https://math.huji.ac.il/~kazhdan/}
	
\author[Sayag]{Eitan Sayag}
	\address{Eitan Sayag,
    Department of Mathematics, Ben Gurion University of the Negev, P.O.B. 653,
Be’er Sheva 84105, ISRAEL}
	\email{eitan.sayag@gmail.com}
	\urladdr{www.math.bgu.ac.il/~sayage}
	
	\date{\today}

	%
	%
	%
	%
	%
	%
	%
	%


\title{On Harish-Chandra's integrability theorem in positive characteristic}
\maketitle
	\begin{center}
\textit{With an appendix by I. Glazer and Y. Hendel}
\end{center}

\begin{abstract} 
The celebrated Harish-Chandra's integrability theorem states that
the distributional character
of an irreducible smooth representation of a p-adic group $\bfG(F)$ is integrable, that is represented by an $L^{1}_{loc}(\bfG(F))$ function. Here $F$ is a non-Archimedean local field of characteristic $0$
and $\bfG$ is a reductive algebraic group defined over $F$.
In this paper we focus on cuspidal representations of $\GL_n(F)$ for a field $F$ of positive characteristic. We show that in this case the integrability holds under the hypothesis of existence of desingularization of (certain) algebraic varieties in  positive characteristics. 

Furthermore, in the case  $\chara(F)>\frac{n}{2}$  
we establish the regularity of such characters unconditionally.


\end{abstract}

\tableofcontents 
\section{Introduction}
Throughout the paper we fix a non-Archimedian local field \tdef{$F$}
of arbitrary characteristic. Denote by \tdef{$\ell$} the size of the residue field of $F$. All the algebraic varieties and algebraic groups that we will consider are defined over $F$. 
We will also fix a natural number $n$ and set $\mdef{\bfG}=\RamiA{\GL_n}$\RamiA{, considered as an algebraic group defined over $F$}.
 Denote $\mdef{G}=\bfG(F)$.

 We will denote by $C^{-\infty}(G)$ the space of generalized functions on $G$\RamiC{, {\it i.e.} functionals on the space of smooth compactly supported measures}. 
We also denote by $L_{loc}^1(G)$ the space of locally $L^1$-functions on $G$ and consider it as a subspace of the space of generalized functions $C^{-\infty}(G)$ in the usual way.

\subsection{Main results}
We study the following conjecture:
\begin{introconj}\label[introconj]{conj:1}
Let $\RamiA{\rho}$  be an \RamiA{irreducible} cuspidal \RamiA{smooth} representation \DimaB{of $G$} and  let $\mdef[\chi]{\chi_\rho}\in C^{-\infty}(G)$ be its character. Then $\chi_\rho\in L_{loc}^1(G)$. 
\end{introconj}
When the characteristic of $F$ is zero, this is a special case of a well known result of Harish-Chandra \cite{HC_VD}. 
In this paper we show that this conjecture follows from the conjectural existence of resolution of singularities in positive characteristic. 

More precisely, consider the following:
\begin{introconj}\label{conj:2}
    Let $\bfZ$ be an algebraic variety defined over the finite field $\F_\ell$. Then there exists a proper birational map $\gamma:\tilde\bfZ\to \bfZ$ s.t. 
    \begin{itemize}
        \item $\tilde\bfZ$ is smooth.
        \item $\gamma$ is an isomorphism outside the singular locus of $\bfZ$. 
        \item The preimage of the singular locus of $\bfZ$ (considered as a subvariety of $\tilde\bfZ$) is a strict normal crossings divisor.
    \end{itemize}
\end{introconj}
In this paper we prove:
\begin{introtheorem}[\S\ref{Sec: proof of Theorem C and D}]\label{thm:main}
    \Cref{conj:2} implies \Cref{conj:1}.
\end{introtheorem}
We also prove the following unconditional result:
\begin{introprop}[\S\ref{Sec: proof of Theorem C and D}]\label{thm:minor}
   If $\chara(F)>\frac n2$ then  \Cref{conj:1} holds.
\end{introprop}
\begin{remark}
    In fact, for given $F$ and $n$ it is enough to assume \Cref{conj:2} for a specific variety defined over $\F_\ell$. We also give some other alternatives \RamiC{that replace the role of} 
    \Cref{conj:2} \RamiC{in \Cref{thm:main}}, see \S \ref{sec:AltMain}.
\end{remark}
\begin{remark}
We also prove analogues of \Cref{thm:main} and \Cref{thm:minor} for orbital integrals. 
See \S\ref{subsec:OrbInt} below.
\end{remark}
\subsection{Background}
\subsubsection{Previous results}
In \cite[Theorem 2.2]{CGH} it was established that \RamiA{local} integrability of characters of irreducible representations of reductive groups over $\F_{\ell}((t))$ holds true for large enough characteristics (depending on the group $G$). However, no explicit bound was given. 

\RamiC{The case of $\GL_2(F)$ was already proven in \cite[Chapter 9]{JL}.}

In \cite{rod} 
it was established that \RamiC{local} integrability of characters of irreducible representations of $\GL_n(\F_{\ell}((t)))$ holds true in neighborhoods of elements with separable characteristic polynomials. In particular the \RamiC{local} integrability holds whenever $\chara(\F_\ell)>n$.

In a series of papers (\cite{Le4}, \cite{Le2}, \cite{Le3})
  it was claimed that local integrability holds true in  arbitrary characteristics for the family of groups $\GL_n(F), \GL_n(D),\SL_N(D)$ where $F=\F_{\ell}((t))$ is a local non-Archimedean field and $D$ a division algebra over $F$. However the arguments in these papers \RamiC{have a flaw}. See more detailed explanation in \Cref{app:lem}.

  
  On the other hand, it seems that the argument \RamiA{in} \cite{Le4} can give a proof for \Cref{thm:minor} of the present paper.

\subsubsection{The original \RamiA{argument of} Harish-Chandra}\label{subsub:HCpf}
Let us shortly present the main parts of the original Harish-Chandra's proof of the local integrability of cuspidal characters from \cite{HC_VD}.
\RamiC{This presentation differs slightly from the original, as it is  adapted to better suit our purposes.}
One can roughly divide Harish-Chandra's proof into two parts:
\begin{enumerate}[$\quad$(1)]
    \item\label{it:intro.1} Bound the character (up to a logarithmic factor) \RamiA{by} the inverse square root of the discriminant  --- $|\Delta|^{-\frac{1}{2}}.$ 
    \item\label{it:intro.2} Prove the integrability of  
     $|\Delta|^{-\frac{1}{2}}$.
\end{enumerate}
In more details, \RamiC{let $p:G\to C:=(\bfG//Ad(\bfG))(F)$ be the Chevalley map.} one can divide the first step into the following sub-steps:

\begin{enumerate}[$\quad$(a)]
    \item\label{it:intro.a} Locally bound the character by the orbital integral $\Omega(f)$ of a smooth function $f\in C_c^\infty(G)$ (up to a logarithmic factor). See \Cref{not:Omega} for the definition of $\Omega(f)$. We did this in \cite{AGKS2_1}.
    \item\label{it:intro.b} Bound the orbital integral $\Omega(f)$ by a product $|\Delta|^{-\frac{1}{2}}\cdot p^*(p_*(f))$ 
    where,
        \RamiC{the push forward is taken w.r.t. some fixed, smooth, nowhere vanishing measures on $G$ and $C$.}
    
    \item\label{it:intro.c} Bound $p_*(f)$.
\end{enumerate}

\subsubsection{Difficulties with Harish-Chandra's argument in positive characteristic}
Step (2) does not hold in positive characteristic (even for \RamiA{the case of} $\GL_2\RamiC{((t))}$). So, one should replace $|\Delta|^{-\frac{1}{2}}$ with a better bound (like the function $\kappa$ described in \S\ref{sec:bnd} below).

Both substep \eqref{it:intro.1}\eqref{it:intro.a} and step \eqref{it:intro.2}  are 
done for each torus in $G$ separately. This is \RamiC{enough} in characteristic zero, as there are only finitely many conjugacy classes of tori. However, \RamiC{the latter} is no longer true in positive characteristic. \RamiC{See more details in \cite[\S1.5]{AGKS2}}

Substep \eqref{it:intro.1}\eqref{it:intro.c} uses the assumption on characteristic in many places. See more details in \cite[\S 1.5.1]{AGKS2}.
\subsubsection{The approach of \cite[\RamiC{Chapter 9}]{JL}}
The  proof of \cite[\RamiC{Chapter 9}]{JL} in the $\GL_2$ case goes essentially along the same lines \RamiC{as the proof of \cite{HC_VD}}. All the bounds are much more explicit, and the bound $|\Delta|^{-\frac{1}{2}}$ is replaced by a different bound which differs from $|\Delta|^{\Rami{-}\frac 12}$ by a multiplicative constant on each torus.
\subsubsection{Results of \cite{AGKS2}}
In \cite{AGKS2} we obtain bounds for $p_*(f)$. These bounds are conditional on the assumption of existence of a resolution of singularities or the assumption $\chara(F)>n/2$ as in \Cref{thm:main} and \Cref{thm:minor}. 

In fact, the only reason that we need the assumptions above is the fact that  we rely on the results of \cite{AGKS2}. 
\subsubsection{The approach of  \cite{rod}}
\cite{rod} took a different approach. Instead of bounding $\Omega(f)$ and then bounding the character using it, they bound the character directly. They do it using a formula of Howe, that expresses the character (near $1$) as a combination of the Fourier transform of nilpotent orbital integrals. Then they use the fact that \DimaB{all the} nilpotent orbits \DimaB{of $\GL_n$} are Richardson, in order to prove that these Fourier transforms are locally integrable. 

\cite{rod} adapted this argument to work near semi-simple elements, thus covers all elements with separable characteristic polynomial, and therefore proves the result whenever $\chara(F)>n$. If one would like to adapt the argument in \cite{rod} to the general case, one has to deal with closed orbits with non-separable characteristic polynomial, like  the orbit of  $$\begin{pmatrix}
    0 &1\\ t &0
\end{pmatrix}\in \GL_2(\F_2((t))).
$$ 
\RamiC{Such an adaptation was attempted in \cite{Le4}.} 
\RamiC{A similar approach to the one in \cite{rod}} was used in \cite{HC_SD} \RamiC{(for the characteristic $0$ case)} in order to show local integrability for general (not necessarily cuspidal) characters. However, since \cite{HC_SD} is not limited to the generality of $\GL_n$ it could not use the Richardson property of the nilpotent orbits, and thus had to prove the local integrability of the Fourier transforms of nilpotent orbital integrals in a different way. This is done \Dima{using} the local integrability of $\Omega(f)$ proven in \cite{HC_VD} \RamiC{(for the characteristic $0$ case)}.
\subsection{Our approach}
Our approach follows the original  approach \RamiA{of Harish-Chandra} \RamiC{(for the cuspidal case)}, thus we circumvent the need to deal directly with elements with non-separable characteristic polynomial. Also, this approach gives a bound on $\Omega(f)$ and not only on the character. Additionally, it does not use the fact that \DimaB{all the} nilpotent orbits of $\GL_n$ are Richardson (see \S\ref{sssec:gl} below).

We replace $|\Delta|^{-\frac{1}{2}}$ with a function $\kappa$  described in \S\ref{sec:bnd} below. One can write $\kappa=\RamiC{\kappa^0} |\Delta|^{-\frac{1}{2}}$ where $\RamiC{\kappa^0}$ is $Ad(G)$-invariant and constant on any torus.
Thus the difference between 
$|\Delta|^{-\frac{1}{2}}$
and $\kappa$
is almost invisible in the characteristic zero
case. The construction of 
$\kappa$ 
generalizes the construction of the bound from \cite[\RamiC{Chapter 9}]{JL}.

Roughly speaking, {our} general strategy is to replace the 
{torus-by-torus} arguments \RamiC{(from \cite{HC_VD})} with global geometric arguments. Let us describe it in more details.


The original proof of substep \eqref{it:intro.1}\eqref{it:intro.a} 
is based on an effective bound on the averaging \RamiC{(w.r.t. the adjoint action)} of a matrix coefficient of a cuspidal representation and the stabilization of that averaging.
 We had to redo this bound in a way that is uniform on the entire group and not only on a single torus. 
 We did this in \cite{AGKS2_1}.

Substep \eqref{it:intro.1}\eqref{it:intro.b} 
\RamiC{in the argument in \cite{HC_VD}}
is rather straightforward. However,
\RamiC{as explained above, it would not be enough just to adapt it to positive characteristic as is. I}n order to make step \eqref{it:intro.2}  possible we \RamiC{replace} the function $|\Delta|^{-\frac{1}{2}}$ with the function $\kappa$. After this change, 
 the proof of 
substep \eqref{it:intro.1}\eqref{it:intro.b} \RamiC{(in arbitrary characteristic)} becomes more subtle and we do it in \S \ref{sec:bnd}. 

We dealt with substep \eqref{it:intro.1}\eqref{it:intro.c} in \cite{AGKS2}, note that this is \RamiC{the first} of the \Rami{two} places where we use \Rami{\cite{AGKS2}, which in turn depends on the} assumption of resolution of singularities.

So we are left with the adapted version of step \eqref{it:intro.2}: we have to prove that $\kappa$ is locally integrable. Here also, the original proof of Harish-Chandra \RamiC{treated} each torus separately. In case $n=2$  one can obtain \RamiA{a} bound on the integral on each torus  separately that will lead to the convergence of the entire integral. This is essentially what is done in \cite[\RamiC{Chapter 9}]{JL}. In the general case, we could not do it. Instead we developed a geometric formula for $\kappa$ (see \S\ref{sec:kap.form}). Essentially, this formula presents $\kappa$ as a pushforward of an (\Rami{a priori} not necessarily locally finite) measure $m$ w.r.t. a morphism $\tau:\bfX\to \bfG$ for a certain variety $\bfX$.  The  measure $m$ on $\bfX(F)$ is given by a (rational) form $\omega_\bfX$ on $\bfX$. To make this formula useful we have to prove that \Dima{$\omega_\bfX$} is regular on the smooth locus of $\bfX$ (see \S\ref{Sec:regularity}). Finally we prove \Rami{that} $m$ is locally finite and use this \RamiC{geometric} formula to prove the local integrability of $\kappa$. Here we again used the results of \cite{AGKS2} \Rami{(}and hence the assumption of existence of a resolution\Rami{)}.

Therefore, the main innovation of this paper 
is the factor $\kappa$, the geometric formula for it, and 
the successful execution of step \eqref{it:intro.2}. 

\subsubsection{The role of the assumption $\bfG=\GL_n$}\label{sssec:gl}
We used the assumption $\bfG=\GL_n$ in order to make all explicit computations easier. However, our argument does not use any statement that inherently depend on this assumption (such as existence of mirabolic subgroup, stability of adjoint orbits, or the Richardson property of all nilpotent orbits).

We also use the results of \cite{AGKS2,AGKS2_1} that are limited to the $\GL_n$ case, however the situation there is similar (see \cite[\S1.5.7]{AGKS2}, \cite[\S1.5.1]{AGKS2_1}). 

In conclusion we expect that the methods of the present paper can provide a proof of the regularity of characters of cuspidal representations for any reductive group over a non-Archimedean local field $F$ of good characteristic (see {\it e.g.}
\cite[I, §4]{SS} for this notion). 

\subsection{\Rami{Statements for the orbital integrals}}\label{subsec:OrbInt}
\RamiC{\Cref{thm:main} and \Cref{thm:minor} are also} valid  when we replace the character of $\rho$ with the orbital integral of a function $f\in C_c^\infty(G)$. \RamiC{Let us recall the notion of orbital integral of a function.}
\begin{notation}\label{not: orb.int}
$ $
\begin{itemize}
    \item Denote by $G^{rss}$ the collection of regular semi-simple elements in $G$.
\item
    For $f\in C_c^\infty(G)$ denote \Dima{by }$\Omega(f)\in C^\infty(G^{rss})$ \Dima{the orbital integral} $$\Omega(f)(x)=\int_{y\in G\cdot x} f(y)dy$$ where $dy$ is an appropriate measure on $G\cdot x$, see \Cref{not:Omega} below. 
\end{itemize}
\end{notation}
\begin{introtheorem}[
\Cref{rem:intro.Omega1}]\label{thm:intro.Omega}
\RamiC{Assume either \Cref{conj:2} or $\chara(F)>\frac n2$.}
    Let $\gamma\in C_c^\infty(G)$. 
        \RamiC{Then} $\Omega(\gamma)\in L^1(G)$.    
\end{introtheorem}
It is easy to see that this theorem implies its version for the Lie algebra $\fg$ of $G$. Namely we have:
\begin{customthm}{\ref{thm:intro.Omega}'}\label{thm:intro.Omega.p}
For $\gamma\in C_c^\infty(\fg)$ define $\Omega(\gamma)$ analogously to the case when $\gamma\in C_c^\infty(G)$.  
Then \Cref{thm:intro.Omega} is valid 
\RamiC{with $G$ replaced with $\fg$.}
\end{customthm}
In view of \cite[Theorem A']{AGKS2_1} this theorem implies a version of the main results for Fourier transforms of \RamiC{characteristic measures of elliptic orbits. Namely, f}or a regular semi-simple element $x\in \fg$, fix an $ad(G)$-invariant measure on $\fg$ supported on the adjoint orbit $G\cdot x$, and denote it by $\mu_{G\cdot x}$.
Let $\hat \mu_{G\cdot x}$ be its Fourier transform.
\begin{introtheorem}[\Cref{rem:intro.Omega1}]
\label{thm:Lie}
\Cref{thm:main} \RamiC{and} \Cref{thm:minor} are valid when we replace $\chi_{\rho}$ with $\hat \mu_{G\cdot x}$ for elliptic (regular semi-simple) $x\in \fg$ (with the obvious modifications).
\end{introtheorem}
Moreover, the arguments of \cite[\S1.4]{HC_SD} (which are also valid for positive characteristic) allow to deduce from this theorem the following one.
\begin{customthm}{\ref{thm:Lie}'}\label{thm:Lie.p}
\Cref{thm:Lie} is valid
\RamiC{when we replace $x$ with any} regular semi-simple \Rami{element in} $ \fg$.
\end{customthm}
This theorem is a partial positive characteristic analog of 
\cite[Theorem 1.1]{HC_SD} that states that $\hat \mu_{G\cdot x}\in L^{1}_{loc}(\fg)$. Harish-Chandra used this result in order to prove that the character of an arbitrary irreducible (smooth) representation of $G$ is locally integrable \cite[Theorem 16.3]{HC_SD}. However, \RamiC{at this point, we do not know how to adapt this part of Harish-Chandra's argument to} positive characteristic, so we still can not prove local integrability for  character of an arbitrary irreducible (smooth) representation in positive characteristic even under our additional assumptions.

\subsection{\Rami{Unconditional} results}
We prove \Cref{thm:main}  using  an unconditional bound on the character of a cuspidal representation. 
In order to formulate it we need the following notation:
\begin{notation}
    We denote by:
    \begin{enumerate}
        \item \tdef{$\bfC$}  $-$ the variety of monic polynomials of degree $n$ that do not vanish at $0$. We will identify it with $\bG_m\times \A^{n-1}$. 
        \item $\mdef{C}:=\bfC(F)$.   
        \item $\mdef{p}:\bfG\to \bfC$ $-$ the Chevalley map, i.e. the map that sends a matrix to its characteristic polynomial.
        \item $\mdef{\mu_G}$ - the Haar measure on $G$, normalized on a maximal compact subgroup of $G$.
        \item  $\mdef{\mu_C}$ - the Haar measure on $C$, given by the identification $C\cong F^\times \times F^{n-1}$, normalized on  
        $ O_F^\times \times O_F^{n-1}$, where \tdef{$O_F$} is the ring of integers in $F$.
    \end{enumerate}
\end{notation}

\begin{introthm}[\S\ref{Sec: proof of Theorem E}]\label{thm:bnd1}
Let $\rho$ be an irreducible cuspidal representation of $G$ and $U\subset G$ be an open compact subset. Then there exist\RamiC{:}
\begin{enumerate}
    \item $\eps>0$,
    \item a real valued non-negative $f\in L^{1+\eps}(C)$, and
    \item a real valued non-negative $h\in C_c^\infty(G)$
\end{enumerate}
such that for any $g\in C_c^\infty(U)$ we have:
$$|\langle \chi_\rho, g\rangle|\leq  \langle p^*(fp_*(h)), |g|\rangle.$$    
More precisely:
$$|\langle \chi_\rho, g\mu_G\rangle|\leq  \left\langle p^*\left(f\frac{p_*(h \mu_G)}{\mu_C}\right)\mu_G, |g|\right\rangle.$$ 
\end{introthm}
 \begin{remark}     
 Note that the Radon-Nikodym derivative $\frac{p_*(h \mu_G)}{\mu_C}$ does not have to be bounded (or finite) but only measurable, so the measure on the RHS does not have to be locally finite. Hence, a-priory, the RHS might be infinite (in this case, the statement is void).
\end{remark}
\Cref{thm:main}  follows from \Cref{thm:bnd1} using  the following weaker version of   \cite[Theorem D]{AGKS2}:
\begin{theorem}[{cf. \cite[Theorem D]{AGKS2}}]\label{thm:alm.an.frs}  
Assume \Cref{conj:2}.    
Then for any $t\in[1,\infty)$ and any smooth compactly supported measure $\mu$  on $G$, we have  $p_*(\mu)=f\mu_C$ for some  $f\in L^{t}(C)$.
\end{theorem}

Similarly, \Cref{thm:minor} follows from \Cref{thm:bnd1} using  the following special case of \cite[Theorem E]{AGKS2}:
\begin{theorem}[{cf. \cite[Theorem E]{AGKS2}}]\label{thm:uncond.an.frs}
Suppose $\chara(F)>\frac{n}{2}$.  
Then for any smooth compactly supported measure $\mu$  on $G$, the measure $p_*(\mu)$ can be written as a product of a  function in  $L^{\infty}(C)$ and a Haar measure on $C$.
\end{theorem}

 
In fact, we prove a more explicit \RamiC{version of the bound in \Cref{thm:bnd1}}.
In order to formulate it we will need the following notation. 

\begin{notation} Denote: 
    \begin{enumerate}
        \item By \tdef{$\bfT$} the standard maximal torus of $\bfG$. 
        \item By $\mdef{W}\cong S_n$ the Weyl group.
        \item We will 
        identify the Chevalley space \tdef{$\bfC$} with \RamiC{the categorical quotient} $\bfT//W$.
        \item By $\mdef{\bfY}:=(\bfT\times\bfT)//W$  the categorical quotient by the diagonal action.\footnote{See \S\ref{sec:fac} below for its existence.}
        \item By $\mdef\pi:\bfY\to \bfC$   the projection to the first coordinate. 
        \item $\mdef{{\bf \Upsilon}}:=\bfG\times_{\bfC}\bfG\times_{\bfC}\bfY.$
\item By $\mdef{\zeta}:{\bf \Upsilon}\to \bfG$ the projection on the second coordinate.
\item  By  $\mdef{\Delta}\in \DimaB{\cO}_\bfG(\bfG)$ the discriminant, i.e. $\Delta(g)$ is the discriminant of the characteristic polynomial of $g$.
\item  By $\mdef{\cR}:G \to \N \cup \{\infty\}$ the function given by $$\mathcal R(x)=\max(0,-\min val(x_{ij}),val(\det(x)),val(\Delta(x))).$$
    \end{enumerate}
\end{notation}
\begin{remark}
    Throughout the paper we use various notations for specific varieties, sets and maps between them. We summarize these objects in some diagrams in \Cref{sec:diag}. It might be easier to follow some parts of the paper with these diagrams visible. Of course we will not rely on this, and all the objects will be defined before their first use. 
\end{remark}

\begin{introtheorem}[\S \ref{sec:Pfbnd2}]\label{thm:bnd2}
Let $\rho$ be an irreducible cuspidal representation of $G$. Then there exist:
\begin{enumerate}
    \item a real valued non-negative  function $e\in C^{\infty}({\bf \Upsilon}(F))$ such that $\zeta|_{\Supp e}$ is proper,
\item a top differential form $\omega$ on the smooth locus of ${\bf \Upsilon}$, and
\item an integer $k$
\end{enumerate}
such that for any $g\in C_c^\infty(G)$ we have:
$$|\langle \chi_\rho, g\mu_G\rangle|\leq  \langle \zeta_*(|\omega|e)\cR^k, |g|\rangle.$$

\end{introtheorem}
In order to deduce \Cref{thm:bnd1} from \Cref{thm:bnd2} we prove another statement (\Cref{thm:Y.geo.int} below) about the geometric structure of $\bfY$ and use a general result about integrability of pushforward of a smooth measure under a dominant morphism (\Cref{thm.IT} below). In order to formulate these results we make the following:
\begin{definition}
        We say that an algebraic variety $\bfZ$ is \tdef{ geometrically integrable} if there exists a resolution of singularities $\gamma:\tilde \bfZ\to \bfZ$ s.t. the natural morphism        
             $\gamma_*(\Omega_{\tilde \bfZ})\to i_*(\Omega_{\bfZ^{sm}})$ is an isomorphism. Here\Dima{ 
           $\bfZ^{sm}$ is the smooth locus of $\bfZ$, and  $i:\bfZ^{sm} \hookrightarrow \bfZ$ is the embedding.}
\end{definition}
\begin{introprop}[\S\ref{sec:Y}]\label{thm:Y.geo.int}
    The variety $\bfY\RamiC{=(\bfT\times\bfT)//W}$ is geometrically integrable.
\end{introprop}
\begin{remark}
$ $
\begin{itemize}
    \item
    In characteristic zero case, the singularities of a   variety  are rational iff it is geometrically integrable and Cohen-Macaulay (see e.g. \cite[{Appendix B,} Proposition 6.2]{AA}). 
    \item 
    In characteristic $0$, \Cref{thm:Y.geo.int} follows immediately from the fact that a quotient singularity is rational
(see \cite[Corollaire]{Boutot}).
\item \RamiC{We do not know whether $\bfY$ is Cohen-Macaulay (in positive characteristic.)}
\end{itemize}

\end{remark}
In order to deduce \Cref{thm:bnd1} from \Cref{thm:bnd2} and \Cref{thm:Y.geo.int} we need the following:
\begin{introprop}[\Cref{thm:main theorem for alg maps}]\label{thm.IT}
    Let $\gamma:\bfM\to \bfN$  be a generically smooth morphism of smooth irreducible algebraic varieties. Then there exists $\eps>0$ s.t. for 
    any smooth compactly supported measure $\mu_M$ on $M:=\bfM(F)$ there exist smooth compactly supported measure $\mu_N$ on $N:=\bfN(F)$ and a  function $f\in L^{1+\eps}(N)$ such that $$\RamiA{\gamma}_*(\mu_M)=f\mu_N.$$        
\end{introprop}
\RamiC{
\begin{introremark}
    Theorems \ref{thm:bnd1} and \ref{thm:bnd2} have versions for orbital integrals and for Fourier transforms of characteristic measures of regular semisimple orbits analogous to Theorems \ref{thm:intro.Omega.p} and \ref{thm:Lie.p}. The proofs are identical.
\end{introremark}    
}

\subsection{Summary of 
the logic of the paper
}
The following diagram provides a guideline regarding the logic of the proofs of the main results of the paper.
$$
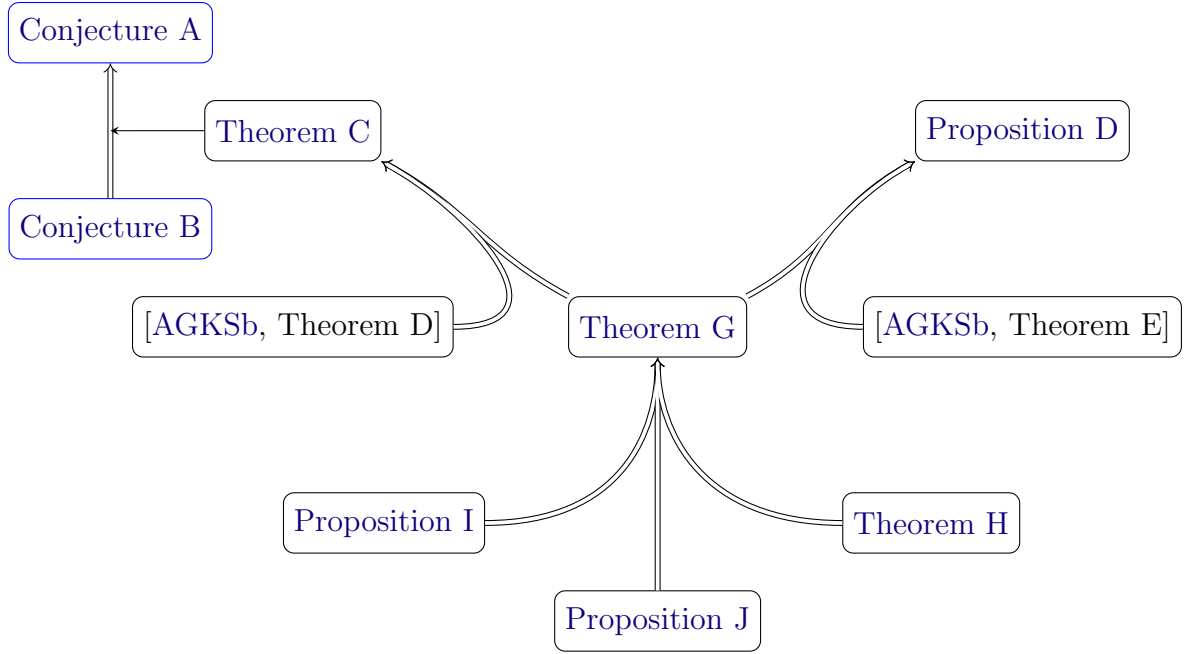
\begin{tikzpicture}[node distance=1.8cm,>=stealth, xscale=2.4, yscale=1.3]

\tikzstyle{rect}=[draw,rounded corners,minimum width=12mm,minimum height=8mm]

\tikzset{
  impraw/.style={
    -{implies[width=7pt]},
    double,
    double distance=0.5mm,
    line width=0.4pt
  }
}

\node[rect,draw=blue] (A) at (-3,3) {\Cref{conj:1}};
\node[rect,draw=blue] (B) at (-3,1) {\Cref{conj:2}};
\node[rect] (C) at (-2,2) {\Cref{thm:main}};
\node[rect] (D) at (2,2) {\Cref{thm:minor}};

\node[rect] (E) at (0,0) {\Cref{thm:bnd1}};
\node[rect] (CC) at (-2,0) {{\cite[Theorem D]{AGKS2}}};
\node[rect] (DD) at (2,0) {{\cite[Theorem E]{AGKS2}}};

\node[rect] (G) at (-1.5,-2) {\Cref{thm:Y.geo.int}};
\node[rect] (H) at (0,-3) {\Cref{thm.IT}};
\node[rect] (F) at (1.5,-2) {\Cref{thm:bnd2}};


\draw[impraw] (B) -- (A);

\draw[impraw] (E.north west) to[out=135,in=-45] (C.south east);

\draw[impraw] (E.north east) to[out=45,in=-135] (D.south west);

\draw[impraw] (G) to[out=0,in=-90] (E.south);
\draw[impraw] (H) to[out=90,in=-90] (E.south);
\draw[impraw] (F) to[out=180,in=-90] (E.south);

\draw[impraw] (CC) to[out=0,in=-45] (C.south east);
\draw[impraw] (DD) to[out=180,in=-135] (D.south west);


\draw[white,line width=0.9pt]
  (G) to[out=0,in=-90] ([yshift=-0.5mm]E.south);

\draw[white,line width=0.9pt]
  (H) to[out=90,in=-90] ([yshift=-0.5mm]E.south);

\draw[white,line width=0.9pt]
  (F) to[out=180,in=-90] ([yshift=-0.5mm]E.south);

\draw[white,line width=1.0pt]
  (E.north west) to[out=135,in=-45] ([xshift=0.4mm,yshift=-0.5mm]C.south east);

\draw[white,line width=1.0pt]
  (E.north east) to[out=45,in=-135] ([xshift=-0.4mm,yshift=-0.5mm]D.south west);

\draw[-stealth] (C) -- ($(B)!0.5!(A)$);

\end{tikzpicture}
$$

\subsection{Ideas of the proofs}
Most of the paper is devoted to the proof of \Cref{thm:bnd2}. The proofs of \Cref{thm:Y.geo.int} and \Cref{thm.IT} are significantly simpler. The rest of the results of the paper follow relatively easily from these 3 results (and the results of \cite{AGKS2}).

\subsubsection{Idea of the proof of \Cref{thm:bnd2}}

In fact, we will prove the following equivalent version of \Cref{thm:bnd2}:

\begin{customthm}[\S \ref{sec:Pfbnd3}]{\ref{thm:bnd2}'}
\label{thm:bnd3}
Let $\rho$ be an irreducible cuspidal representation of $G$. Then there exist:
\begin{enumerate}
    \item a real valued non-negative function $f'\in C^{\infty}({\bfY}(F))$ such that $\pi|_{\Supp f'}$ is proper,
\item a real valued non-negative function $h\in C^{\infty}(G)$ such that $p|_{\Supp h}$ is proper,
\item an invertible top differential form $\omega^0_\bfX$ on the smooth locus of $\bfX:=\bfG\times_{\bfC}\bfY$, 
\item an integer $k$, and
\item a real valued non-negative  function $\gamma\in C^\infty(X)$, where $X:=\bfX(F)$,
\end{enumerate}
such that for any $g\in C_c^\infty(G)$ we have:
$$|\langle \chi_\rho, g \mu_G\rangle|\leq \left  \langle \frac{\tau_*(|\omega^0_\bfX| \gamma \sigma^*(f'))}{\mu_G} p^*\left( \frac{p_*(h\mu_G )}{\mu_C}\right)\cR^k, |g| \mu_G\right\rangle,$$
where $\sigma:\bfX\to \bfY$ \DimaB{and $\tau: \bfX \to \bfG$ are the projections.}
\end{customthm}

We prove this theorem using the following steps:
\begin{enumerate}
    \item Following \cite{HC_VD}, for any function $f$ on $G$ whose support is compact modulo the center we define the orbital integral $\Omega(f)$ which is a function on the set $G^{rss}$ of regular semi-simple elements in $G$.  See \Cref{not:Omega}.
    \item\label{it:steps.F.2} Following \cite{HC_VD} we showed in \cite{AGKS2_1} that the character of a cuspidal representation $\rho$ is bounded by $\Omega(|m|)$ (up to a logarithmic factor), where $m$ is a matrix coefficient of $\rho$. Note that we have to explain what it means for a partially defined function to bound a generalized function. See \Cref{thm:OrbIntBoundChar} below for an exact formulation. 
    \item\label{it:steps.F.3} We construct an explicit  function $\kappa$ on $G^{rss}$ (see \S\ref{sec:bnd} below) and prove  
    that  $\Omega(|m|)$ is bounded by $\kappa\cdot p^*(p_*(|m|))$. Here the pushforward $p_*$ is taken with respect to  appropriate measures.
    \item \label{it:steps.F.4} We study the varieties
    $\bfY=(\bfT\times \bfT)//W$ and  $\bfX=\bfG\times_\bfC \bfY=\bfG\times_\bfC (\bfT\times \bfT)//W$ and construct:
    
        \begin{itemize}
        \item a rational section $\omega^2_{\bfX}$ of the square of the canonical bundle on the smooth locus of ${\bfX}$, and  
        \item an open set $\mathcal{B}\subset {\bfY}(F)$ such that $\pi|_{\mathcal{B}}$ is proper. Here \RamiC{$\pi:\bfY\to \bfC$} is the projection.
    \end{itemize}  
    such that 
    $$\tau_*\left (\left. \sqrt{|\omega^2_\bfX|}\right|_{\sigma^{-1}(\mathcal{B})} \right)=\kappa|\omega_\bfG|.$$
Here:
\begin{itemize}
    \item 
 $\sqrt{|\omega^2_\bfX|}$ is the measure on ${\bfX(F)}$ corresponding to $\omega^2_\bfX$, see \S\ref{ssec:forms} below for precise definition.     
    \item $\omega_\bfG$ is the standard top form on $\bfG$.
\end{itemize}    
See \S\ref{sec:kap.form} for the construction.
\item\label{it:steps.F.5} We prove that the section $\omega^2_\bfX$ is regular. See \S\ref{Sec:regularity} below.
\item We construct an invertible top form $\omega_\bfY$ on the smooth locus of $\bfY$. 
\item\label{it:steps.F.7} We use $\omega_\bfY$ and the standard form $\omega_\bfG$ on $\bfG$ in order to construct an invertible top form $\omega^0_\bfX:= \omega_\bfG\boxtimes_{\omega_\bfC} \omega_\bfY$ on the smooth locus of $\bfX$ (see \Cref{def:boxtimes} below for the notation $\boxtimes_{\omega_\bfC}$). 
\item \label{it:steps.F.8} \RamiC{We use steps \eqref{it:steps.F.5} and \eqref{it:steps.F.7} to note that
s}ince $\omega_\bfX^0$ is invertible and $\omega^2_\bfX$ is regular, the measure
$|\omega_\bfX^0|$ locally dominates $\sqrt{|\omega^2_\bfX|}$.
\item\label{it:steps.F.9} \RamiC{We use steps (\ref{it:steps.F.2},\ref{it:steps.F.3},\ref{it:steps.F.4},\ref{it:steps.F.8}) to} obtain that, up to a logarithmic factor, the character $\chi_\rho$ is bounded by $\tau_*(|\omega_\bfX^0|\cdot 1_{\sigma^{-1}(\cB)}) p^*(p_*(|m|)$.
\item\label{it:steps.F.10} We bound $\tau_*(|\omega_\bfX^0|\cdot 1_{\sigma^{-1}(\cB)})$ by $p^*(\pi_*(|\omega_\bfY| \cdot {1_\cB}))$.
\item We deduce  \Cref{thm:bnd3} from \RamiC{steps (\ref{it:steps.F.9},\ref{it:steps.F.10})}. \item We deduce  \Cref{thm:bnd2}. 

\end{enumerate}
\subsubsection{Idea of the proof of \Cref{thm:Y.geo.int}}
We embed $\bfY$ into the quotient $\DimaB{(\bA^2)^n//S_n}$ and thus reduce to showing the integrabilty of $(\bA^2)^n//S_n$. This we did in \cite{AGKS2_2}. 

\subsubsection{Idea of proof of \Cref{thm:bnd1}}
We first deduce from  \Cref{thm:bnd3} another slightly different version of \Cref{thm:bnd2}:

\begin{customthm}[\S \ref{sec:Pfbnd4}]{\ref{thm:bnd2}''}
\label{thm:bnd4}
Let $\rho$ be an irreducible cuspidal representation of $G$. Then there exist:
\begin{enumerate}
    \item a real valued non-negative  function $f'\in C^{\infty}({\bfY}(F))$ such that $\pi|_{\Supp f'}$ is proper,
\item a real valued non-negative  function $h\in C^{\infty}(G)$ such that $p|_{\Supp h}$ is proper,
\item an invertible top differential form $\omega_\bfY$ on the smooth locus of ${\bfY}$, and
\item an integer $k$
\item a real valued non-negative  function $\gamma\in C^\infty(G)$
\end{enumerate}
s.t.  for any $g\in C_c^\infty(G)$ we have:
$$|\langle \chi_\rho, g \mu_G\rangle|\leq \left  \langle \gamma p^*\left(\frac{\pi_*(|\omega_\bfY|f')}{\mu_C} \frac{p_*(|\omega_\bfG|h)}{\mu_C}\right)\cR^k, |g| \mu_G\right\rangle.$$
\end{customthm}
Then we prove \Cref{thm:bnd1} using the following steps:
\begin{enumerate}
    \item Let $f',h, \omega_\bfY$ be as in \Cref{thm:bnd4}. 
        \item Let $\pi:\bfY\to \bfC$ be the natural map and set $f:= \pi_*(f')$. Here we choose appropriate measures to define the pushforward.
        \item We use \Cref{thm:Y.geo.int} and \Cref{thm.IT} in order to show that $f\in L_{loc}^{1+\eps}$.
        \item We deduce \Cref{thm:bnd1}.
\end{enumerate}

\subsubsection{Idea of proofs of Theorems \ref{thm:main} and \ref{thm:minor}}

Let us start by sketching the proof of \Cref{thm:main}.
\begin{enumerate}
    \item \Cref{thm:alm.an.frs}  and  \Cref{conj:2} imply that
    $p_*$ maps every $L^\infty$ compactly supported function to an $L^t$ function for any $t\in (1,\infty)$.
    
        \item This implies that $p^*$ maps every $L^{1+\eps}$  function to an $L^1_{loc}$ function.
    
    \item Let $f,h$ be as in \Cref{thm:bnd1}. We obtain that $p_*(h)\in L^{t}(\bfC(F))$ for all $t\in (1,\infty)$. Therefore, $fp_*(h) \in L^{1+\delta}(\bfC(F))$ for some $\delta>0$. Thus $p^*(fp_*(h))\in L^{1}_{loc}$ as required.
\end{enumerate}

The proof of \Cref{thm:minor} is the 
same when we replace \Cref{thm:alm.an.frs} 
by \Cref{thm:uncond.an.frs} and  \Cref{conj:2}  by 
the assumption $\chara(F) >\frac n2$.

\subsection{Structure of the paper}

In \S \ref{sec:Con-Nota} we fix some conventions and recall some standard facts on forms and measures.

In \S \ref{sec:OIcusp} we formulate the main result of \cite{AGKS2_1} that bounds the character of a cuspidal representation in terms of orbital integrals of the absolute value of its matrix coefficient.
This establishes our version of substep \eqref{it:intro.1}\eqref{it:intro.a} from the outline in \S\ref{subsub:HCpf}.

In \S \ref{sec:bnd} we begin our study of orbital integrals in the language of algebraic geometry. For this we construct an auxiliary function $\kappa:G^{rss}\to \R$ that allows us to describe the orbital integrals in terms of the pull of the push w.r.t. the Chevalley map $p^{rss}:\bfG^{rss}\to \bfC^{rss}$. See \Cref{thm:bnd.cusp.char} for an exact formulation. This established our version of substep \eqref{it:intro.1}\eqref{it:intro.b}   from the outline \RamiC{in} \S\ref{subsub:HCpf}. Roughly speaking $\kappa$ introduces an arithmetic correction to the more traditional factor $|\Delta|^{-\frac 12}$. 

In \S \ref{sec:fac}
we provide the proof of some standard facts regarding quotients of algebraic varieties by finite groups. Some of these are slightly less standard in positive characteristic.

In \S \ref{sec:basic} we introduce and study a few algebraic varieties that are related to 
$\bfG$. These varieties and properties of certain maps between them (such as flatness, irreducible fibers and reduced fibers) will play a key role in our arguments in the next sections. The reader is advised to consider the diagram below \Cref{lem.GG.nice} when reading this section. In \S\ref{sec:Y} we prove that $\bfY$ is geometrically integrable (\Cref{thm:Y.geo.int}). This bridges between \Cref{thm:bnd2}  and \Cref{thm:bnd1}. 

In \S\ref{sec:kap.form} 
we obtain a  geometric formula for $\kappa$ that relates it to a form $\omega_\bfX$ on the variety $\bfX$.

In \S\ref{Sec:regularity} we prove that  $\omega_\bfX$ is regular (on the smooth locus of $\bfX$). This makes the formula in \S\ref{sec:kap.form} useful.

In \S\ref{sec:om0} we construct a regular invertible form $\omega^0_\RamiC{\bfX}$ that can bound $\omega_\RamiC{\bfX}$  in the formula from \S\ref{sec:kap.form}.

In \S\ref{sec:Pfbnd234} we prove \Cref{thm:bnd2} and its versions. \RamiC{This provides an} explicit geometric bound on the character of a cuspidal representation.

In \S \ref{Sec: proof of Theorem E}
we provide a proof of \Cref{thm:bnd1}.

In \S \ref{Sec: proof of Theorem C and D}
we deduce \Cref{thm:main} and \Cref{thm:minor} 
from  \Cref{thm:bnd1} combined with results of our previous paper \cite{AGKS2}.

In \S\ref{sec:AltMain} we provide several alternatives to the condition of existence of a resolution in \Cref{thm:main}.

\Cref{app:IY} \RamiC{by I. Glazer and Y. Hendel provides} a proof of \Cref{thm.IT}.

In \Cref{app:lem} we explain the mistake in \cite{Le4}.

In \Cref{sec:diag} we present several diagrams containing the main objects in the paper. These diagrams can help to follow the arguments in the paper.

\subsection{Acknowledgments}
 We thank Bertrand Lemaire for detailed discussions of  his work.
 
 We would like to thank Dan Abramovich and Michael Temkin for enlightening conversations about resolution of singularities. 
We would also like to thank Nir Avni for many conversations on algebro geometric analysis. 

\RamiC{We thank Itay Glazer and Yotam Hendel for  their useful suggestions.}

During the preparation of this paper, A.A., D.G. and E.S. were partially supported by the ISF grant no. 1781/23. 
D.K. was partially supported by an ERC grant 101142781.

\section{Notations and Preliminaries}\label{sec:Con-Nota}
\subsection{Conventions}\label{ssec:conv}
\begin{enumerate}
    \item By a \tdef{variety} we mean a reduced scheme of finite type over $F$. 
    \item When we consider a fiber product of varieties, we always consider it in the category of schemes. \Rami{We use set-theoretical notations to define subschemes, whenever no ambiguity is possible.} 
    \item We will usually denote algebraic varieties by bold face letters (such as $\bfX$) and the spaces of their $F$-points by the corresponding usual face letters (such as $X:=\bfX(F)$). We use the same conventions when we want to interpret vector spaces as algebraic varieties.
    \item For Gothic letters we use underline instead of boldface.
    \item We will use the same letter to denote a morphism between algebraic varieties and the corresponding map between the sets of their $F$-points.
    \item We will use the symbol \tdef{$\square$} in a middle of a square diagram  in order to indicate that the square is Cartesian. 
    \item We will use numbers in a middle of a square diagram in order to refer to the square by the corresponding number.
    \item By an \tdef{$F$-analytic manifold} we mean an analytic manifold over $F$ in the sense  of \cite{Ser92}.
    \item A \tdef{big open set} of an algebraic variety $\bfZ$  is an open set whose complement is of co-dimension at least 2 (in each component).
    \item \RamiC{When no ambiguity is possible  we will denote the adjoint action simply by $``\cdot"$.}    
    \item For a measure space $(Z,\mu)$ we denote by $\mdef{L^{<\infty}}(Z,\mu):=\bigcap_{p<\infty} L^p(Z,\mu)$.  We also introduce $\mdef{L_{loc}^{<\infty}}(Z,\mu):=\bigcap_{p<\infty} L_{loc}^p(Z,\mu)$. Note that if $Z$ is an $F$-analytic manifold and $\mu$ is a nowhere vanishing smooth measure then the  spaces $L^p_{loc}(Z,\mu)$ and  $L^{<\infty}_{loc}(Z,\mu)$ do not depend on $\mu$, so we will omit $\mu$ from the notation.   
        \item \RamiC{We will use the symbol \tdef{$<$} to denote the (not necessarily proper) containment relation for groups}.

\end{enumerate}

\subsection{Notations}

We denote by:
\begin{enumerate}
\item \tdef{$\omega_\bfT$} - the standard $\bfT$-invariant form on \RamiC{the torus} $\bfT$.
\item For a group (or an algebraic group) $H$ we denote by \tdef[$Z(\cdot)$]{$Z(H)$} the center of $H$.
    \item $\mdef{G^{ad}}:=G/Z(G)$,     
    $\mdef{\bfG^{ad}}:=\bfG/Z(\bfG)$. 
     \RamiC{Note that $G^{ad}{\lneq} \bfG^{ad}(F).$}
        \item $\mdef{\mu_{Z(G)}}$ the Haar measure on $Z(G)$ normalized on the maximal compact subgroup of $Z(G)$.
\item $\mdef{\mu_{G^{ad}}}$ the Haar measure on $G^{ad}$ that corresponds to $\mu_{G}$ and $\mu_{Z(G)}$.
    \item We equip $\bf C$ with a group structure using the identification $\bfC\cong \bG_m\times \bA^{n-1}$.
    \item \tdef{$\ug$} is the Lie algebra of $\bfG$ (considered as an algebraic variety).
    \item $\mdef \fg:=\ug(F)$.
    \item \tdef{$\Delta$} the discriminant considered as a regular function on $\bfG$.
    \item $\mdef{\bfG^{rss}}\subset \bfG$ the non-vanishing locus of $\Delta$. This is the locus of regular-semi-simple  elements.    
    \item $\mdef{\bfT^r}:=\bfG^{rss}\cap \bfT$.
    \item $\mdef{\Delta^{rss}}:=\Delta|_{\bfG^{rss}}$.
    \item $\mdef{G^{rss}}:=\bfG^{rss}(F)$.    
    \item $\mdef{\bfC^{rss}}$ and $\mdef{C^{rss}}$ the images of $\bfG^{rss}$ and $G^{rss}$ in $\bfC$ and $C$.
    \item $\mdef{p^{rss}}:\bfG^{rss}\to \bfC^{rss}$ the restriction of $p\RamiC{:\bfG\to \bfC}$.
    \item  $\mdef{\Delta_C}$  the discriminant considered as a function on $\bfC$.
    \item $\mdef{\uc}:=\ug//\bfG$, $\fc:=\uc(F)$.
    \item We identify $\uc$ with the collection of monic polynomials of degree $n$. Under this identification $\bfC$ is identified with $\{f\in\uc|f(0)\neq 0\}.$
    
    \item 
   Similarly $\bfC^{rss}$ is  identified with the collection of all separable polynomials in $\bfC$.

\end{enumerate}

\subsection{Forms and measures}\label{ssec:forms}
By a measure on a topological space $Z$ we mean a $\sigma$-additive complete measure that is defined on all Borel subsets of $Z$. We will usually assume that it is positive, but in-general we will not assume that it is locally finite.
\begin{definition}
Let $E$ be a line bundle on an algebraic variety $\bfZ$.
\begin{itemize}
    \item A \gendef{rational section} of $E$ is a section defined over an open dense set in $\bfZ$.\index{section}\index{section!rational}
    \item 
A \tdef[section!$\Q$-]{$\Q-$section}\index{$\Q$-!section} of $E$ is a pair $(n,\xi)$ where $n\in \N$ and $\xi\in \Gamma(\bfZ,E^{\otimes n})$ up to the equivalence relation generated by:
$$(n,\xi)\sim(nk,\xi^{\otimes k})$$
\item We define the notion of a \gendef{rational $\Q$-section} correspondingly.
\item We will use the notion of rational sections and rational $\Q$-sections also when $E$ is defined only on an open dense subset of $\bfZ$.
\item In the notions above, if $E$ is the bundle of (relative) top differential forms we will refer to sections of $E$ as (relative) \tdef[form!$\Q$-, rational]{top forms}.
If $E$ is a trivial bundle, we will refer to 
sections of $E$ as \gendef{functions}. If $E$ is a 
trivial bundle and \RamiC{$\bfZ$} is a point, we will refer to 
sections of $E$ as numbers.\index{form}
\index{$\Q$-!function}\index{$\Q$-!number}In particular, we will refer to 
a $\Q$-section of the trivial bundle over a point
as a \gendef{$\Q$-number}.
\item Note that any rational $\Q$-function can be raised in any  rational power. 
\end{itemize}
\end{definition}
\begin{definition}
\RamiC{Let $Z$ be an $F$-analytic manifold.}
\begin{itemize}
\item Denote by \tdef[$D_\bullet$]{$D_Z$} the sheaf of densities on $Z$, i.e. the sheaf whose sections are smooth measures.
\item If $\omega$ is a top form on $Z$ we denote the corresponding measure on $Z$ by $\mdef{\vert\omega\vert}$. If $\omega$ is invertible then this is a section of $D_Z$.
    \item Define the space of generalized functions $\mdef{C^{-\infty}}(Z)$ to be the space of functionals on the space $C_c^\infty(Z,D_Z)$ of smooth compactly supported measures.

\end{itemize}
\end{definition}
\begin{definition}
\RamiC{Let $\bfZ$ be a smooth algebraic variety.}
\begin{itemize}
\item Denote by $\Omega_\bfZ$ the sheaf of top differential forms on $\bfZ$.
\item For 
 \RamiC{a top form  $\omega$ on $\bfZ$  denote the corresponding measure on $Z:=\bfZ(F)$ by $|\omega|$.}
\item Based on the above, for an invertible section $\omega$ of $\Omega_\bfZ^{\otimes k}$ we can define the corresponding section $|\omega|$ of $D_Z^{\otimes k}$. Note that we have a natural positive structure on $D_Z$, and this section is positive with respect to this structure.
\item For an invertible  $\Q$-top form $\omega:=(k,\omega_1)$ we define $\mdef{\vert\omega\vert}:=|\omega_1|^{\frac 1k}$. Here we take the positive $k$-\RamiC{th} root. 
\item If $\omega$ is not invertible, the definition above defines a density on the non-vanishing locus of $\omega$. This section naturally extends to a Radon measure on $Z$ which we denote also by $|\omega|$.
\item If $\omega$ is a rational $\Q$-top-form we get a measure on an open dense set. We can push this measure to $Z$ and get a not-necessarily-Radon measure. However this measure \RamiC{is} still  $\sigma$-finite.  We denote this measure also by $|\omega|$.

\end{itemize}
\end{definition}
\begin{definition}
$ $
\begin{itemize}
\item For a pair of Borel (not-necessarily locally finite) $\sigma$-finite measures $\mu_1,\mu_2$ on the same topological space s.t. $\mu_1$ is absolutely continuous w.r.t. $\mu_2$ we denote by \tdef{$\frac{\mu_1}{\mu_2}$} to be the Radon-Nikodym derivative. We consider it as an almost everywhere defined function. 
\item Given a morphism of $F$-analytic varieties $\gamma:Z_1\to Z_2$, define the sheaf of relative densities $\tdef[$D_\bullet$]{D_\gamma}:=D_{Z_1}\otimes \gamma^*(D_{Z_2})^*$. Here $*$ denotes the internal Hom to the constant sheaf. \item Given a relative $\Q$-top-form on $Z_1$ w.r.t. $\gamma$, we denote the corresponding relative density  by $|\omega|$. If $\omega$ is a rational \Rami{$\Q$}-top form we consider $|\omega|$ as an almost everywhere defined relative density (defined on the regular locus of $\omega$, and smooth over its invertible locus).

\end{itemize}
\end{definition}
\begin{remark}
    Note that if $\gamma:\bfZ_1\to \bfZ_2$ is a generically smooth morphism of algebraic varieties,  $\omega_i$ are rational $\Q$-top forms on $\bfZ_i$ and $f\in C^{\infty}(Z_1)$ then 
$\gamma_*(|\omega_1|f)$ is absolutely continuous w.r.t. $|\omega_2|$. However $\gamma_*(|\omega_1|f)$ is not necessarily a locally finite measure so $\frac{\gamma_*(|\omega_1|f)}{|\omega_2|}$ is not necessarily in $L^1$ (or even generically finite).
    
\end{remark}
\begin{notation}
$ $
\begin{itemize}
    \item 
    For a smooth morphism $\gamma:\bfZ_1\to \bfZ_2$, a top differential form $\omega_{\bfZ_2}$ on $\bfZ_2$, and a relative top differential form $\omega_{\gamma}$ on $\bfZ_1$ with respect to $\gamma$, denote the corresponding top differential form on $\bfZ_1$ by $\omega_{\bfZ_2}\mdef *\omega_{\gamma}$. 

    We use the same notation for rational $\RamiC{\bQ}$-top-forms. Also in this case, we do not have to require that $\bfZ_i$ and $\gamma$ are smooth, instead it is enough to require that $\gamma$ is generically smooth.

    \item
    \RamiC{
    Conversely, if we are given
    (rational $\RamiC{\bQ}$-)top-forms
    $\omega_{\bfZ_1}, \omega_{\bfZ_2}$ 
    there is a unique (rational 
    $\RamiC{\bQ}$-)top-form $\omega_\gamma$  
    such that $\omega_{\bfZ_1
    } =\omega_{\bfZ_2} 
    *\omega_{\gamma}$. We call this 
    form the \tdef{Gelfand-Leray form} 
    w.r.t. the map $\gamma$ and the forms $\omega_{\bfZ_1},\omega_{\bfZ_2}$.
    }
\end{itemize}
\end{notation}

\begin{definition}\label{def:boxtimes}
    Given a Cartesian square of smooth morphism and smooth varieties:
    $$
    \begin{tikzcd}
\bfV \arrow[r] \arrow[d] \arrow[dr, phantom, "\square"] & \bfZ_1 \arrow[d] \\
\bfZ_2  \arrow[r] & \bfZ
\end{tikzcd}
$$
and top-forms $\omega,\omega_i$ on $\bfZ,\bfZ_i$ define a form $\omega_1\boxtimes_{\omega} \omega_2$ on  $\bfV$ in the following way:
\begin{itemize}
    \item Let $\omega_i'$ be a Gelfand-Leray relative form on $\bf Z_i$ w.r.t. the map $\bfZ_i\to \bfZ$.
    \item \RamiC{Let} $\omega_1'\boxtimes_\bfZ \omega_2'$ \RamiC{be} the corresponding relative form on $\bfV$ w.r.t. the map $\gamma:\bfV\to \bfZ$.
    \item \RamiC{Define} $\omega_1\mdef{\boxtimes_{\omega}} \omega_2:=\omega*(\omega_1'\boxtimes_\bfZ \omega_2')$. 
\end{itemize}

    We use the same notation for rational $\RamiC{\bQ}$-top-forms. Also in this case, we do not have to require $\bfZ_i$, $\bfZ$ and $\gamma$ \RamiC{to} be smooth, instead it is enough to require the maps \RamiC{to} be generically smooth.
\end{definition}

\section{Orbital integrals and characters of cuspidal representations}\label{sec:OIcusp}
In this section we formulate the main result of \cite{AGKS2_1} that bounds the character of a cuspidal representation in terms of the orbital integrals of the absolute value of its matrix coefficient.

\begin{notation}\label{not:Omega} Let $x\in G^{rss}$.
\begin{itemize}
    \item  Denote by $\mdef{\mu_{G_x}}$ the Haar measure on the torus $G_x$ normalized such that the measure of the maximal compact subgroup of $G_x$ is 1. 
\item Denote by $\mdef{\mu_{G\cdot x}}$ the $Ad(G)$-invariant measure on the conjugacy class $Ad(G)\cdot x$ 
that corresponds to the measures $\mu_G$ and   $\mu_{G_x}$ under the identification
$Ad(G)\cdot x\cong G/G_x$. 

\item Let $f\in C^{\infty}(G)$ have compact support modulo the center of $G$. Let $\mdef{\Omega(f)}:G^{rss}\to \C$ be the function defined by $\Omega(f)(x)=\int f|_{G\cdot x}\mu_{G\cdot x}$.

\end{itemize}
\end{notation}

\begin{theorem}[{\cite[Theorem A]{AGKS2_1}}]\label{thm:OrbIntBoundChar}
    Let $\rho$ be a cuspidal irreducible representation of $G$. Then there exist:
    \begin{itemize}
        \item a function $m:G \to \mathbb{C}$ \RamiC{with a compact support modulo the center, and}
        \item a polynomial $\mdef{\alp^\rho}\in \RamiC{\mathbb{N}[t]}$
    \end{itemize}
        

such that for every $\eta\in C^{\infty}_c(G)$ we have 
    $$|\langle \chi_{\rho},\eta\cdot \mu_G\rangle|< \langle f\cdot \Omega(|m|),(|\eta|\cdot \mu_G)|_{G^{rss}}\rangle,$$
\RamiC{where} $f \in \Dima{C^{\infty}}(G^{rss})$ is defined by 
    $$f(g)=\alp^\rho(ov_{G^{rss}}(g)).$$

\end{theorem}

\begin{remark}

    A priori, the right hand side of the above inequality can be infinity.
We interpret the statement in that case as void. 
\end{remark}

\section{Expressing the orbital integral through $\kappa$}\label{sec:bnd}
In this section we construct the function $\kappa:G^{rss}\to \R$  and prove:
\begin{theorem}\label{thm:bnd.cusp.char}
    Recall that $p^{rss}:\bfG^{rss}\to \bfC^{rss}$ is the Chevalley map. Let $f\in C^\infty(G)$ be a  function s.t. its support is compact modulo $Z(G)$.
    Then there exists $\gamma\in C^{\infty}(G)$ such that     
    $$\Omega(f)=\kappa \gamma|_{G^{rss}} (p^{rss})^*\left(\frac{p^{rss}_*(f\mu_G|_{G^{rss}})}{\mu_C|_{C^{rss}}}\right)$$
\end{theorem}
\DimaA{Explicitly, $\gamma(x)=\frac{|\omega_\bfG|}{\mu_G}\frac{\mu_C}{|\omega_\bfC|} |\det(x)^{n-1}|$.}
\subsection{Construction of $\kappa$}\label{sec.kappa}

Let us start with an informal description of the construction.
We first define a canonical $\Q$-top form on any torus, see \Cref{def:wtor} below.
For $x\in G^{rss}$ we define $\Dima{\kappa^0}(x)$ to be the volume of the maximal compact subgroup of $G_x$ with respect to this form on $\bfG_x$. We define $\kappa:=\kappa^0/|\Delta|^{\frac{1}{2}}$.

\begin{notation}
Let     $\bfS$ be a torus 
defined over $F$.
        By \cite[Lemma 8.11]{BorelLAG}   the extension of   
        scalars $\bfS_{F^{sep}}$ of $\bfS$ to the separable closure ${F^{sep}}$ of $F$ is a split torus. Choose an isomorphism 
        $$\phi: \bfS_{F^{sep}} \to (\bG_m^n)_{F^{sep}}.$$
        Let         $\omega_{(\bG_m^n)_{F^{sep}}}$ be the standard top form on 
        $(\bG_m^n)_{F^{sep}}$. Let
        $$\omega_{\bfS_{F^{sep}},\phi}:=
        \phi^*(\omega_{(\bG_m^n)_{F^{sep}}}).$$ 
        Denote  by $$\omega_{\bfS_{F^{sep}},\phi}^2$$ its square considered as a 
        section of 
        $\Omega_{\bfS_{F^{sep}}}^{\otimes 2}$.
\end{notation}
\begin{lemma}\label{lem:phi.indp} The section
    $\omega_{\bfS_{F^{sep}},\phi}^2$  does not depend on $\phi$. 
\end{lemma}
\begin{proof}
Let $\phi,\phi': \bfS_{F^{sep}} \to (\bG_m^n)_{F^{sep}}$ be 2 isomorphisms. Then 
        $$ \omega_{\bfS_{F^{sep}},\phi'}=\phi^{*}\mu^{*}(
        \omega_{(\bG_m^n)_{F^{sep}}}
        ), $$
        where $\mu: (\bG_m^n)_{F^{sep}}\to (\bG_m^n)_{F^{sep}}$ is an automorphism. This automorphism corresponds to an element $\beta\in \GL_n(\Z)$. So we have 
        $$\mu^{*}(
        \omega_{(\bG_m^n)_{F^{sep}}}      )=\det(\beta)\omega_{(\bG_m^n)_{F^{sep}}}.$$
        We get 
    $$\omega_{\bfS_{F^{sep}},\phi'}=\det(\Dima{\beta})\omega_{\bfS_{F^{sep}},\phi}$$ and hence 
    $$\omega^2_{\bfS_{F^{sep}},\phi'}=\det(\Dima{\beta})^2\omega^2_{\bfS_{F^{sep}},\phi}=\omega^2_{\bfS_{F^{sep}},\phi}$$
\end{proof}
\begin{remark}
    Note that this notation is compatible with our notation $\omega_\bfT$ in the sense that the top form $\omega_\bfT$ coincides with the form defined here for the case $\bfS=\bfT$ when considered as a $\Q$-top-form. So in the case $\bfS=\bfT$ the expression $\omega_\bfT$ will continue to denote the actual top-form (and not just the $\Q$-top form).
\end{remark}

\begin{defn}\label{def:wtor}
Let $\bfS$ be a torus
defined over $F$.
By the above lemma (\Cref{lem:phi.indp}) $\omega^2_{\bfS_{F^{sep}},\phi}$ does not depend on $\phi$. So we will denote it by 
$\omega^2_{\bfS_{F^{sep}}}$. By Galois descent there exists a unique section $\omega^2_{\bfS}$ of $\Omega^{\otimes 2}_{\bfS_{F}}$ s.t. its extension of scalars to  $F^{sep}$ is $\omega^2_{\bfS_{F^{sep}}}$.
Define $\mdef[\omega_\bullet]{\omega_{\bfS}}:=[(2,\omega^2_{\bfS})]$ considered as a $\Q$-top form on $\bfS$.      
\end{defn}

Let us now define the function
$\kappa:G^{rss}\to \C$:
\begin{notation}\label{def:kappa}
    Let $x\in G^{rss}$ be a regular semi-simple element.
    \begin{enumerate}        
        \item 
        Denote by \tdef{$K_x$} the unique maximal compact subgroup of $G_x$. 
        \item Define $\kappa^0(x)=\int_{K_x}|\omega_{\bfG_x}|$
        \item Recall that $\Delta^{rss}:G^{rss}\to\C$ is the Weyl discriminant. 
        \item Define 
        $$\mdef{\kappa(x)}=\frac{\kappa^0(x)}{\sqrt{|\Delta(x)|}}.$$
    \end{enumerate}
\end{notation}
Note that the definition of $\kappa^0$ implies:
\begin{lemma}\label{lem:kappa0}
    For $x\in G^{rss}$ we have:    $$|\omega_{\bfG_x}|=\kappa^0(x)\mu_{G_x}$$
\end{lemma}
    
\subsection{Proof of \Cref{thm:bnd.cusp.char}}
Let us first describe the idea of the proof. 
For $x\in G^{rss}$ we consider two $G$-invariant measures on $G\cdot x$:
\begin{enumerate}
    \item the Gelfand-Leray measure with respect to the map $p:\bfG \to \bfC$. This is the absolute value of the Gelfand-Leray form that we denote by $\omega^{G-L}_{\bfG\cdot x}$.
    \item The measure $\mu_{G\cdot x}$ defined in \Cref{not:Omega} above.  
\end{enumerate}
We need to show that the ratio between these measures is $\kappa$.
For this we construct a third measure, which is the absolute value of the $\Q$-top-form $\omega_{\bfG\cdot x}$ that comes from the identification $\bfG \cdot x\cong \bfG/\bfG_x$,  the standard form $\omega_\bfG$ on $\bfG$ and the 
 canonical $\Q$-top-form $\omega_{\bfG_x}$ on the torus $\bfG_x$. Thus it remains to compute the ratios $\omega^{G-L}_{\bfG\cdot x}/\omega_{\bfG\cdot x}$ and $|\omega_{\bfG\cdot x}|/\mu_{G(F)\cdot x}$. The computation of the first ratio is an algebraic problem which is not sensitive to a field extension. Thus we can assume that $x \in T$, in which case the computation is straightforward. This part is responsible for the $\RamiC{|\Delta|}^{-\frac 12}$ factor.  The computation of the second ratio follows from \Cref{lem:kappa0}. This part is responsible for the $\RamiC{\kappa^0}$ factor.

For the proof, we will need some preparation.
\begin{notation}
    Denote by:
    \begin{itemize}
        \item $\ft$ the Lie algebra of $\bfT$,
        \item $\fg^{\neq 0}:=[\ft,\fg]$,
        \item $\mdef{\omega_\bfG}$ the standard $\bfG$-invariant (both from the left and from the right) top form on $\bfG$,
        \item  $\mdef{\omega_\bfC}$ the $\bfC$-invariant top form on $\bfC$ corresponding to the standard top form on $\bG_m \times \bA^{n-1} $ under the identification $\bfC\cong \bG_m \times \bA^{n-1}$.
    \end{itemize}
\end{notation}
The following lemma is standard.
\begin{lemma}\label{lem:lin.alg}
    Let $x\in T\cap G^{rss}$. 
    \begin{enumerate}
        \item \label{lem:lin.alg:1}
    Let $c_x:\fg^{\neq 0}\to \fg^{\neq 0}$ be defined by $c_x(y)=[x,y]$. Then $$\det(c_x)=\Delta(x)$$
    \item \label{lem:lin.alg:2} Let $I:\fc\to \ft$  be the isomorphism given by the identification $$\fc\cong F^n\cong \ft.$$ Then $$\det(I \circ d_x p|_{\ft})^2=\Delta(x).$$
   Here we identify $T_x\bfT\cong \ft$ and $T_{p(x)}\bfC\cong \fc$ using the group structures on $\bfT$ and $\bfC$.    
    \end{enumerate}
\end{lemma}
\begin{notation}
    Let $x\in G^{rss}$. Denote by \begin{itemize}
    \item 
 $\omega^{G-L}_{\bfG\cdot x}$  the Gelfand-Leray form on $\bfG\cdot x=p^{-1}(p(x))$ w.r.t. the map $p:{\bf G} \to \bf{C}$ and the forms $\omega_\bfG$ and $\omega_\bfC$. Consider it as a $\Q$-top-form. 
 \item $\omega_{\bfG/\bfG_x}$ the $\Q$-top-form on $\bfG/\bfG_x$ corresponding to the  $\Q$-top-forms $\omega_\bfG$ and $\omega_{\bfG_x}$. 
 \item 
    $\omega_{\bfG\cdot x}$ be a $\Q$-top-form on $\bfG\cdot x$ corresponding to $\omega_{\bfG/\bfG_x}$ under the identification $\bfG/\bfG_x\cong\bfG\cdot x$
\end{itemize}
\end{notation}
\Cref{lem:lin.alg}
gives us:
\begin{cor}\label{cor:om.or.om.g-l}
    Let $x\in G^{rss}$. Then 
    $\omega_{\bfG\cdot x}={\Delta^{-\frac 12}}(x) \omega^{G-L}_{\bfG\cdot x}\det(x)^{n-1}$. Here, ${\Delta^{-\frac 12}}(x)$ is considered as a $\Q$-number, and thus can multiply $\Q$-forms.
\end{cor}
\begin{proof}
    Note that validity of the statement for a given $x$ does not change when we extend the field $F$. Therefore we can assume without loss of generality that $x$ is diagonalizable. Also the validity of the statement  for a given $x$ does not change when we conjugate $x$.  Therefore we can assume WLOG that $x\in T\cap G^{rss}$. In this case $\bfG_x=\bfT$. We have a canonical top-form \Dima{on $\bfT$} that represents the $\Q$-top form $\omega_\bfT$. We will denote it also by  $\omega_\bfT$. 

    Since both of the forms in the desired equality are $\bfG$ invariant, it is enough to verify their equality at the point $x$.
    Using the left action of $\bfG$ we can identify 
    \begin{equation}\label{eq:g.tan}
     T_x(\bfG)\cong \fg   
    \end{equation}
    Under this identification we get     \begin{equation}\label{eq:orb.tan}
     T_x(\bfG\cdot x)\cong \Im(\Id_{\fg}-ad_{x^{-1}})=\g^{\neq 0}.
    \end{equation}
    
    Set $\omega_\g:=\omega_\bfG|_{x}$ considered as a form on $T_x(\bfG)\cong \fg$.
    (note that it does not depend on $x$ since $\omega_\bfG$ is $\bfG$-invariant). 
    Set also     $\omega_\ft:=\omega_\bfT|_{1}$ considered as a form on $\ft$. 
    \Dima{Now we would like to compute  $\omega^{G-L}_{\bfG\cdot x}|_x$
    under the identification \eqref{eq:orb.tan}.
    Consider the following exact sequences.}
\[
\begin{tikzcd}[column sep=large, row sep=large]
0 \arrow[r]
  & \ker d_x p \arrow[r]
    \arrow[d, phantom, "\rotatebox{-90}{$\overset{\rotatebox{90}{\eqref{eq:orb.tan}}}{\cong}$}"{xshift=1em}]
  & T_x\bfG \arrow[r,"d_xp"]
    \arrow[d, phantom, "\rotatebox{-90}{$\overset{\rotatebox{90}{\eqref{eq:g.tan}}}{\cong}$}"{xshift=1em}]
  & T_{p(x)}\bfC \arrow[r] \arrow[d,"I"]
  & 0 \\
0 \arrow[r]
  & \fg^{\ne 0} \arrow[r]
  & \fg \arrow[r,"I \circ d_x p "]
  & \ft \arrow[r]
  & 0
\end{tikzcd}
\]
    \Dima{Here, $I$ is the identification from \Cref{lem:lin.alg}\eqref{lem:lin.alg:2}.
    Let $\omega_{\g^{\neq 0}}$ be a form s.t. $\omega_{\g^{\neq 0}}\boxtimes \omega_\ft=\omega_\g$. From the exact sequences we obtain
     $\omega^{G-L}_{\bfG\cdot x}|_x=\det(I \circ d_x p|_{\ft})^{-1}\omega_{\g^{\neq 0}}$.
     By \Cref{lem:lin.alg}\eqref{lem:lin.alg:2} we have
     $\det(I \circ d_x p|_{\ft})^{-1} \omega_{\g^{\neq 0}}=\Delta^{-\frac{1}{2}}(x)\omega_{\g^{\neq 0}}$, and hence $\omega^{G-L}_{\bfG\cdot x}|_x=\Delta^{-\frac{1}{2}}(x) \omega_{\g^{\neq 0}}$.}

    To calculate $\omega_{\bfG\cdot x}$,
    note that the Lie algebra of $\bfG_x$ is 
    $\ft$. So we can 
    identify 
    $T_1(\bfG/\bfG_x)$ 
    with $\g^{\neq 0}$ 
    (where $1\in 
    \bfG/\bfG_x$  denotes 
    the class of 
    identity). Under this 
    identification we have     $\omega_{\bfG/\bfG_x}|_1=\omega_{\g^{\neq 0}}$. 
    \Dima{Let $i:\bfG/\bfG_x\cong\bfG\cdot x$ denote the standard identification.
    We have the following commutative diagram:    }

\[
\begin{tikzcd}
T_1(\bfG /\bfG_x) \arrow[r, "d_1i"] & T_x(\bfG\cdot x) \arrow[d,hook,] \\
T_1\bfG \arrow[d, phantom, "\rotatebox{-90}{$\cong$}"] \arrow[u, two heads] & T_x\bfG \arrow[d, phantom, "\rotatebox{-90}{$\overset{\rotatebox{90}{\eqref{eq:g.tan}}}{\cong}$}"{xshift=1em}]\\
\fg  \arrow[r,"-\Id+ad_{x^{-1}}"] & \fg
\end{tikzcd}
\]
\Dima{Thus, under  the identification 
    \eqref{eq:orb.tan}, we have  $\omega_{\bfG\cdot x}|_x=\det((-\Id+ad_{x^{-1}})|_{\fg^{\neq 0}})^{-1}\omega_{\g^{\neq 0}}$.
    Let $c_x$ be as in \Cref{lem:lin.alg}\eqref{lem:lin.alg:1}.
    We have  $$\det(-\Id_{\fg}+ad_{x^{-1}})^{-1}\omega_{\g^{\neq 0}}=\det(x)^{n-1}\det(c_x)^{-1}\omega_{\g^{\neq 0}}.$$ By \Cref{lem:lin.alg}\eqref{lem:lin.alg:1} we have $\det(x)^{n-1}\det(c_x)^{-1}\omega_{\g^{\neq 0}}=\det(x)^{n-1}\Delta(x)^{-1}\omega_{\g^{\neq 0}}.$   Altogether, we have }    
$$\omega_{\bfG\cdot x}|_x=\Delta^{-\frac 12}(x)\det(x)^{n-1}\omega^{G-L}_{\bfG\cdot x}|_x$$ as required.
\end{proof}
\Cref{lem:kappa0} gives us:
\begin{cor}\label{cor:orb.mu.om} For $x\in G^{rss}$ we have:
    $$\mu_{G\cdot x}=\kappa^0(x)|\omega_{\bfG\cdot x}|$$
\end{cor}

    
\begin{proof}[Proof of \Cref{thm:bnd.cusp.char}]
Let $y\in C^{rss}$.
By the definition of the Gelfand-Leray form we have 
\begin{equation}    \label{thm:bnd.cusp.char:g-l}
\int (f|_{p^{-1}}(y))|\omega^{G-L}_{p^{-1}(y)}|=
\left(\frac{p^{rss}_*((f|\omega_\bfG|)|_{G^{rss}})}{(|\omega_\bfC|)|_{C^{rss}}}\right)(y)
\end{equation}
Note that $p:G^{rss}\to C^{rss}$ is onto. 
Let $x\in G^{rss}$ s.t. $p(x)=y$. 
Set $\gamma(x):=\frac{|\omega_\bfG|}{\mu_G}\frac{\mu_C}{|\omega_\bfC|} |\det(x)^{n-1}|$.
We have 
\begin{align*}  
\Omega(f)(x)&=\int (f|_{G\cdot x}) \mu_{G\cdot x}\overset{\text{Cor \ref{cor:orb.mu.om}}}{=}\int (f|_{G\cdot x})\kappa^0(x)|\omega_{\bfG\cdot x}|\overset{\text{Cor \ref{cor:om.or.om.g-l}}}{=}\\&=
\int (f|_{G\cdot x})\kappa^0(x)|\det(x)^{n-1}{\Delta^{-\frac 12}}(x) \omega^{G-L}_{\bfG\cdot x}|=
\\&=
\kappa(x)|\det(x)^{n-1}|
\int (f|_{G\cdot x})|\omega^{G-L}_{\bfG\cdot x}|
\overset{\text{\eqref{thm:bnd.cusp.char:g-l}}}{=}
\kappa(x)|\det(x)^{n-1}| \left(\frac{p^{rss}_*((f|\omega_\bfG|)|_{G^{rss}})}{(|\omega_\bfC|)|_{C^{rss}}}\right)(p(x))
\\&=
\kappa(x)|\det(x)^{n-1}|\frac{|\omega_\bfG|}{\mu_G}\frac{\mu_C}{|\omega_\bfC|} \left(\frac{p^{rss}_*(f\mu_G|_{G^{rss}})}{\mu_C|_{C^{rss}}}\right)(p(x))
\\&=
\kappa(x)\gamma(x)\left(\frac{p^{rss}_*(f\mu_G|_{G^{rss}})}{\mu_C|_{C^{rss}}}\right)(p(x))
\end{align*}
as required.
\end{proof}

\section{Factorizable actions}\label{sec:fac}
In this section we give some standard facts about the quotient of an algebraic variety by a finite group which are slightly less standard in positive characteristic.

\begin{definition}
    Let a finite group $\Gamma$ act on a variety $\bfZ$. We say that this action is \tdef[factorizable action]{factorizable} if the categorical quotient $\bfZ//\Gamma$ exists (as a variety), and the map $\bfZ \to \bfZ//\Gamma$ is finite. 
\end{definition}

\begin{prop}[{See e.g. \cite[\RamiC{Corollary 3.1.8}]{AGKS2_2} 
}]\label{lem:fac}
    Let a finite group $\Gamma$ act on a quasi-projective variety $\bfZ$. Then the action is factorizable. 
\end{prop}

\begin{lemma}[{See e.g. \cite[\RamiC{Corollary 3.1.5}]{AGKS2_2}}]\label{lem:GamU}
    Let a finite group $\Gamma$ act factorizably on a variety $\bfZ$. 
    Let $\bfU\sub \bfZ$ be an open $\Gamma$-invariant set. 
    Then the action of $\Gamma$ on $\bfU$ is factorizable and the following diagram is a Cartesian square.
        \begin{equation}\label{eq:csg}
        \xymatrix{
        \bfU \ar[d] \ar[r] &\bfZ\ar[d]\\
        \bfU//\Gamma \ar[r] &\bfZ//\Gamma}
        \end{equation}
        Moreover, the bottom arrow is an open embedding. 
\end{lemma}

\begin{lemma}[{See e.g. \cite[Lemma 3.2.3]{AGKS2_2}}]\label{lem:Gam}
    Let a finite group $\Gamma$ act  factorizably on a 
    variety 
    $\bfZ$. \RamiC{Assume that the action is free} (i.e. the action of $\Gamma$ on $\bfZ(\bar F)$ is free). 
    Then
    \begin{enumerate}
        \item \label{it:etal} The map $\bfZ\to \bfZ//\Gamma$ is {\'e}tale. 
        \item The natural morphism $m:\bfZ\times \Gamma \to \bfZ\times_{\bfZ//\Gamma}\bfZ$ is an isomorphism. 
    \end{enumerate}
\end{lemma}

\begin{lem}    
[Galois descent for free actions, {see e.g. \cite[Corollary 3.2.4]{AGKS2_2}}]\label{cor:GalDes}
    In the setting of the previous lemma, let 
      $Sch_{\bfZ//\Gamma}$ denote the category of schemes over ${\bfZ//\Gamma}$ and \RamiC{let} $Sch^{\Gamma}_{\bfZ}$ denote the category of schemes over $\bfZ$ equipped with an action of $\Gamma$ which is compatible with the action of $\Gamma$ on $\bfZ$. Consider the functor 
    $\cF:Sch_{\bfZ//\Gamma}\to Sch^{\Gamma}_{\bfZ}$ defined by $\cF(\bfX)=\bfX \times_\RamiC{\bfZ//\Gamma} \bfZ,$ with $\Gamma$ acting on the second coordinate.  Let $\beta:\cF(\bfX)\to \bfX$ be the projection on the first component.  Then 
    \begin{enumerate}[(i)]
        \item \label{it:ff} $\cF$ is fully faithful. 
        \item \label{it:sh} Given $\bfX \in Sch_{\bfZ//\Gamma}$ and a sheaf 
        $\cV$ on it,
        \RamiC{the pullback $\cV(\bfX)\to (\beta^*\cV)({\cF}(\bfX))$ with respect to
        $\beta$} gives
        an isomorphism 
        $$\cV(\bfX)\cong (\beta^*\cV)(\Dima{\cF}(\bfX))^{\Gamma}.$$
    \end{enumerate}    
\end{lem}

\begin{lemma}\label{lem:pullbackquo}
    Let a finite group $\Gamma$ act on an affine variety $\bfZ$. Let $\gamma:\bfZ_1\to \bfZ//\Gamma$ be a flat morphism of affine varieties. Then the projection on the second coordinate $\bfZ\times_{\bfZ//\Gamma}\bfZ_1\to \bfZ_1$ defines an isomorphism 
    $$(\bfZ\times_{\bfZ//\Gamma}\bfZ_1)//\Gamma\cong \bfZ_1.$$
\end{lemma}
\DimaC{We note that the fiber product in the lemma scheme-theoretical, and we do not claim that in general it is a variety.}
\begin{proof}
    We need to show that the natural map $$\DimaB{\cO}_{\bfZ_1}(\bfZ_1)^\Gamma\to\left(\DimaB{\cO}_{\bfZ_1}(\bfZ_1) \otimes_{\DimaB{\cO}_{\bfZ}(\bfZ)^\Gamma} \DimaB{\cO}_{\bfZ}(\bfZ)\right)^\Gamma$$ is an isomorphism.
    Equivalently it is enough to show that the natural map $$\DimaB{\cO}_{\bfZ_1}(\bfZ_1) \otimes_{\DimaB{\cO}_{\bfZ}(\bfZ)^\Gamma} \DimaB{\cO}_{\bfZ}\Dima{(\bfZ)}^\Gamma \to\left(\DimaB{\cO}_{\bfZ_1}(\bfZ_1) \otimes_{\DimaB{\cO}_{\bfZ}(\bfZ)^\Gamma} \DimaB{\cO}_{\bfZ}(\bfZ)\right)^\Gamma$$ is an isomorphism.
    This follows from  the  assumption that 
    $\DimaB{\cO}_{\bfZ_1}(\bfZ_1)$ is flat over $\DimaB{\cO}_{\bfZ}(\bfZ)^\Gamma$ and thus the functor $$M\mapsto \DimaB{\cO}_{\bfZ_1}(\bfZ_1) \otimes_{\DimaB{\cO}_{\bfZ}(\bfZ)^\Gamma} M$$ commutes with finite limits.
\end{proof}

The following lemma follows immediately from miracle flatness (see \cite[\href{https://stacks.math.columbia.edu/tag/00R4}{Lemma 00R4}]{SP}).
\begin{lemma}\label{lem:flatFactor}
    Let a finite group $\Gamma$ act factorizably on a smooth variety $\bfX$. Suppose that $\bfX//\Gamma$ is smooth. Then the factor map $\bfX \to \bfX//\Gamma$ is flat.
\end{lemma}

\section{Some geometric objects related to $\bfG$}\label{sec:basic}
In this section we introduce certain algebraic varieties related to $\bfG$. The diagram it \S\ref{ssec:geo.sum} summarizes most of them. We also prove \Cref{thm:Y.geo.int}.

\subsection{The maps $p$ and $q$}
\begin{notation}
Identify $\bfC \cong \bfT//W$ and $\uc\cong \ut//W$.
    Denote by $\mdef{q}:\bfT\to \bfC$ and $\mdef{q_0}:\ut\to \uc$ the quotient maps.  Denote by $\mdef{p_0}:\ug\to \uc$ the Lie algebra version of the Chevalley map $p\DimaC{:\bfG\to \bfC}$.
\end{notation}

\Cref{lem:flatFactor} and \Cref{lem:fac} imply

\begin{lemma}\label{lem:q.flat}
    The maps $q_0,q$ are flat.
\end{lemma}

\begin{lem}\label{cor:pqflat}
    The maps $p_0,p$ are flat.
\end{lem}
\begin{proof}
    See \cite[Corollary 5.0.5]{AGKS2} for   the flatness of $p_0$. This implies the flatness of $p$. 
\end{proof}

\begin{lem}\label{lem.fib.p.irr}
The fibers of $p$ (and of $p_0$) are absolutely irreducible. 
\end{lem}
\begin{proof}
    This follows from the Jordan decomposition.     
\end{proof}

\begin{notation}
Denote by $\bfG^r$ the smooth locus of $p$.
\end{notation}

\begin{lem}[cf. {\cite[Lemma 5.0.8]{AGKS2}}]\label{lem:Gr.onto}
    $p|_{\bfG^r}:\bfG^r\to \bfC$ is onto.
\end{lem}
\begin{proof}
    This follows from the notion of companion matrix.
\end{proof}

\begin{cor}\label{cor:pFibReduced}
    $\bfG^r$ is big in $\bfG$, and the fibers of $p$ are absolutely reduced. Additionally the same holds for the fibers of $p_0$.   
\end{cor}

\begin{proof}
The analogous statements for $\ug^r$ and $p_0$ are proven in \cite[Corollary 5.0.9]{AGKS2}. The statement for $\bfG^r$ and $p$ follows from that.
\end{proof}

\subsection{The varieties $\bfX$,  $\bfY$, $\mathbf \Upsilon$}

\begin{notation}
    Denote $\mdef{\bfG'}:=\bfG\times_\bfC \bfT$.   Denote by $\mdef{\psi}:\bfG'\to \bfG$ the projection on the first coordinate and by $\mdef{\varphi}:\bfG'\to \Dima{\bfT}$ the projection on the second coordinate.
\end{notation}

\begin{lem}\label{lem:flatIrr}
Let $\bfZ_2$ be an irreducible variety.
    Let $\gamma:\bfZ_1\to \bfZ_2$ be a flat map of finite type of schemes. Assume that the fibers of $\gamma$ are irreducible. Then $\bfZ_1$ is irreducible.
\end{lem}
\begin{proof}
We have to show that every two non-empty open subsets $\bfU_1,\bfU_2\subset \bfZ_1$ intersect. 
By \cite[\href{https://stacks.math.columbia.edu/tag/01UA}{Lemma 01UA}]{SP},  $\gamma$ is an open map. Thus $\gamma(\bfU_i)$ are open, and since $\bfZ_2$ is irreducible they intersect. Let $p\in \gamma(\bfU_1)\cap \gamma(\bfU_2)$. Then $\bfU_i\cap \gamma^{-1}(p)$ are non-empty open subsets of the fiber $\gamma^{-1}(p)$. Since the fiber is irreducible, they have to intersect. Thus $\bfU_1$ and $\bfU_2$ intersect.
\end{proof}

\begin{lem}\label{lem:Gp}
    $\bfG'$ is absolutely reduced, locally complete intersection, and irreducible.
\end{lem}
\begin{proof}
By \Cref{cor:pqflat}, $\bfG'$ is a locally complete intersection, and the maps $p\circ \psi:\bfG'\to \bfC$ and $\varphi:\bfG'\to \bfT$ are flat. By \Cref{lem.fib.p.irr}, the fibers of $\varphi$ are absolutely irreducible. 
Therefore, by \Cref{lem:flatIrr}, $\bfG'$ is absolutely irreducible.
Thus it is enough to show that $\bfG'$ is generically absolutely reduced. Since $p\circ \psi:\bfG'\to \bfC$ is flat, $(p\circ \psi)^{-1}(\bfC^{rss})$ is dense in $\bfG'$ (since the preimage of a dense subset under a flat morphism is dense, see 
\cite[\href{https://stacks.math.columbia.edu/tag/01UA}{Lemma 01UA}]{SP}). 
Thus it is enough to show that $(p\circ \psi)^{-1}(\bfC^{rss})$ is absolutely reduced. Note that  
$(p\circ \psi)^{-1}(\bfC^{rss})\cong \bfT^{r}\times_{\bfC^{rss}}\bfG^{rss}$. The assertion follows now from the statement that the natural map $\bfT^{r}\to \bfC^{rss}$ is {\'e}tale. This in turn follows from \Cref{lem:Gam}\eqref{it:etal}. 
\end{proof}

\begin{notation}$\,$
\begin{itemize}
    \item $\mdef{\bfY}:=(\bfT\times \bfT)//W$, where $W$ acts diagonally. Let $\mdef{\mu}:\bfT\times \bfT\to \bfY$ denote the quotient map. Note that the quotient exists by \Cref{lem:fac}.
    \item Let $\mdef{\pi}:\bfY\to \bfC$ denote the map induced by the projection on the first coordinate $\bfT\times \bfT\to \bfT$.
    Let $\mdef{\alp}:\bfY \to \bfC\times \bfC$ denote the natural map.
    \item $\mdef{\bfX}:=\bfG\times_{\bfC}\bfY$, and let $\mdef{\tau}:\bfX\to \bfG$ and $\mdef{\sigma}:\bfX\to \bfY$ be the projections.
    \item $\mdef{\widetilde{\bfX}}:=\bfG'\times \bfT$ and let $\mdef{\nu}:\widetilde{\bfX}\to \bfX$ be the natural map given by the identification $\widetilde{\bfX}\cong \bfG\times_{\bfC}\bfT\times \bfT$ and the quotient map $\bfT\times \bfT\to \bfY$.
\end{itemize}
\end{notation}

From \Cref{lem:Gp} we obtain
\begin{cor}
    $\tilde \bfX$ is absolutely reduced, absolutely irreducible and \DimaC{locally} complete intersection.
\end{cor}

\begin{lem}
    The  map $\nu:\widetilde{\bfX} \to \bfX$ \DimaC{induces} an isomorphism $\widetilde{\bfX}//W\cong \bfX$.
\end{lem}
\begin{proof}
By \Cref{cor:pqflat}, $p:\bfG\to \bfC$ is flat. Since the base change of \DimaC{a flat map} is flat, the projection $\sigma:\bfX\to \bfY$ is flat.  Thus, by \Cref{lem:pullbackquo}, the natural map 
$\bfX \times_{\bfY}(\bfT\times\bfT)\to \bfX$ gives an isomorphism 
$$(\bfX \times_{\bfY}(\bfT\times\bfT))//W\cong \bfX$$
The assertion follows now from the fact that $$\bfX \times_{\bfY}(\bfT\times\bfT)\cong \bfG\times_{\bfC}\bfY\times_{\bfY}(\bfT\times\bfT)\cong \widetilde{\bfX}.$$
\end{proof}

\begin{cor}\label{cor:X.irr}
    $ \bfX$ is absolutely reduced and absolutely irreducible.
\end{cor}

\begin{notation}
    Denote by $\mdef{(\bfT\times \bfT)^f}$ the free locus of the action of $W$. Denote by $\mdef{\mu}:\bfT\times \bfT$ the quotient map. Denote $\mdef{\bfY^f}:=\mu((\bfT\times \bfT)^f)$. 
\end{notation}
The following lemma is standard.
\begin{lem}
    $(\bfT\times \bfT)^f$ is a big open set in $\bfT\times \bfT$.
\end{lem}

\begin{cor}\label{cor:YfBig}
    $\bfY^f$ is big in $\bfY$.
\end{cor}

Lemmas \ref{lem:GamU} and  \ref{lem:Gam} imply
\begin{cor}\label{cor:YfSmooth}
$ $
    \begin{enumerate}[(i)]
        \item $\bfY^f$ is smooth.
        \item\label{cor:YfSmooth:2} $\mu|_{(\bfT\times \bfT)^f}$ is  smooth.
        \item $q$ is generically smooth.
    \end{enumerate}
    
\end{cor}
\begin{proof}
    By \Cref{lem:GamU}, $(\bfT\times\bfT)^f//W\cong \bfY^f$.
    By \Cref{lem:Gam}, the quotient map 
    $(\bfT\times\bfT)^f\to (\bfT\times\bfT)^f//W$ is {\'e}tale. Since it is also finite, and  $(\bfT\times\bfT)^f$ is smooth, this implies that $\bfY^f$ is smooth.
\end{proof}

\begin{lem}\label{lem:BigLoc}
    Let $\gamma:\bfZ_2\to \bfZ_1$ be a flat morphism of algebraic varieties. Assume that the fibers of $\gamma$ are reduced and $\gamma$ is smooth over an open dense subset of $Z_1$. Assume that $Z_1$ has a big smooth locus. Then $Z_2$ has a big smooth locus.    
\end{lem}
\begin{proof}
    Without loss of generality we can assume that $\bfZ_1$ is smooth. Let $\bfU\subset \bfZ_1$ be an open dense subset such that $\gamma$ is smooth over $\bfU$. Let $\bfZ_3$ be the complement of $\bfU$. It is enough to show that $\gamma$ is smooth in every generic point of $\bfZ_3$. This follows from the fact that $\gamma$ is flat and its fibers are generically smooth (since they are reduced).
\end{proof}

\begin{cor}\label{cor:Xbig}
    $\bfX$ has a big smooth locus.
\end{cor}
\begin{proof}
    By \Cref{lem:BigLoc} and Corollaries \ref{cor:YfSmooth} and \ref{cor:YfBig}, it is enough to show that:
    \begin{enumerate}[(i)]
        \item \label{it:sigflat} The map $\sigma:\bfX\to \bfY$ is flat.
        \item \label{it:sigred} The fibers of $\sigma$ are reduced.
        \item \label{it:SigOpSm}There exists an open dense subset of $\bfY$ such that $\sigma$ is smooth over it.
    \end{enumerate}
    Note that $\sigma$ is a base change of $p:\bfG\to \bfC$. Note also that
    $p$ is flat by \Cref{cor:pqflat}, and its fibers are reduced by \Cref{cor:pFibReduced}. Thus (\ref{it:sigflat}) and \eqref{it:sigred} hold.

Now, $p$ is smooth over the open dense subset $\bfC^{rss},$ 
and $\pi$ is locally dominant (since $\bfY$ is irreducible and $\pi$ is dominant).  This implies \eqref{it:SigOpSm}.
\end{proof}

\begin{lem}\label{lem.up.nice}
$\,$
\begin{enumerate}[(i)]
    \item \label{it:UpIrr}${\bf \Upsilon}$ is reduced and irreducible.
    \item\label{it:Upbig} The regular locus of ${\bf \Upsilon}$ is big in ${\bf \Upsilon}$.
    \end{enumerate}
\end{lem}
\begin{proof}
$ $
\begin{itemize}
    \item[\eqref{it:UpIrr}.] 
    Consider the Chevalley map $p:\bfG\to \bfC$. It is flat and its fibers are reduced and irreducible. Therefore, so is the natural map $p':{\bf \Upsilon}\to \bfX$. By \Cref{lem:flatIrr}, this implies the assertion.
    \item[\eqref{it:Upbig}.] 
    By \Cref{cor:Xbig}, the regular locus of $\bfX$ is big. By \Cref{cor:pFibReduced}
    the fibers of $p$ are reduced. It is well known that    the regular loci of the (reduction of the) fibers of $p$  are big (in these fibers). So such are also the regular loci of the fibers of $p'$. This implies the assertion. 
\end{itemize}
\end{proof}
\begin{lemma}\label{lem.GG.nice}
$ $
\begin{enumerate}[(i)]
    \item \label{it:GGIrr}$\bfG\times_\bfC \bfG$ is reduced and irreducible.
    \item The regular locus of $\bfG\times_\bfC \bfG$ is big in $\bfG\times_\bfC \bfG$.
    \end{enumerate}
\end{lemma}
\begin{proof}
    The proof is similar to the proof of \Cref{lem.up.nice}.
\end{proof}
\subsection{Summary}\label{ssec:geo.sum}
The following diagram summarizes the main objects discussed in this section.
$$
\begin{tikzcd}
    \bfG'\times \bfT\arrow[swap,"\mathfrak{pr}_{\bfT}"]{ddd}\arrow[swap,bend right=30, "\mathfrak{pr}_{\bfG'}"]{dddrr}\arrow[phantom, "="]{rr}
    &
    &\tilde \bfX \drar[phantom, "\square"]\rar[""]\dar["\nu"] & \bfT\times \bfT\dar["\mu"]\drar["q\times q"] 
    \arrow[rounded corners, to path={ -- ([xshift=15ex]\tikztostart.east) -- 
    node[right]{$\mathfrak{pr}_1$}
    ([xshift=17ex]\tikztotarget.east) -- (\tikztotarget)}]{ddd}&  &
    \\      
  &{\bf \Upsilon}\rar["p'"] \dar[] \drar[phantom, "\square"] \drar[bend left=15, "\zeta"{yshift=-4}]
  &\bfX \drar[phantom, "\square"]\rar["\sigma"]\dar["\tau"] & \bfY\rar["\alpha"]\dar["\pi"] & \bfC\times \bfC  \dlar["\fp\fr_\bfC^1"] & 
  \\  
&\bfG\times_\bfC\bfG\rar[swap, "\fp\fr_\bfG^2"]&\bfG\drar[phantom, "\square"]\rar[swap, "p"] & \bfC &    & 
\\\bfT 
&&\bfG'\uar[swap,"\psi"]\rar[swap, "\varphi"] & \bfT\uar["q"] &   & 
\end{tikzcd}
$$
In this diagram
\begin{itemize}
    \item 
$\mdef{\mathfrak{pr}_{\bfG'}}$,  $\mdef{\mathfrak{pr}_1}$ and $\mdef{\mathfrak{pr}_{\bfC}^{\RamiA{1}}}$ are the projections on the first coordinate.
\item $\mdef{\mathfrak{pr}_{\bfT}}$, $\mdef{\fp\fr_{\bfG}^2}$ and $\mdef{p'}$ are  the projections on the second coordinate.
\end{itemize}

\begin{lem}\label{lem:gen.sm}
        The maps in the above diagram are generically smooth.
\end{lem}
\begin{proof}
    $\mu$ and $q$ are generically smooth by  \Cref{cor:YfSmooth}. $p$ is generically smooth by \Cref{cor:pFibReduced}. This implies that $q\times q$ is generically smooth and hence so is $\pi\circ \mu$. Therefore (in view of the irreducibility of $\bfY$) $\pi$ is generically smooth.
    The rest of the statements are either obvious or obviously follow from the above.    
\end{proof}

\subsection{Integrability of $\bfY$ -- Proof of \Cref{thm:Y.geo.int}}\label{sec:Y}

We now deduce \Cref{thm:Y.geo.int}, which states that $\bfY$ is geometrically integrable, from the results of \cite{AGKS2_2}. 
For this we introduce the following notation.
\begin{notn}
Let $\ut$ be the Lie algebra of $\bfT$ and let
    ${\uy}:=\ut\times \ut//W$ where the action of $W$  is diagonal.
\end{notn}
By \Cref{lem:GamU} $\bfY$ can be embedded as an open set in $\uy$. Thus, \Cref{thm:Y.geo.int} follows from the following one.
\begin{prop}\label{thm:int.y}
$\uy$ is geometrically integrable.    
\end{prop}
\begin{proof}
    Note that $\uy\cong (\A^2)^n//S_n$. The assertion follows now from \cite[Corollary C]{AGKS2_2}.
\end{proof}

\section{Algebro-geometric formula for $\kappa$}\label{sec:kap.form}
Recall that $\bfX =\bfG\times_{\bfC}\bfY$ and that $\tau:\bfX \to \bfG$ is the projection on the first factor. In \S\ref{sec.kappa} we introduced a function $\kappa$ on $G^{rss}$.
In this section we construct a clopen $\mathcal{A}\subset X$ and a \RamiA{rational} $\Q$-top-form $\omega_\bfX$ on $\bfX$ and prove 
\begin{theorem}\label{thm:kappa.alg.geo}
We have
  $\tau_*((|\omega_\bfX|)|_{\mathcal{A}})=\kappa |\omega_\bfG|.$    
\end{theorem}

\subsection{Construction of $\omega_\bfX$}
The construction is based on the relative rational $\Q$-top form $\omega_\tau$ on $\bfX$ with respect to the map $\tau:\bfX\to \bfG$. The idea of the construction of $\omega_\tau$ is based on the observation that the generic fibers of $\tau$ admit  natural group structures of tori.
The relative form $\omega_\tau$ is defined in such a way that its restrictions to the generic fibers of $\tau$ are the canonical $\Q$-top forms on the fibers (see \Cref{def:wtor}).
\RamiA{We use $\omega_\tau$ and the standard top form $\omega_{\bfG}$ on $\bfG$ in order to construct a form $\omega'_{\bfX}$ on $\bfX$. Finally we divide the form \DimaC{$\omega'_{\bfX}$} by the square root of the discriminant to obtain} \DimaC{$\omega_{\bfX}$.}

To implement this idea we start with the following notation.

\begin{notation}\DimaC{Recall that $\bfT^r=\bfT \cap \bfG^{rss}$.}
    Denote $\mdef{\bfY^r}:=(\bfT^r\times \bfT)//W$.
\end{notation}
The following lemma is standard:
\begin{lem}\label{lem:cancel}
Consider the commutative diagram of affine algebraic varieties
$$
\begin{tikzcd}
    \bfZ_{11} \arrow[d,"\delta_{1}"] \arrow[r, "\gamma_{11}"]\arrow[dr, phantom, "\square"] & \bfZ_{12} \arrow[d,"\delta_{2}"] \arrow[r,"\gamma_{12}"]  & \bfZ_{13} \arrow[d,"\delta_{3}"] \\
    \bfZ_{21} \arrow[r,"\gamma_{21}"]  & \bfZ_{22} \arrow[r, "\gamma_{22}"]  & \bfZ_{23}
\end{tikzcd}
$$
Assume also that we have:
$$
\begin{tikzcd}
    \bfZ_{11} \arrow[d,"\delta_{1}"] \arrow[r, "\gamma_{12}\circ\gamma_{11}"]\arrow[dr, phantom, "\square"] & \bfZ_{13} \arrow[d,"\delta_{3}"]\\
    \bfZ_{21} \arrow[r,"\gamma_{22}\circ\gamma_{21}"]  & \bfZ_{23} 
\end{tikzcd}
$$
and that the map $\gamma_{\RamiA{2}1}$ is faithfully flat.
Then we have:
$$
\begin{tikzcd}
    \bfZ_{12} \arrow[d,"\delta_{2}"] \arrow[r,"\gamma_{12}"] \arrow[dr, phantom, "\square"] & \bfZ_{13} \arrow[d,"\delta_{3}"] \\
     \bfZ_{22} \arrow[r, "\gamma_{22}"]  & \bfZ_{23}
\end{tikzcd}
$$
\end{lem}
\RamiA{
\begin{proof}
     We want to show that the natural map $$\bfZ_{12}\to \bfZ_{13}\times_{\bfZ_{23}} \bfZ_{22}  $$ is an isomorphism.
     We know that the natural map $$ \bfZ_{12}\times_{\bfZ_{22}} \bfZ_{21}
      \to 
     (\bfZ_{13}\times_{\bfZ_{23}}  \bfZ_{22}) \times_{\bfZ_{22}} \bfZ_{21}$$ is an isomorphism.
The assertion follows now from the fact that $\bfZ_{21}$ is faithfully flat  over $\bfZ_{22}$ using faithfully flat descent for isomorphisms (see {\it e.g.} \cite[\href{https://stacks.math.columbia.edu/tag/02L4}{Lemma 02L4}]{SP}).
\end{proof}
}
\begin{lem}\label{lem: caresian square N1}
    The square 
        \begin{equation}
        \xymatrix{
         \bfT^r\times \bfT  \ar[d]^{\mathfrak{pr}_1^r} \ar[r]^{\mu^r} & \bfY^r\ar[d]^{\pi^r}  \\
        \bfT^r \ar[r]^{q^r}& \bfC^{rss}}
        \end{equation} 
        is Cartesian. 
        Here, $\pi^r,\,q^r,\,\mu^r,$ and $\mathfrak{pr}^r_1$ are restrictions of $\pi,\,q,\, \mu$ and $\mathfrak{pr}_1$ respectively.
\end{lem}
\begin{proof}
Consider the following diagram 
        \begin{equation}
        \xymatrix{
        \RamiA{W \times} \bfT^r\times \bfT \ar[d]^{        
        pr_{\RamiA{\bfT^r\times \bfT}}
                }\ar@/_2.0pc/[dd] \ar[r]^{a_1} & \bfT^r\times \bfT\ar[d]^{\mu^r}\ar@/^2.0pc/[dd]^{q^r\circ \fp\fr_1^r}\\
        \bfT^r\ar[d]^{\mathfrak{pr}_1^r}\times \bfT \ar[r]^{\mu^r} &\bfY^r\ar[d]^{\pi^r}\\
        \bfT^r \ar[r]^{q^r}& \bfC^{rss}
        }
        \end{equation}
\RamiA{where $a_1$ is the diagonal action map, and $pr_{\bfT^r\times \bfT}$ is the projection.}
By Lemmas \ref{lem:Gam} and \ref{lem:GamU}, the squares
        \begin{equation}\label{eq:TTWY}
        \xymatrix{
       \RamiA{W \times} \bfT^r\times \bfT  \ar[d]^{pr_{\RamiA{\bfT^r\times \bfT}}} \ar[r]^{a_1} & \bfT^r\times \bfT\ar[d]^{\mu^r}\\
        \bfT^r\times \bfT \ar[r]^{\mu^r} &\bfY^r}
        \end{equation}
and 
        \begin{equation}
        \xymatrix{
         \RamiA{W \times} \bfT^r  \ar[d]^{pr_\RamiA{\bfT^r}} \ar[r]^{a_2} & \bfT^r\ar[d]^{q^r}  \\
        \bfT^r \ar[r]^{q^r}& \bfC^{rss}}
        \end{equation} 
are Cartesian,  \RamiA{where $a_2$ is the action map and $pr_{\bfT^r}$ is the projection.}

\RamiA{Also the square
        $$
        \xymatrix{
       \RamiA{W \times} \bfT^r\times \bfT  \ar[d]^{pr_{W\times \bfT^r}
       } \ar[r]^{a_1} & \bfT^r\times \bfT\ar[d]^{\fp\fr_1^r}\\
W \times \bfT^r  \ar[r]^{a_2} & \bfT^r}          $$
is Cartesian, where $pr_{W\times \bfT^r}$ is the projection.
}Hence the square 
        \begin{equation}\label{eq:TTWC}
        \xymatrix{
         \RamiA{W \times} \bfT^r\times \bfT  \ar[d]_{\fp\fr_1^r \circ pr_{\RamiA{\bfT^r\times \bfT}}\RamiA{=pr_{\bfT^r}\circ pr_{W\times\bfT^r}}} \ar[r]^{a_1} & \bfT^r\times \bfT\ar[d]^{q^r\circ \fp\fr_1^r=\pi^r\circ \mu^r}  \\
        \bfT^r \ar[r]& \bfC^{rss}}
        \end{equation} 
        is Cartesian.
\RamiA{By \Cref{cor:YfSmooth}\eqref{cor:YfSmooth:2} the map $\mu^r$ is etale. Hence b}y \Cref{lem:cancel}, and from \eqref{eq:TTWY}  and \eqref{eq:TTWC}, we get that the square 
        \begin{equation}
        \xymatrix{
         \bfT^r\times \bfT  \ar[d]^{\mathfrak{pr}_1^r} \ar[r]^{\mu^r} & \bfY^r\ar[d]^{\pi^r}  \\
        \bfT^r \ar[r]^{q^r}& \bfC^{rss}}
        \end{equation} 
        is Cartesian, as required. 
\end{proof}

\begin{definition}\label{defn:om.X}
$ $
    \begin{itemize}
        \item Define a group-scheme structure on the $\bfC^{rss}$-scheme $\bfY^r\to\bfC^{rss}$ in the following way.
        Consider the Cartesian square \RamiA{given by \Cref{lem: caresian square N1}}
        \begin{equation}\label{eq:crt}
        \xymatrix{
         \bfT^r\times \bfT  \ar[d]^{\mathfrak{pr}_1^r} \ar[r]^{\mu^r} & \bfY^r\ar[d]^{\pi^r}  \\
        \bfT^r \ar[r]^{q^r}& \bfC^{rss}}
        \end{equation} 
        
The left column has a natural structure of a group scheme (over $\mathbf{T}^r$). $W$ acts homomorphically w.r.t. this structure. By \Cref{cor:GalDes}\eqref{it:ff}, this gives a group scheme structure on the right column (over $\bfC^{rss}$). 
\item  \RamiA{Recall that} $\omega_\bfT$ \RamiA{is} the standard 
top differential form on $\bfT$.  Let 
$\mdef{\omega_{\mathfrak{pr}^r_1}}$
be the relative top differential form on $\bfT^r\times \bfT$ 
w.r.t to the map $\fp\fr^r_1$ obtained from 
$\omega_\bfT$. Consider it as a $\Q$-top 
differential form. As such it is $W$ 
invariant. Hence by \Cref{cor:GalDes}\eqref{it:sh} it descends to a relative 
$\Q$-top differential form \tdef{$\omega_\pi$} 
on $\bfY^r$ w.r.t. \Dima{$\pi^r$}. Consider it as a 
relative rational  $\Q$-top differential form on 
$\bfY$.


        \item Consider the Cartesian square 
                \begin{equation}\label{eq:crt2}
        \xymatrix{
        \bfX \ar[d]^{\tau} \ar[r]^{\sigma} &\bfY\ar[d]^\pi\\
        \bfG \ar[r]^{p} &\bfC}
        \end{equation}
Denote $\mdef{\omega_{\tau}}:=\sigma^*(\omega_{\pi})$ considered as a relative rational  $\Q$-top differential form on $\bfX$ w.r.t. $\tau$. 


\item 
Let $\RamiA{\mdef{\omega'_\bfX}}:=\omega_{\bfG}*\omega_{\tau},$ considered as a  rational $\Q$-top differential form on $\bfX$. 
Here we use the fact that the morphism $\tau:\bfX \to \bfG$ is generically smooth, as provided by     \Cref{lem:gen.sm}.
\item Denote $\mdef{\omega_\bfX}:=\tau^*(\Delta^{-1/2}) \RamiA{\omega'_\bfX}$.
\end{itemize}
\end{definition}
The definition of $\omega_\tau$ gives us the following:
\begin{lemma}\label{lem:om.atu}
    For any $x\in G^{rss}$ the form $\omega_\tau|_{\tau^{-1}(x)}$ is the canonical $\Q$-top-form on the torus $\tau^{-1}(x)$ (as defined in \Cref{def:wtor})
\end{lemma}

\subsection{The fibers of $\tau:\bfX\to\bfG$}
Let $x\in G^{rss}$. In this subsection we prove that the algebraic group $\tau^{-1}(x)$ is (non-canonically) isomorphic to the centralizer $\bfG_x$ of $x$ - see \Cref{cor:Gx} below. 

We start with the following standard lemma:
\begin{lemma}\label{lem:ConjT}
    Let $x\in \bfG^{rss}(F)$. Then there exists $z\in \bfG(F^{sep})$ s.t. $zxz^{-1}\in \bfT(F^{sep})$.
\end{lemma}


Next, we describe certain fibers of the map $\pi: Y \to C$ in terms of centralizers.

\begin{lemma}\label{lem:Gx}
    Let $x\in G^{rss}$. Then the algebraic group $\pi^{-1}(p(x))$ is (non-canonically) isomorphic to the centralizer $\bfG_x$ of $x$.
\end{lemma}
\begin{proof}
    We will construct an isomorphism of $F^{sep}-$schemes 
    $$\eps:(\bfG_x)_{F^{sep}/F}\to\pi^{-1}(p(x))_{F^{sep}/F}$$
    and show that for any $F^{sep}$-scheme $S$, and for any $\gamma\in Gal(F^{sep}/F)$ the following diagram is commutative
  \begin{equation}\label{eq:pip}
        \xymatrix{
        (\bfG_x)_{F^{sep}/F}(S) \ar[d]^{\gamma} \ar[r]^{\eps} &\pi^{-1}(p(x))_{F^{sep}/F}(S)\ar[d]^{\gamma}\\
        (\bfG_x)_{F^{sep}/F}(S)  \ar[r]^{\eps} &\pi^{-1}(p(x))_{F^{sep}/F}(S)}
        \end{equation}   

\begin{enumerate}[Step 1.]
    \item Construction of $\eps$.\\
    By \Cref{lem:ConjT} we can choose $z\in \bfG(F^{sep})$ such that $zxz^{-1}\in \bfT^r(F^{sep}/F)$. Denote $y:=zxz^{-1}$. 
    Let 
    $$\mu_y:\{y\}\times \bfT_{F^{sep}/F}\to \pi^{-1}(p(y))_{F^{sep}/F} = \pi^{-1}(p(x))_{F^{sep}/F}$$ be the restriction of $(\mu)_{F^{sep}/F}$.   
    By definition of the group structure on $\pi^{-1}(p(y)),$
    $\mu_y$ is a group isomorphism. 
    Take $\eps$ to be the composition
    $$ (\bfG_x)_{F^{sep}/F}\overset{ad(z)}{\to} \bfT_{F^{sep}/F}\to \bfT_{F^{sep}/F}\times \{y\}\overset{\mu_y}{\to} \pi^{-1}(p(x))_{F^{sep}/F}.$$
    It is an isomorphism of algebraic groups (over $F^{sep}$). 
    \item Proof of commutativity of the diagram \eqref{eq:pip}.
    Let $n:=\gamma(z)z^{-1}$. 
    Note that $z(\bfG_x)_{F^{sep}/F}z^{-1}=(\bfG_y)_{F^{sep}/F}=\bfT_{F^{sep}/F}$ and thus 
    $$\gamma( z)(\bfG_{\gamma(x)})_{F^{sep}/F})\gamma(z)^{-1}=\gamma( z(\bfG_x)_{F^{sep}/F}z^{-1})=\gamma(\bfT_{F^{sep}/F})=\bfT_{F^{sep}/F}$$
    Thus $n$ normalizes $\bfT_{F^{sep}/F}$. Therefore $ad(n)$ acts on $\bfT$ by an element $w\in W$. 
    Let $u\in (\bfG_x)_{F^{sep}/F}(S)$. We have 
    \begin{align*}
    \eps(\gamma(u))&= \mu(z\gamma(u)z^{-1},y)= \mu (w\cdot z\gamma(u)z^{-1},w\cdot y) =  \mu (n z\gamma(u)z^{-1}n^{-1},n y n^{-1}) = \\ 
    &=\mu (\gamma(z) z^{-1}z \gamma(u)z^{-1}z\gamma(z)^{-1}, \gamma(z)z^{-1} z x z^{-1} z   \gamma(z)^{-1})=\\
    &=\mu(\gamma(zuz^{-1}), \gamma(z)x\gamma(z)^{-1})=\mu (\gamma(zuz^{-1}), \gamma(z)\gamma(x)\gamma(z)^{-1})=\\
    &=\mu( \gamma(zuz^{-1},zxz^{-1})) =\mu( \gamma(zuz^{-1},y)) =\gamma(\eps(u))
    \end{align*}
    
\end{enumerate}
        
\end{proof}

\begin{cor}\label{cor:Gx}
    Let $x\in G^{rss}$. Then the algebraic group $\tau^{-1}(x)$ is (non-canonically) isomorphic to the centralizer $\bfG_x$ of $x$.
\end{cor}

\subsection{Construction of $\mathcal{A}$ and its properties}
In this subsection we construct a clopen subset $\mathcal{A}\subset X$ s.t. $\tau|_{\mathcal{A}}$ is proper and a generic fiber of $\tau$ intersects $\mathcal{A}$ along the maximal compact subgroup of this fiber (see \Cref{cor:A.prop}).

The following lemma is straightforward:
\begin{lem}\label{lem:cancel.set}
Consider the following commutative diagram in arbitrary category.
$$
\begin{tikzcd}
    Z_{11} \arrow[d,"\delta_{1}"] \arrow[r, "\gamma_{11}"] & Z_{12}\arrow[dr, phantom, "\square"] \arrow[d,"\delta_{2}"] \arrow[r,"\gamma_{12}"]  & Z_{13} \arrow[d,"\delta_{3}"] \\
    Z_{21} \arrow[r,"\gamma_{21}"]  & Z_{22} \arrow[r, "\gamma_{22}"]  & Z_{23}
\end{tikzcd}
$$
Assume also that we have:
$$
\begin{tikzcd}
    Z_{11} \arrow[d,"\delta_{1}"] \arrow[r, "\gamma_{12}\circ\gamma_{11}"]\arrow[dr, phantom, "\square"] & Z_{13} \arrow[d,"\delta_{3}"]\\
    Z_{21} \arrow[r,"\gamma_{22}\circ\gamma_{21}"]  & Z_{23} 
\end{tikzcd}
$$
Then we have:
$$
\begin{tikzcd}
    Z_{11} \arrow[d,"\delta_{1}"] \arrow[r,"\gamma_{11}"] \arrow[dr, phantom, "\square"] & Z_{12} \arrow[d,"\delta_{2}"] \\
     Z_{21} \arrow[r, "\gamma_{21}"]  & Z_{22}
\end{tikzcd}
$$
\end{lem}

\begin{defn}
$\,$ 
\begin{itemize}
    \item Recall that $\alp:\bfY=(\bfT\times\bfT)/W\to \bfT/W\times \bfT/W=\bfC\times \bfC$ is the natural map. 
    \item Let $\mdef{\mathcal{B}}:=\alp^{-1}(\bfC(F)\times \bfC(O_F))\subset Y=\bfY(F)$.
    \item Let $\mdef{\mathcal{B}^r}:=\mathcal{B}\cap Y^r$.
\end{itemize}
\end{defn}

\begin{prop}\label{lem:BY}
    $ \,$
    \begin{enumerate}[(i)]
    \item $\mathcal{B}\subset Y$ is clopen.
    \item \label{it:piB}$\pi|_{\mathcal{B}}$ is proper.
    \item\label{lem:BY:max} For any $x\in C^{rss}:=\bfC^{rss}(F)$ the set $\pi^{-1}(x)(F)\cap \mathcal{B}$ is the maximal compact group of $\pi^{-1}(x)(F)$.    
    \end{enumerate}
\end{prop}

For the proof we will need the following lemmas.

\begin{lemma}\label{lem:Cart}
            Consider the commutative diagram 
              \begin{equation}
 \begin{tikzcd}
    T^r\times \bfT(O_F) \drar[phantom, "1"] \arrow[d] \arrow[r] & \mathcal{B}^r \drar[phantom, "2"] \arrow[d] \arrow[r] & C \times \bfC(O_F) \arrow[d] \\
    T^r\times T \arrow[d, "\fp\fr_1^r"] \arrow[r,"\mu^r"] \drar[phantom, "3"]& Y^r \arrow[d,"\pi^r"] \arrow[r,"\alpha|_{Y^r}"] & C\times C \\
    T^r \arrow[r,"q^r"] & C^{rss}
\end{tikzcd}
\end{equation} 

Then all the squares in this diagram are Cartesian.        
\end{lemma}
\begin{proof}
The square 2 is Cartesian by the definition of $\mathcal{B}$. The square 3 is Cartesian by \Cref{lem: caresian square N1}. It remains to show that 1 is a Cartesian square.         The fact that $O_F$ is integrally closed inside $F$ implies that the square
        \begin{equation}
        \xymatrix{
         \bfT(O_F) \ar[d] \ar[r] &\bfC(O_F)\ar[d]\\
         T \ar[r]^q &C}
        \end{equation}
        is  Cartesian.
Thus we 
have the Cartesian square
        \begin{equation}
        \xymatrix{
        C\times \bfT(O_F) \ar[d] \ar[r] & C\times \bfC(O_F) \ar[d]\\
        C\times T \ar[r]^{\Id\times q} &C \times C}.
        \end{equation} 
        Composing it with the Cartesian square 
        \begin{equation}
        \xymatrix{
        T^r\times \bfT(O_F) \ar[d] \ar[r] & C\times \bfT(O_F) \ar[d]\\
        T^r\times T \ar[r]^{q|_{T^r}\times q} &C\times T}.
        \end{equation} 
        we obtain the Cartesian square:
        $$\begin{tikzcd}
    T^r\times \bfT(O_F) \arrow[d] \arrow[r]  & C \times \bfC(O_F) \arrow[d] \\
    T^r\times T \arrow[r,"q|_{T^r}\times q"]   & C\times C 
\end{tikzcd}
$$
This square is also the composition of squares 1,2. Since we already showed that square 2 is Cartesian, it follows by \Cref{lem:cancel.set} that the square 1 is Cartesian.
\end{proof}

\begin{lemma}\label{lem:TE}
    Let \RamiC{$\bf S$} be a torus defined over $F$. Let $E/F$ be a finite field extension. Let $K\sub \RamiC{\bfS}(E)$ be the maximal compact subgroup. Then $K\cap \RamiC{\bfS}(F)$ is the maximal compact subgroup of $\RamiC{\bfS}(F)$.
\end{lemma}
\begin{proof}
    This follows from the uniqueness of the maximal compact subgroup of a torus.
\end{proof}

\begin{proof}[Proof of \Cref{lem:BY}]
$ $
\begin{enumerate}[(i)]
    \item is obvious. 
    \item  Consider the diagram:
 \[
\begin{tikzcd}
    \mathcal{B}\arrow[dd,bend right=30,"\pi|_B"'] \arrow[d] \arrow[r,"\alpha_B"]\drar[phantom, "\square"] & C \times \bfC(O_F) \arrow[d] \arrow[ddl, bend left=110, "pr"]  \\
    Y \arrow[d, "\pi"] \arrow[r, "\alpha"] & C\times C \\
     C
\end{tikzcd}
\]   
Here $pr$ is the projection to the first coordinate and $\alpha_B$ is the restriction of $\alpha$.
    The morphism $\alp$ is finite, thus proper on the level of $F$-points. Therefore, $\alpha_B$ is proper. Since $\bfC(O_F)$ is compact, we obtain that $pr$ is proper.
    Thus $\pi|_B=pr \circ \alpha_B$ is proper.
    \item \begin{enumerate}[Step 1.]
        \item Proof for the case when $x\in q(T^r)$.\\ 
        Follows from the Cartesian squares 
        \begin{equation}
    \begin{tikzcd}
        T^r\times \bfT(O_F) \drar[phantom, "\square"]\arrow[d] \arrow[r] & \mathcal{B}^r \arrow[d]  \\
        T^r\times T \arrow[d] \arrow[r]\drar[phantom, "\square"] & Y^r \arrow[d]  \\
        T^r \arrow[r] & C^{rss}
    \end{tikzcd}
        \end{equation}      
        given by \Cref{lem:Cart}.
        \item General case.\\
        Follows from the previous case and \Cref{lem:TE}.        
    \end{enumerate}
\end{enumerate}
\end{proof}
\begin{notation}
    $\mdef{\mathcal{A}}=G\times_C \mathcal{B}\subset X$.
\end{notation}
\Cref{lem:BY} gives us:
\begin{cor}\label{cor:A.prop}
    $ \,$
    \begin{enumerate}[(i)]
    \item $\mathcal{A}\subset X$ is clopen.
    \item \label{it:piA}$\tau|_{\mathcal{A}}$ is proper.
    \item \label{cor:A.prop:max}For any $x\in G^{rss}$ the set $\tau^{-1}(x)(F)\cap \RamiC{\mathcal{A}}$ is the maximal compact \Dima{sub}group of $\tau^{-1}(x)(F)$.    
    \end{enumerate}
\end{cor}

\subsection{Proof of \Cref{thm:kappa.alg.geo}}
    It is enough to show that $$\tau_*((|\RamiA{\omega'_\bfX}|)|_{\mathcal{A}})=\RamiC{\kappa^0} |\omega_\bfG|.$$
    For this it is enough to show that $$\tau_*((|\omega_{\tau}|)|_{\mathcal{A}})=\RamiC{\kappa^0},$$ almost everywhere. 
    For this it is enough to show that for every $x\in G^{rss}$, we have    
    $$\int_{\mathcal{A}\cap \tau^{-1}(x)}\left|\omega_{\tau}|_{\tau^{-1}(x)}\right|=\RamiC{\kappa^0}(x).$$
    Fix $x\in G^{rss}$. 
    Recall that $K_x$ denotes the maximal compact subgroup of $\RamiC{\bfG_x}$. By \Cref{lem:Gx} we can choose an isomorphism $\gamma:\RamiC{\bfG_x}\simeq \tau^{-1}(x)$. The group $\gamma(K_x)$ is the maximal compact subgroup of $\tau^{-1}(x)$. So, by Corollary  \ref{cor:A.prop}\eqref{cor:A.prop:max}, $\gamma(K_x)= \mathcal{A}\cap \tau^{-1}(x)$. Thus we have 
    $$\int_{\mathcal{A}\cap \tau^{-1}(x)}\left |\omega_{\tau}\vert_{\tau^{-1}(x)}\right|=
    \int_{K_x}\left |\gamma^*(\omega_{\tau}\vert_{\tau^{-1}(x)})\right|$$
    By Lemma  \ref{lem:om.atu}, $\gamma^*(\omega_{\tau}\vert_{\tau^{-1}(x)})=\omega_{\RamiC{\bfG_x}}$. Thus we obtain
    $$\int_{K_x}\left |\gamma^*(\omega_{\tau}\vert_{\tau^{-1}(x)})\right|=\int_{K_x} |\omega_{\RamiC{\bfG_x}}|=\RamiC{\kappa^0}(x).$$


\section{Regularity of $\omega_\bfX$}\label{Sec:regularity}
Recall that $\bfX =\bfG\times_{\bfC}\bfY,$ with $\bfY :=(\bfT\times \bfT)//W$, where $W$ acts diagonally and that $\mu:\bfT\times \bfT\to \bfY$ is the quotient map.

In this section we prove the following theorem.
\begin{theorem}\label{thm:reg}
    $\omega_\bfX$ is a regular $\Q$-top differential form on the smooth locus of $\bfX$.
\end{theorem}
Before we begin the proof we give a short description of the idea.
Recall that by \Cref{lem:Gp}, $\bfG'=\bfG\times_\bfC \bfT$   
    is absolutely reduced, locally  complete intersection, and irreducible.
The idea of the proof is as follows: we pullback $\omega_\bfX$ under  $$\nu:\bfG'\times \bfT =\tilde \bfX\to \bfX,$$  and obtain a form that can be written as a product $\omega_{\bfG'}\boxtimes \omega_{\bfT}$. The form $\omega_{\bfG'}$ has an explicit description, see \Cref{not.form} below. We deduce the regularity of $\omega_\bfX$ from the regularity of $\omega_{\bfG'}$ which we prove in \S \ref{ssec:reg} below.

\begin{notation}
\DimaA{
 For a Cartezian square    
} 
$$ \begin{tikzcd}
    \bfZ_{11} \arrow[d,"\delta_{1}"] \arrow[r, "\gamma_{1}"]\arrow[dr, phantom, "\square"] & \bfZ_{12} \arrow[d,"\delta_{2}"] \\
    \bfZ_{21} \arrow[r,"\gamma_{2}"]  & \bfZ_{22} 
\end{tikzcd}
$$
\DimaA{
and a relative (rational $\bQ$-)top form $\omega_{\delta_2}$ on $\bfZ_{12}$ w.r.t. $\delta_2$ we denote by 
$\gamma_2^*(\omega_{\delta_2})$ its pullback to a
relative (rational $\bQ$-)top form on $\bfZ_{11}$ w.r.t. $\delta_1$. 

As the bundle of $\delta_1$-relative top-differential forms on $\bfZ_{11}$  is the pullback of the bundle of $\delta_2$-relative top-differential forms on $\bfZ_{12}$ w.r.t. $\gamma_1$, one can also denote the form $\gamma_2^*(\omega_{\delta_2})$ by 
$\gamma_1^*(\omega_{\delta_2})$, as we did in \Cref{defn:om.X}.
}
\end{notation}

\begin{notation}Define the following algebraic varieties.
\begin{enumerate}[(i)]
    \item $\widetilde{\bfX}^{rss}:=\bfG^{rss}\times_{\bfC^{rss}}\bfT^r\times \bfT$
    \item $\widetilde{\bfX}^f:=\bfG \times _{\bfC}(\bfT\times \bfT)^f$
\end{enumerate}    
\end{notation}

\begin{lemma}\label{lem:mu}
$\,$
Let $\mu:\bfT\times \bfT\to \bfY$ be the quotient map. Then 
    \begin{enumerate}[(i)]
    \item\label{lem:mu:1} $\mu$ is finite.
    \item\label{lem:mu:2} $\mu|_{(\bfT\times \bfT)^f}$ is {\'e}tale.
\end{enumerate}    
\end{lemma}
\begin{proof}
Items \eqref{lem:mu:1} follows from the fact that the action of $W$ on $\bfT\times \bfT$ is factorizable, see \Cref{lem:fac}.
    Item \eqref{lem:mu:2} follows from Lemmas \ref{lem:Gam} and \ref{lem:GamU}. 
\end{proof}

\begin{cor}\label{cor:nu}
$\,$
    \begin{enumerate}[(i)]
    \item \label{cor:nu:1} $\nu$ is finite.
    \item \label{cor:nu:2} $\nu|_{\widetilde{\bfX}^f}$ is {\'e}tale.
    \item \label{cor:nu:3} $\widetilde{\bfX}^f\subset \widetilde{\bfX}$  is big in $\Dima{\widetilde{\bfX}}$. 
\end{enumerate}    
\end{cor}
\begin{proof}
    Items \eqref{cor:nu:1} and \eqref{cor:nu:2} follow from the previous lemma. Item \eqref{cor:nu:3} follows from the fact that $\bfY^f\subset \bfY$ is big (\Cref{cor:YfBig}) and the fact that $p$ \RamiA{(and hence $\sigma$) are} flat (\Cref{cor:pqflat}).
\end{proof}

As we will see  below, this lemma implies that in order to prove \Cref{thm:reg} it is enough to show that 
$\nu^*(\omega_\bfX)$ is a regular $\Q$-top differential form on the smooth locus of $\widetilde \bfX$. 

\begin{notation}\label{not.form}
    $ $
        Recall that $\psi:\bfG'\to \bfG$ is the projection on the first factor.
        Let $$\mdef{\omega_{\bfG'}}:=\psi^*(\omega_\bfG\cdot \Delta^{-\frac{1}{2}})$$
\end{notation}
\begin{lemma}\label{lem:omegaDiag}
$\nu^*(\omega_{\bfX})=\omega_{\bfG'} \boxtimes \omega_\bfT$.
    \end{lemma}
\begin{proof}
    Recall that  $\fp\fr_{\bfG'}:\widetilde{\bfX}=\bfG'\times \bfT\to \bfG'$ is the projection. Denote $\bfX^{rss}:=\bfG^{rss}\times_{\bfC^{rss}}\bfY^r$. Consider the following diagram 
        \begin{equation}
        \begin{tikzcd}[arrows={-Stealth},row sep=large]
        \widetilde{\bfX}^{rss} \arrow[bend left=10, "\fp\fr_\bfT^{rss}"]{rrrr}
        \dar[swap, "\fp\fr_{\bfG'}^{r}"] \drar[phantom, "1"] \rar[swap, "\nu^r"] & \bfX^{rss}\dar["\tau^r"]\rar \drar[phantom, "2"] &\bfY^r \dar["\pi^r"] \drar[phantom, "3"]  & \bfT^r\times \bfT  \lar  \dar["\fp\fr^r_1"] \rar \drar[phantom, "4"]& \bfT \dar["\phi_\bfT"]  \\
        (\bfG')^{rss}\rar["\psi^{rss}"]  \arrow[swap, bend right=10, "\phi_{(\bfG')^{rss}}"]{rrrr} & \bfG^{rss}\rar["p^r"] & \bfC^{rss}  &\bfT^r \lar[swap, "q^r"]\rar["\phi_{\RamiA{\bfT^r}}"] & pt  
        \end{tikzcd}
        \end{equation} 
        where \RamiA{$(\bfG')^{rss}:=\bfG^{rss}\times_{\bfC^{rss}} \bfT^r$,} the maps $\fp\fr_{\bfG'}^{r}, \tau^r,\pi^r, \nu^r,p^r,q^r, \psi^{rss},\fp\fr_\bfT^{rss}$ are obtained by restriction of the maps $\Dima{\fp\fr}_{G'}, \tau, \pi,\nu,p,q,\psi,\fp\fr_\bfT$\RamiA{, and $\phi_\bfT, \phi_{\bfT^r}$ and $\phi_{(\bfG')^{rss}}$ are the projections to the point}.
        The squares $2,4$ are Cartesian by definition, the square $3$ is Cartesian by \Cref{lem: caresian square N1}, and the square $1$ is Cartesian since it is the base change of square $3$ along square $2$. Also, the outside square 

        \begin{equation}
        \begin{tikzcd}[arrows={-Stealth},row sep=large]
        \widetilde{\bfX}^{rss} \arrow["\fp\fr_\bfT^{rss}"]{r}
        \dar[swap, "\fp\fr_{\bfG'}^{r}"] & \bfT \dar["\phi_{\RamiA{\bfT}}"]  \\
        (\bfG')^{rss} \arrow[swap,"\phi_{\RamiA{(\bfG')^{rss}}}"]{r} &pt  
        \end{tikzcd}
        \end{equation}        
        is Cartesian \RamiA{by definition}. 
        Consider  $\omega_\bfT$ as a relative rational $\Q$-differential form with respect to $\phi_\bfT$. Denote $\omega_{\fp\fr_{\bfG'}}=\phi_{(\bfG')^{rss}}^*(\omega_{{\DimaB{\bfT}}})$ and consider it as a relative rational $\Q$-top differential form on $\RamiA{\widetilde{\bfX}}$ w.r.t. to $\fp\fr_{\bfG'}$. 

For each vertical arrow in the above diagram we have a relative form. These forms are compatible with all the squares possibly except (a-priori) square 1. We would like to show that it is compatible with square 1 as well. Explicitly, we have:
        \begin{eqnarray}\label{eq:pulls}
            \phi_{\RamiA{\bfT^r}}^*(\omega_{\bfT})=\omega_{\fp\fr_1^r}\\
            (q^r)^*(\omega_{\pi})=\omega_{\fp\fr_1^r}\\
            (p^r)^*(\omega_{\pi})=\omega_{\tau}\\
            \phi_{\RamiA{(\bfG')^{rss}}}^*(\omega_{\bfT})=\omega_{\fp\fr_{G'}}\label{eq:pulls.last}
        \end{eqnarray}
        We would like to deduce that  $(\psi^{rss})^*(\omega_{\tau})=\omega_{pr_{G'}^{r}}$. 
        It is enough to check  this equality after extension of scalars to $\bar F$. For this it is enough to show that for any $x\in  (G')^{rss}(\bar F)$ we have 
$$(\psi^{rss})^*(\omega_{\tau})|_{(\fp\fr_{G'}^{r})^{-1}(x)}=\omega_{\fp\fr_{G'}^{r}}|_{(\fp\fr_{G'}^{r})^{-1}(x)}.$$ This follows from \eqref{eq:pulls}-\eqref{eq:pulls.last}. We  obtained:
        $$(\psi^{rss})^*(\omega_{\tau})=\omega_{\fp\fr_{G'}^{r}}.$$
        Now we have 
        $$\nu^*(\omega_{\bfX})=\nu^*((\omega_\bfG \cdot \Delta^{-1/2})*\omega_{\tau})=\psi^*(\omega_\bfG \cdot \Delta^{-1/2})*\psi^*(\omega_{\tau})= \omega_{\bfG'}*\omega_{\fp\fr_{\bfG'}^{r}}= 
         \omega_{\bfG'} \boxtimes \omega_\bfT$$
\end{proof}

\begin{lemma}\label{lem:reg}
$\omega_{\bfG'}$ is regular on the smooth locus of $\bfG'$.
\end{lemma}
We postpone the proof of this lemma to \S\ref{ssec:reg}. Let us now deduce \Cref{thm:reg}. 

\begin{proof}[Proof of \Cref{thm:reg}]
By \Cref{lem:reg}, $\omega_{\bfG'}$ is regular on the smooth locus of $\bfG'$. Therefore, by \Cref{lem:omegaDiag}, $\nu^*(\omega_{\bfX})$ is regular on the smooth locus of $\tilde \bfX$. Therefore, by \Cref{cor:nu}\eqref{cor:nu:2}, $(\omega_{\bfX})|_{\nu(\tilde \bfX^f)}$ is regular on the smooth locus of $\nu(\tilde \bfX^f)$. By \Cref{cor:nu}(\ref{cor:nu:1},\ref{cor:nu:3}), the open set  $\nu(\tilde \bfX^f)$ in $\bfX$ is big. Therefore $\omega_{\bfX}$ is regular on the smooth locus of $\bfX$, as required.    
\end{proof}

\subsection{Proof of \Cref{lem:reg}}\label{ssec:reg}
We will use the following ad-hoc definition. 
\begin{definition}
    Let $\phi:\bfC\to \bfD$  be a finite map of algebraic varieties defined over $F$, with $\bfD$ being smooth. Let $f$ be a rational $\Q$-function on  $\bfD$. We say that the pair $(\phi,f)$ is \tdef[good pair]{good}, if for any open set $\bfU\subset \bfD$ and any (regular) top-differential form $\omega$ on $\bfU$, the rational $\Q$-form $\phi^{*}(f\cdot \omega)$ is regular on the smooth locus of $\phi^{-1}(\bfU)$.
\end{definition}
\begin{lemma}\label{lem:good.prop}
$\,$
Let $\phi:\bfL\to \bfD$  be a finite map of algebraic varieties defined over $F$, with $\bfD$ being smooth. Let $f$ be a rational $\Q$-function on $\bfD$.
\begin{enumerate}[(i)]
    \item \label{it:good:inv}If there is an invertible form $\omega$ on $\bfD$ s.t. the rational $\Q$-form $\phi^{*}(f\cdot \omega)$ is regular on $\bfD$ then $(\phi,f)$ is good.
    \item \label{it:good:smooth} The property of being good is local on $\bfD$ in the smooth topology. I.e.
    \begin{enumerate}[(a)]
    \item \label{it:pull:good}If  $(\phi,f)$  is good and $\gamma:\bfE\to \bfD$ is smooth then $(\gamma^*(\phi),\gamma^*(f))$ is good.
    \item \label{it:good:pull} If  $\gamma:\bfE\to \bfD$ is smooth and surjective and $(\gamma^*(\phi),\gamma^*(f))$ is good then $(\phi,f)$  is good.
    \end{enumerate}
    \item \label{it:good:codim} If  $\bfU\subset \bfL$ is big and $(\phi|_{\bfU},f)$ is good then 
    $(\phi,f)$  is good.
    
\end{enumerate}
\end{lemma}

\begin{proof}
    (\ref{it:good:inv}, \ref{it:good:codim}) are obvious. Let us prove \eqref{it:good:smooth}. 
    \begin{enumerate}[{Case} 1.]
        \item\label{lem:good.prop:1} $\gamma: \coprod \bfU_i\to \bfD$ is a Zariski cover of an open subset of $\bfD$.\\
        This case is obvious.
        
        \item $\gamma$ is an {\'e}tale map, and $\bfD$ admits an invertible top differential form.\\
        In this case \eqref{it:pull:good} follows from \eqref{it:good:inv}, and \eqref{it:good:pull} is trivial. 

        \item\label{lem:good.prop:c3} $\gamma$ is an {\'e}tale map.\\
        Follows from the two previous cases. 

        \item $\gamma$ can be decomposed as $\bfU\overset{i}{\to}\bfD\times \A^n\overset{pr}{\to}\bfD$ where $i$ is an open embedding and $pr$ is the projection, and $\bfD$ admits an invertible top differential form. \\
        In this case \eqref{it:pull:good} follows from \eqref{it:good:inv}, and \eqref{it:good:pull} is trivial. 
        \item\label{lem:good.prop:c5} $\gamma$ can be decomposed as $\bfU\overset{i}{\to}\bfD\times \A^n\overset{pr}{\to}\bfD$ where $i$ is an open embedding and $pr$ is the projection. \\
        Follows from the previous case and Case \ref{lem:good.prop:1}.
        \item the relative dimension of $\gamma$  is constant.\\
        By \cite[\href{https://stacks.math.columbia.edu/tag/054L}{Lemma 054L}]{SP}, we can find a commutative diagram 
        $$
        \begin{tikzcd}[arrows={-Stealth},row sep=large]
        \bfE \arrow["\gamma"]{r}
        & \bfD    \\
        \tilde \bfE\uar["\varepsilon"] \arrow[swap,"\gamma_1"]{r} &\bfD\times\bA^n  \uar["pr"],
        \end{tikzcd}
$$
        Where $pr$ is the projection, $\varepsilon$ is surjective {\'e}tale map and $\gamma_1$ is {\'e}tale.
        \begin{itemize}
            \item Proof of \eqref{it:pull:good}\\
            If $(\phi,f)$ is good then, by Case \ref{lem:good.prop:c5} so is $(pr^*(\phi),pr^*(f))$. Thus, by Case \ref{lem:good.prop:c5} so is 
            $$(\gamma_1^*pr^*(\phi),\gamma_1^*pr^*(f))=(\varepsilon^*\gamma^*(\phi),\varepsilon^*\gamma^*(f)).$$  Therefore, by Case \ref{lem:good.prop:c3}\eqref{it:good:pull} the pair $(\gamma^*(\phi),\gamma^*(f))$ is good.
            \item Proof of \eqref{it:good:pull}\\
            If $(\gamma^*(\phi),\gamma^*(f))$ is good then, by Case \ref{lem:good.prop:c3}\eqref{it:pull:good} so is $$(\varepsilon^*\gamma^*(\phi),\varepsilon^*\gamma^*(f))=(\gamma_1^*pr^*(\phi),\gamma_1^*pr^*(f)).$$ Thus, by Case \ref{lem:good.prop:c3} so is 
            $(pr^*(\phi)|_{\Im \gamma_1},pr^*(f)_{\Im \gamma_1})$.  Therefore, by Case \ref{lem:good.prop:c5}, the pair $(\phi,f)$ is good.
        \end{itemize}
            \item General case.\\
            Follows immediately from the previous case.
    \end{enumerate}
\end{proof}

\begin{notation}
    $ $
    \begin{enumerate}
        \item Let $\mdef{(\bfG')^r}:=\psi^{-1}(\bfG^r)\subset \bfG'$.
        \item Let $\mdef{\psi^r}:(\bfG^r)'\to \bfG^r$ be the restriction of $\psi$.
    \end{enumerate}    
\end{notation}
\Cref{lem:lin.alg}\eqref{lem:lin.alg:2} gives us:
\begin{cor}\label{lem:good}
    The  pair $(q,\Delta_C^{-1/2})$ is good.
\end{cor}
\begin{proof}[Proof of \Cref{lem:reg}]
We have to show that 
    $(\psi, \Delta_{\bfG}^{-\frac{1}{2}})$ is good.

Consider the Cartesian square 
        \begin{equation}
        \xymatrix{
        (\bfG')^r \ar[d]^{\psi^r} \ar[r] &\bfT\ar[d]^q\\
        \bfG^r \ar[r]^{p|_{G^r}} &\bfC}
        \end{equation}

By Lemmas \ref{lem:good.prop}\eqref{it:good:smooth} \RamiA{and} \ref{lem:Gr.onto}, the last corollary (\Cref{lem:good}) implies that $(\psi^r, \Delta_{\bfG}|_{\bfG^r}^{-\frac{1}{2}})$ is good. Therefore $(\psi|_{(\bfG')^r}, \Delta_{\bfG}^{-\frac{1}{2}})$ is good.

Since $q$ is finite and flat (see \Cref{lem:fac} and \Cref{lem:q.flat}), so is $\psi$. So \Cref{cor:pFibReduced} implies that 
$(\bfG')^r$ is big in $\bfG'$. Therefore, by \Cref{lem:good.prop}\eqref{it:good:codim} we obtain that 
    $(\psi, \Delta_{\bfG}^{-\frac{1}{2}})$ is good, as required.
\end{proof}

\section{Regularity and invertability of the form $\omega^0_\bfX$}\label{sec:om0}
In this section we construct the rational form $\omega^0_\bfX$ on $\bfX$ and prove the following theorem.
\begin{theorem}\label{thm:inv}
    $\omega^0_\bfX$ is regular and invertible over the smooth locus of $\bfX$.
\end{theorem}
We also prove regularity and invertability of some other forms (see \Cref{lem:invY} and \Cref{cor.inv.UP} below).
\begin{notation}\label{not:omegaX0}
$ $
    \begin{enumerate}
        \item Let $\mdef{\omega_{\bfT\times \bfT}}:=\omega_{\bfT}\boxtimes \omega_{\bfT}$
        
        \item Note that {$\omega_{\bfT\times \bfT}$} is $W$ invariant, since for any $w\in W$ we have $$w^*(\omega_{\bfT\times \bfT})=sign(w)^2 \omega_{\bfT\times \bfT}=\omega_{\bfT\times \bfT}.$$  So, by \Cref{cor:GalDes}\eqref{it:sh} it descends to a  top form on $\bfY^f$. Denote this form by \tdef{$\omega_\bfY$} and interpret is as a rational top form on $\bfY$.
        \item $\mdef{\omega^0_\bfX}:= \omega_\bfG\boxtimes_{\omega_\bfC} \omega_\bfY$
        \item Let \tdef{$\bfY^{sm}$} be the smooth locus of $\bfY$ and $\mdef{\bfX^0}:=\bfG^r\times_\bfC \bfY^{sm}$.
    \end{enumerate}
\end{notation}
The following lemma is obvious.
\begin{lemma}\label{lem:invY}
    $\omega_\bfY$ is regular and invertible over the smooth locus of $\bfY$.
\end{lemma}

\begin{lemma}\label{lem:X0big}
    $\bfX^0$ is smooth, and is big in $\bfX$.
\end{lemma}
\begin{proof}
The map $\bfX^0\to \bfY^{sm}$ is a base change of a smooth map, and hence is smooth. Thus $\bfX^0$ is smooth.     Since $q:\bfY\to \bfC$ is flat (see \S\ref{sec:basic}), we have 
$$\dim \bfX = \dim\bfG+\dim \bfY-\dim \bfC,$$ and $$\dim (\bfG\smallsetminus \bfG^r)\times_{\bfC} \bfY= \dim (\bfG\smallsetminus \bfG^r)+\dim \bfY-\dim \bfC.$$ 
As $\bfX$ is irreducible (see \S\ref{sec:basic}) and $\bfG^r$ is big in $\bfG$ (see \S\ref{sec:basic}), this implies that $\bfG^r \times_{\bf C} \bfY$ is big in $\bfX$. 
By \S\ref{sec:basic} $\bfY^{sm}$ is big in $\bfY$.  
Similarly to the above argument, we obtain that $\bfG\times_\bfC \bfY^{sm}$ is big in $\bfX$. Thus, $\bfX^0$ is big in $\bfX$.
\end{proof}

\begin{proof}[Proof of \Cref{thm:inv}]

By \Cref{lem:X0big}  it is enough to show that $\omega_\bfX^0$ is regular and invertible on $\bfX^0$. Consider the diagram:

$$ 
\begin{tikzcd}[arrows={-Stealth}]
& \bfC\times\bfC\drar["d"]&\\
    \bfG^r \times \bfY^{sm}
    \arrow[phantom, "\square"]{drr}
    \arrow{rr}
    \urar["(p|_{\bfG^r})\times (q|_{\bfY^{sm}})"] 
    & & \bfC    
    \\      
  \bfX^0 \uar["(\tau{,}\sigma)"]
  \arrow[""]{rr}  & & 1\uar[""]
\end{tikzcd}
$$
where $d$ is the ratio map w.r.t. the group structure on $\bfC$ and $1$ is the neutral element w.r.t. this structure.

It is easy to see that $\omega^0_\bfX|_{\bfX^0}$ is, up to a sign, the Gelfand-Leray form w.r.t. the smooth map $d\circ ((p|_{\bfG^r})\times (q|_{\bfY^{sm}}))$  and the forms $\omega_\bfG|_{\bfG^r}\boxtimes \omega_\bfY|_{\bfY^{sm}}$ and $\omega_\bfC$. Hence it is regular and invertible. 
\end{proof}



Finally, we introduce some more forms and prove their regularity and integrability.
\begin{notation}
Denote
\begin{itemize}
    \item 
     $\mdef{\omega_{\bf \Upsilon}}:=\omega_{\bfG}\boxtimes_{\omega_{\bfC}}\omega_{\bfX}^0=\omega_{\bfG}\boxtimes_{\omega_{\bfC}}\omega_{\bfG}\boxtimes_{\omega_{\bfC}}\omega_{\bfY},$ considered as a rational top form on ${\bf \Upsilon}$.
    \item      $\mdef{\omega_{\bfG\times_\bfC\bfG}}:=\omega_{\bfG}\boxtimes_{\omega_{\bfC}}\omega_{\bfG}$ considered as a rational top form on ${\bfG\times_\bfC\bfG}$.
\end{itemize}
Note that here we use that the relevant maps are generically smooth, as guaranteed by \Cref{lem:gen.sm}.
\end{notation}
\Cref{thm:inv} gives us:
\begin{cor}\label{cor.inv.UP}
$ $
\begin{enumerate}[(i)]
    \item\label{cor.inv.UP:1} The form $\omega_{{\bf \Upsilon}}$ is invertible on the regular locus of ${\bf \Upsilon}$.
    \item\label{cor.inv.UP:2} The form $\omega_{\bfG\times_\bfC\bfG}$ is invertible on the regular locus of $\bfG\times_\bfC\bfG$.
    \end{enumerate}
\end{cor}

\begin{proof}
$ $
\begin{itemize}
\item[\eqref{cor.inv.UP:1}] Let ${\bf \Upsilon}'$ be the smooth locus of $p'$. The above shows that ${\bf \Upsilon}'':=(p')^{-1}(\bfY^{sm})\cap {\bf \Upsilon}'$ is big in ${\bf \Upsilon}$ also, 
    by \Cref{thm:inv}, $\omega_{\bfX}^0$  is invertible. By definition so is 
    $\omega_\bfG$.  Therefore $\omega_{\bf \Upsilon}|_{{\bf \Upsilon}''}$ is invertible. This implies the assertion.
    \item[\eqref{cor.inv.UP:2}]
    The proof is similar to the proof of previous item.
\end{itemize}
\end{proof}

\section{Explicit geometric bound\RamiA{s} on the character}\label{sec:Pfbnd234}
\subsection{Proof of \Cref{thm:bnd3}}\label{sec:Pfbnd3}
Let $m,\alpha^\rho,f$ be as in \Cref{thm:OrbIntBoundChar}.
Let $f'=1_\cB$, and $h=|m|$. By \Cref{lem:BY}, $\pi|_{\supp(f')}$  is proper.

By \Cref{thm:bnd.cusp.char}
  there exists $\gamma_0\in C^{\infty}(G)$ such that     
    \begin{equation}\label{=OmegaKappa}
    \Omega(|m|)= \gamma_0|_{G^{rss}} (p^{rss})^*\left(\frac{p^{rss}_*(|m|\mu_G|_{G^{rss}})}{\mu_C|_{C^{rss}}}\right) \kappa    
    \end{equation}

Let $k,C$ be s.t. $f<C\cR^k$.
Let $g\in C_c^\infty(G)$.
By \Cref{thm:OrbIntBoundChar} we have:
$$\langle \chi_\rho,g \mu_G \rangle\leq \langle f\cdot \Omega(|m|),(|g|\cdot \mu_G)|_{G^{rss}}\rangle\RamiA{\leq} \langle C\cR^k\cdot \Omega(|m|),(|g|\cdot \mu_G)|_{G^{rss}}\rangle.$$

By \eqref{=OmegaKappa} we obtain 
$$\langle \chi_\rho,g \mu_G \rangle\leq 
\langle C\cR^k\cdot \gamma_0|_{G^{rss}} (p^{rss})^*\left(\frac{p^{rss}_*(|m|\mu_G|_{G^{rss}})}{\mu_C|_{C^{rss}}}\right) \kappa  ,(|g|\cdot \mu_G)|_{G^{rss}}\rangle.$$

By \Cref{thm:kappa.alg.geo} we obtain  
\begin{align*}\langle \chi_\rho,g \mu_G \rangle &\leq 
\langle C\cR^k\cdot \gamma_0|_{G^{rss}} (p^{rss})^*\left(\frac{p^{rss}_*(|m|\mu_G|_{G^{rss}})}{\mu_C|_{C^{rss}}}\right) \frac{\tau_*(1_{\mathcal{A}}\cdot |\omega_\bfX| )|_{G^{rss}}}{\RamiA{(|\omega_G|)}|_{G^{rss}}} ,(|g|\cdot \mu_G)|_{G^{rss}}\rangle=
\\&=
\langle C\cR^k\cdot \gamma_0 p^*\left(\frac{p_*(|m|\mu_G)}{\mu_C}\right) \frac{\tau_*(1_{\mathcal{A}}\cdot |\omega_\bfX| )}{\RamiA{|\omega_G|}} ,|g|\cdot \mu_G\rangle
\\&=
\langle C\cR^k\cdot \gamma_0 p^*\left(\frac{p_*(|m|\mu_G)}{\mu_C}\right) \frac{\tau_*(1_{\mathcal{A}} \cdot  \left|\frac{\omega_\bfX}{\omega^0_\bfX}\right| \cdot |\omega^0_\bfX|)}{\RamiA{\frac{|\omega_G|}{\mu_G}}\mu_G} ,|g|\cdot \mu_G\rangle
\end{align*}

Let $\cF:=\RamiA{\frac{\mu_G}{|\omega_G|}} \left|\frac{\omega_\bfX}{\omega^0_\bfX}\right|$.
By Theorems \ref{thm:reg} and \ref{thm:inv},  and \Cref{cor:Xbig}, the function $\cF$ is continuous. 
Let $\cF'\in C^\infty(X)$ 
be a real valued function s.t. 
$\cF'>\cF$.
Set $\gamma=\tau^*(\gamma_0C)\cF'$.

We obtain:
\begin{align*}\langle \chi_\rho,g \mu_G \rangle &\leq
\langle C\cR^k\cdot \gamma_0 p^*\left(\frac{p_*(|m|\mu_G)}{\mu_C}\right) \frac{\tau_*(1_{\mathcal{A}} \cdot  \cF \cdot |\omega^0_\bfX|)}{\mu_G} ,|g|\cdot \mu_G\rangle
\\&\leq
\langle C\cR^k\cdot \gamma_0 p^*\left(\frac{p_*(|m|\mu_G)}{\mu_C}\right) \frac{\tau_*(1_{\mathcal{A}} \cdot  \cF' \cdot |\omega^0_\bfX|)}{\mu_G} ,|g|\cdot \mu_G\rangle
\\&=
\RamiA{
\left  \langle \frac{{\tau}_*(|\omega^0_\bfX| \tau^*(\gamma_0C)\cF' \sigma^*(1_\cB))}{\mu_G} p^*\left( \frac{p_*(h\mu_G )}{\mu_C}\right)\cR^k, |g| \mu_G\right\rangle
}
\\&=
\left  \langle \frac{\RamiA{\tau}_*(|\omega^0_\bfX| \gamma \sigma^*(f'))}{\mu_G} p^*\left( \frac{p_*(h\mu_G )}{\mu_C}\right)\cR^k, |g| \mu_G\right\rangle
\end{align*}
as required.
\subsection{Base change for integration}
In order to pass from \Cref{thm:bnd3}  to the other versions of \Cref{thm:bnd2} we will need the following:
\begin{lem}\label{lem:m.base.cha}
    Let     
    $$
\begin{tikzcd}
\bfZ_1 \arrow[r, "\delta'"] \arrow[d,"\gamma'",swap] & \bfZ_2 \arrow[d,"\gamma"] \\
\bfZ_3 \arrow[r, "\delta",swap] & \bfZ_4
\arrow[ul, phantom, "\square"]
\end{tikzcd}
$$
 be  a Cartesian square of algebraic varieties. Assume that all the maps in this diagram are generically smooth. Let $\omega_i$ for $i=2,3,4$ be invertible $\Q$-forms on the smooth loci of $\bfZ_i$ and let $\omega_1:=\omega_2 \boxtimes_{\omega_4} \omega_3$.
Let $Z_i:=\bfZ_i(F)$. Let $h_3\in C^\infty(Z_3)$ s.t. $\delta|_{\supp(h_3)}$ is proper. 
Then,
\begin{enumerate}
    \item \label{lem:m.base.cha:1} 
$$\gamma^*\left(\frac{\delta_*(h_3\cdot |\omega_3|)}{|\omega_4|}\right)=\frac{\delta'_*\left((\gamma')^*(h_3) \RamiC{\cdot} |\omega_1|\right)}{|\omega_2|} $$
\item \label{lem:m.base.cha:2}For every $h_2\in C^\infty(Z_2)$ s.t. $\gamma|_{\supp(h_2)}$ is proper we have:
$$\frac{\delta_*(h_3\cdot |\omega_3|)}{|\omega_4|} \frac{\gamma_*(h_2\cdot |\omega_2|)}{|\omega_4|}=\frac{(\gamma\circ\delta')_*(h_2\boxtimes_{Z_4} h_3\cdot |\omega_1|)}{|\omega_4|}.$$
\end{enumerate}
Note that these are equalities of functions that are defined only almost everywhere and\RamiC{, in particular,} are valid also only almost everywhere. 
\end{lem}
\begin{proof}
    $     $
    \begin{enumerate}[{Case} 1.]
        \item The varieties and maps in the diagram are smooth.\\
        This is  a straightforward computation.
        \item General case.\\
        Follows from the previous case.
    \end{enumerate}
\end{proof}

\subsection{Proof of \Cref{thm:bnd4}}\label{sec:Pfbnd4}
\Cref{lem:m.base.cha}\eqref{lem:m.base.cha:1} gives us:
\begin{cor}\label{cor:A.B}
There exists $\lambda\in \R$ s.t. for any 
 $f\in C^\infty(Y)$ with  $p|_{\supp(f)}$ being proper, 
we have 
$$
\frac{\tau_*(\sigma^*(f) \cdot |\omega^0_\bfX|)}{\mu_G}
=\lambda
p^*\left(\frac{\pi_*(f \cdot |\omega_\bfY|)}{\mu_C}\right)
$$
\end{cor}
\begin{proof}
    We take $\lambda:= \left(\frac{|\omega_\bfG|}{\mu_G}\right)\left(\frac{\mu_C}{|\omega_\bfC|}\right)$ and use   \Cref{lem:m.base.cha}\eqref{lem:m.base.cha:1}, and the fact that $\bfX$ and $\bfY$ are reduced and the maps in the following diagram are generically smooth.    
$$
\begin{tikzcd}
\bfX \arrow[r,"\sigma"] \arrow[d,"\tau"] & \bfG \arrow[d,"\pi"] \\
\bfY \arrow[r,"p"]           & \bfC
\end{tikzcd}
$$
    See  \Cref{lem:gen.sm}.
\end{proof}

\begin{proof}[Proof of \Cref{thm:bnd4}]
Let $f',h,k$ as in \Cref{thm:bnd3}. 
\footnote{We will also use $\omega_\bfX^0$ from \Cref{thm:bnd3}, but this is the fixed $\omega_\bfX^0$ that we defined in \Cref{not:omegaX0}. Formally speaking, if we just use the formulation of \Cref{thm:bnd3} and not its proof, we can not assume that the form there is the same $\omega_\bfX^0$. However, changing $\gamma$ appropriately, we can assume it WLOG.}
By \Cref{thm:bnd3}
  there exists $\gamma_0\in C^{\infty}(X)$ s.t.

  $$|\langle \chi_\rho, g \mu_G\rangle|\leq 
  \left  \langle \frac{\tau_*(|\omega^0_\bfX| \gamma_0 \sigma^*(f'))}{\mu_G} p^*\left( \frac{p_*(h\mu_G )}{\mu_C}\right)\cR^k, |g| \mu_G\right\rangle.$$

By the assumptions on $f'$ the map
$\tau|_{\supp(\sigma^*(f'))}$ is proper.
Let $\gamma_1\in C^\infty(G)$ be defined by $$\gamma_1(g):=\max_{x\in \supp(\sigma^*(f')) \cap \tau^{-1}(g)} \gamma_0(x).$$
We obtain:
\begin{align*}\langle \chi_\rho,g \mu_G \rangle &\leq\left\langle \gamma_1 
\frac{\tau_*(|\omega^0_\bfX| \sigma^*(f'))}{\mu_G} p^*\left( \frac{p_*(h\mu_G )}{\mu_C}\right)\cR^k, |g| \mu_G\right\rangle
\end{align*}
By \Cref{cor:A.B} there is $\lambda\in \R$ s.t.
$$
\frac{\tau_*(\sigma^*(f') \cdot |\omega^0_\bfX|)}{\mu_G}
=\lambda
p^*\left(\frac{\pi_*(f '\cdot |\omega_\bfY|)}{\mu_C}\right)
$$
Set $\gamma=\lambda \gamma_1 \frac{\mu_G}{|\omega_\bfG|}.$
We obtain:
\begin{align*}\langle \chi_\rho,g \mu_G \rangle &\leq
\left\langle \gamma_1 
\lambda
p^*\left(\frac{\pi_*(f' \cdot |\omega_\bfY|)}{\mu_C}\right) p^*\left( \frac{p_*(h\mu_G )}{\mu_C}\right)\cR^k, |g| \mu_G\right\rangle
\\&=
\left\langle \gamma_1 
\lambda
p^*\left(\frac{\pi_*(f' \cdot |\omega_\bfY|)}{\mu_C}\right) p^*\left( \frac{p_*(h|\omega_\bfG|\frac{\mu_G}{|\omega_\bfG|})}{\mu_C}\right)\cR^k, |g| \mu_G\right\rangle
\\&=
 \left  \langle \gamma p^*\left(\frac{\pi_*(|\omega_\bfY|f')}{\mu_C} \frac{p_*(|\omega_\bfG|h)}{\mu_C}\right)\cR^k, |g| \mu_G\right\rangle.
\end{align*}
as required.

\end{proof}

\subsection{Proof of \Cref{thm:bnd2}}\label{sec:Pfbnd2}

    Let $f',h,
\RamiA{\gamma}$ be as in \Cref{thm:bnd3}.
\Rami{Let $g=\gamma\sigma^*(f')$ and s}et $$e:=\left(
    \frac{\mu_G}{|\omega_\bfG|} \frac{|\omega_\bfC|}{\mu_C}\right)^2
    \RamiA{h \boxtimes_C g} . $$ 
    By \Cref{thm:bnd3} it is enough to show that  
    $$\zeta_*(|\omega_{\bf \Upsilon}|e)=\frac{\tau_*(|\omega^0_\bfX| \RamiA{g})}{\mu_G} p^*\left( \frac{p_*(h\mu_G )}{\mu_C}\right)\RamiA{\mu_G}$$
We have 
\begin{align}\label{eq:pf.ff.1}
    \frac{\tau_*(|\omega^0_\bfX| \RamiA{g})}{\mu_G} p^*\left( \frac{p_*(h\mu_G )}{\mu_C}\right) &=   \frac{\mu_G}{|\omega_\bfG|} \left(\frac{|\omega_\bfC|}{\mu_C}\right)^2
    \left(\frac{\tau_*(|\omega^0_\bfX|\RamiA{g})}{|\omega_\bfG|} \right)  p^*\left(\frac{p_*(|\omega_\bfG|h)}{|\omega_\bfC|}\right)    \end{align}
\RamiA{Consider the Cartesian square:}
    $$
\begin{tikzcd}
\bfG\times_\bfC \bfG \arrow[r,"\fp\fr^{\Dima{2}}_{\bfG}"] \arrow[d,"\fp\fr^{\Dima{1}}_{\bfG}"] & \bfG \arrow[d,"p"] \\
\bfG \arrow[r,"p"]           & \bfC
\end{tikzcd}
$$
\RamiA{By \Cref{lem:gen.sm} and 
\Cref{lem.GG.nice} all the objects in this diagram are varieties and all the maps are generically smooth. Thus by \Cref{lem:m.base.cha}\eqref{lem:m.base.cha:1} we have:}
\begin{equation}\label{eq:pf.ff.2}
p^*\left(\frac{p_*(|\omega_\bfG|h)}{|\omega_\bfC|}\right)=\frac{(\RamiA{\fp\fr^2_\bfG})_*(|\omega_{\bfG}\boxtimes_{\omega_\bfC} \omega_{\bfG}| (\RamiA{\fp\fr^1_\bfG})^*(h))}{|\omega_\bfG|}
\end{equation}

\RamiA{Consider the Cartesian square:}
    $$
\begin{tikzcd}
{\bf \Upsilon} \arrow[r,"p'"] \arrow[d,""] & \bfX \arrow[d,"\pi"] \\
\bfG\times_\bfC \bfG \arrow[r,"\fp\fr^2_{\bfG}"]            & \bfG 
\end{tikzcd}
$$
\RamiA{By \Cref{lem:gen.sm}, \Cref{lem.up.nice}. and 
\Cref{lem.GG.nice} all the objects in this diagram are varieties and all the maps are generically smooth. Thus by \Cref{lem:m.base.cha}\eqref{lem:m.base.cha:2} we have:}
\begin{multline}\label{eq:pf.ff.3}
    \RamiA{
    \frac{\tau_*(|\omega^0_\bfX|\RamiA{g})}{|\omega_\bfG|} \frac{(\RamiA{\fp\fr^2_\bfG})_*(|\omega_{\bfG}\boxtimes_{\omega_\bfC} \omega_{\bfG}| (\RamiA{\fp\fr^1_\bfG})^*(h))}{|\omega_\bfG|}
    }
    =\\
    \frac{\zeta_*(|
    (\omega_{\bfG}\boxtimes_{\omega_\bfC} \omega_{\bfG})
    \RamiA{\boxtimes_{\omega_\bfG} 
    \omega^0_\bfX}|
      (\RamiA{\fp\fr^1_\bfG})^*(h)
     \RamiA{\boxtimes_{G}
     g})     
     }{|\omega_\bfG|}. 
\end{multline}
\RamiA{
Finally:
\begin{align*}
     &\frac{\tau_*(|\omega^0_\bfX| \RamiA{g})}{\mu_G} p^*\left( \frac{p_*(h\mu_G )}{\mu_C}\right)\mu_G \overset{\text{\eqref{eq:pf.ff.1}}}{=}\\=&   \frac{\mu_G}{|\omega_\bfG|} \left(\frac{|\omega_\bfC|}{\mu_C}\right)^2
    \left(\frac{\tau_*(|\omega^0_\bfX|\RamiA{g})}{|\omega_\bfG|} \right)  p^*\left(\frac{p_*(|\omega_\bfG|h)}{|\omega_\bfC|}\right)\mu_G\overset{\text{\eqref{eq:pf.ff.2}}}{=}\\=&
    \frac{\mu_G}{|\omega_\bfG|} \left(\frac{|\omega_\bfC|}{\mu_C}\right)^2
    \left(\frac{\tau_*(|\omega^0_\bfX|\RamiA{g})}{|\omega_\bfG|} \right)\left(  \frac{(\RamiA{\fp\fr^2_\bfG})_*(|\omega_{\bfG}\boxtimes_{\omega_\bfC} \omega_{\bfG}| (\RamiA{\fp\fr^1_\bfG})^*(h))}{|\omega_\bfG|}\right)\mu_G\overset{\text{\eqref{eq:pf.ff.3}}}{=}\\=&
    \frac{\mu_G}{|\omega_\bfG|} \left(\frac{|\omega_\bfC|}{\mu_C}\right)^2
    \frac{\zeta_*(|    (\omega_{\bfG}\boxtimes_{\omega_\bfC} \omega_{\bfG})
    \RamiA{\boxtimes_{\omega_\bfG} 
    \omega^0_\bfX}|
      (\RamiA{\fp\fr^1_\bfG})^*(h)
     \RamiA{\boxtimes_{G}
     g})     
     }{|\omega_\bfG|}\mu_G=    
     \zeta_*(|\omega_{\bf \Upsilon}|e),
\end{align*}
}
    as required.

\section{Bounds on characters in terms of the Chevalley map  - Proof of \Cref{thm:bnd1}}\label{Sec: proof of Theorem E}

Let $f', \gamma$ \RamiC{and $k$} be as in \Cref{thm:bnd4} and let $h'$ be the function $h$ from \Cref{thm:bnd4}.\footnote{We will also use $\omega_\bfY$ from \Cref{thm:bnd3}. However it is just the form
$\omega_\bfY$ defined in 
\Cref{not:omegaX0}. The proof of 
\Cref{thm:bnd1} will work with 
any other form satisfying the assertion of \Cref{thm:bnd4}.}
Set $f'':= \frac{\pi_*(f'|\omega_\bfY|)}{\mu_C}$.
By \Cref{thm:Y.geo.int} there is a resolution of singularities $\delta:\tilde \bfY\to \bfY$ s.t. $\delta^*(\omega_\bfY)$ extends to a regular form $\omega_{\tilde \bfY}$ on $\tilde \bfY$. We obtain:
$f''= \frac{(\pi\circ \delta)_*(\delta^*(f')|\omega_{\tilde \bfY}|)}{\mu_C}$ \DimaB{(almost everywhere)}.
Note that by \Cref{lem:gen.sm} the map $\pi\circ \delta$ is generically smooth. Thus
by
\Cref{thm.IT} this implies that 
$f''\in L_{loc}^{1+2\eps}$ for some $\eps>0$. 

Let $M=\max(\gamma|_{U})$.  
Let 
$\mathcal R_1, \mathcal R_2:G \to \N \cup \{\infty\}$ the functions given by 
$$\mathcal R_1(x)=\max(1,-\min val(x_{ij}))$$ and
$$\mathcal R_2(x)=\max(1,val(det(x)),val(\Delta(x))).$$
Note that $$\cR\leq \cR_1\cR_2.$$
Let $\cR_3\in C^\infty(C^{rss})$ s.t. $p^*(\cR_3)=(\cR_2)^k$. 
Let $N:=\max(((\cR_1)^k)|_U) $.
By \cite[Theorem 1.3]{GH_lct_pos}
 $\cR_3\in L^{<\infty}_{loc}(C)$ and thus (by H{\"o}lder's inequality)  
$f''\cR_3\in L_{loc}^{1+\eps}$. 
Set $$f:=MN  \frac{|\omega_\bfG|}{\mu_G}  \cR_3 f''\cdot 1_{p(U)}$$ and $$h:= h'\cdot 1_{p^{-1}(p(U))}.$$ Using \Cref{thm:bnd4} we obtain:

\begin{align*}
|\langle \chi_\rho, g \mu_G\rangle|
&\leq
\left  \langle \gamma p^*\left(\frac{\pi_*(|\omega_\bfY|f')}{\mu_C} \frac{p_*(|\omega_\bfG|h')}{\mu_C}\right)\cR^k, |g| \mu_G\right\rangle
\\&=
\left  \langle \gamma p^*\left(f''\frac{p_*(|\omega_\bfG|h')}{\mu_C}\right)\cR^k, |g| \mu_G\right\rangle
\\&=
\left  \langle \gamma p^*\left(\frac{|\omega_\bfG|}{\mu_G} f''\frac{p_*(\mu_Gh')}{\mu_C}\right)\cR^k, |g| \mu_G\right\rangle
\\&\leq
\left  \langle p^*\left(M\frac{|\omega_\bfG|}{\mu_G}f''\frac{p_*(\mu_Gh')}{\mu_C}\right)\cR^k, |g| \mu_G\right\rangle
\\&\leq
\left  \langle p^*\left(M\frac{|\omega_\bfG|}{\mu_G}f''\frac{p_*(\mu_Gh')}{\mu_C}\right)\cR_1^k\cR_2^k, |g| \mu_G\right\rangle
\\&\leq
\left  \langle p^*\left(M\frac{|\omega_\bfG|}{\mu_G}f''\frac{p_*(\mu_Gh')}{\mu_C} \cR_3\right), \cR_1^k|g| \mu_G\right\rangle
\\&\leq
\left  \langle p^*\left(MN\frac{|\omega_\bfG|}{\mu_G}f''\frac{p_*(\mu_Gh')}{\mu_C} \cR_3\right), |g| \mu_G\right\rangle
\\&=
\left\langle p^*\left(f\frac{p_*(h \mu_G)}{\mu_C}\right)\mu_G, |g|\right\rangle,
\end{align*}
as required.

\section{Proof of the main results - Theorems \ref{thm:main} and \ref{thm:minor}}\label{Sec: proof of Theorem C and D}

\begin{prop}\label{prop:bnd.implies.L1}
    Let $Z$ be an $F$-analytic variety. 
    Let $\xi \in C^{-\infty}(Z)$.  Let $f\in L^1(Z)$. Assume that for any smooth measure $\rho\in C_c^\infty(Z,D_Z)$  we have 
    \begin{equation}\label{eq:bnd}
        \langle \xi,\rho\rangle \leq \langle f,|\rho|\rangle.
    \end{equation}
     Then there exists a function $g\in L^1(X)$ representing $\xi$.
\end{prop}
\begin{proof}
    Choose an invertible smooth measure on $Z$ and identify $C_c^\infty(Z,D_Z)\cong C_c^\infty(Z)$ and the space of generalized functions with the space of distributions.
    \begin{enumerate}[Step 1.]
        \item $\xi$  (as a functional on $C_c^\infty(Z)$) can be continuously extended to $ C_c(Z)$ (and thus can be considered as a Radon measure on $Z$).\\
        This follows from the fact that $\xi$, as a functional on $\RamiC{C_c(Z)}$, is continuous w.r.t. the induced topology from   $C_c^\infty(Z)$, which follows from the inequality \eqref{eq:bnd}.
        \item For any Borel set  $\RamiC{A} \subset Z$ we have $\RamiC{|\xi(A)|}\leq \int_\Omega f \mu$.\\
        This follows from the inequality \eqref{eq:bnd}. 
        \item $\xi$ is an absolutely continuous measure w.r.t. the Lebesgue measure.\\
        Follows from the previous step.        
        \item End of the proof.\\              
        The assertion follows from the previous item and the Radon-Nikodym theorem.
    \end{enumerate}
\end{proof}
\begin{lem}\label{lem: duality of pull-push}
    Let $\gamma:Z_1\to Z_2$ be a morphism of $F$-analytic varieties. Let $\mu_i$ be nowhere vanishing \RamiC{smooth} measures on $Z_i$.    
    Assume that for any  real valued non-negative \RamiC{function}   $f\in C_c^\infty(Z_1)$ we have $\frac{\gamma_*(f\mu_1)}{\mu_2}\in L^{<\infty}(Z_2)$. Then, for any $\eps>0$ and any real valued non-negative $g\in L^{1+\eps}(Z_2)$ we have 
    $$
    \gamma^*(g)\in L^1_{loc}(Z_1).
    $$ 
\end{lem}
\begin{proof}
For an $F$-analytic variety $Z$, define $Mes_{\geq 0}(Z)$ to be the collection of real valued non-negative measurable functions.
If $Z$ is equipped with a nowhere vanishing \RamiC{smooth} measure $\mu$ we have a natural pairing 
$$B_{(Z,\mu)}:Mes_{\geq 0}(Z) \times Mes_{\geq 0}(Z) \to \mathbb{R}\cup \{\infty\}$$ given by integration:
$$B_{(Z,\mu)}(\phi,\psi)=\int_{Z}\phi\psi \mu$$
Notice that by H{\"o}lder's inequality, for any $\eps>0$, this pairing is finite whenever $\psi \in L^{<\infty}_{c}(Z)$ and $\phi \in L^{1+\eps}_{loc}(Z)$.

Furthermore, to show that $h\in Mes_{\geq 0}(Z)$ is in $L^1_{loc}(Z)$ it is enough to show that $B(\psi,h)<\infty$ for any real valued non-negative $\psi \in C_c^\infty(Z)$. 

\RamiC{The fact that} $\gamma^*(g)\in L^1_{loc}(Z_1)$ \RamiC{follows now from:}
$$\RamiC{\forall f\in C_c^\infty(Z_1) \text{ we have }}
B_{(Z_1,\mu_1)}(f,\gamma^{*}(g))=B_{(Z_2,\mu_2)}\left (\frac{\gamma_*(f\mu_1)}{\mu_2},\RamiC{g}\right).
$$
\end{proof}
\begin{proof}[Proof of \Cref{thm:main}]
    \Cref{thm:alm.an.frs}  and  \Cref{conj:2} imply that
    $p_*$ maps every $C_c^\infty$ measure to a  measure with  $L^{<\infty}$ density.
    By \Cref{lem: duality of pull-push} this implies that for any $\eps>0$ the operation $p^*$ maps $L^{1+\eps}(C)$ function to an $L^1_{loc}(G)$ function.

    Let $U\subset G$ \Dima{be an open compact subset,} and let 
 $\eps$,
    $f$, and
    $h$ be as in \Cref{thm:bnd1}.
    We get that 
    $\frac{p_*(h \mu_G)}{\mu_C}\in L^{<\infty}(C)$.
    Thus, by H{\"o}lder's inequality
    \begin{equation}\label{eq:bnd.before.pull}
    f\frac{p_*(h \mu_G)}{\mu_C}\in L^{1+\frac{\eps}{2}}(C).
    \end{equation}
    We obtain 
    $h':=p^*(f\frac{p_*(h \mu_G)}{\mu_C})\in L^{1}_{loc}(G).$ 
    By \Cref{thm:bnd1}, for any $g\in C^\infty(U)$ we have:
    $$|\langle \chi_\rho, g\mu_G\rangle|\leq  \left\langle h'\mu_G, |g|\right\rangle.$$ 
    So by \Cref{prop:bnd.implies.L1} above we obtain $(\chi_\rho)|_{U}\in L^1(U)$ and we are done.  
\end{proof}
\begin{proof}[Proof of \Cref{thm:minor}]
The proof is the 
same as the proof of \Cref{thm:main}  when we replace \Cref{thm:alm.an.frs} 
by \Cref{thm:uncond.an.frs} and  \Cref{conj:2}  by 
the assumption $\chara(F) >\frac n2$.    
\end{proof}

\begin{remark}\label{rem:intro.Omega1}
    Note that \RamiC{these proofs} also prove \RamiC{Theorems \ref{thm:intro.Omega} and  \ref{thm:intro.Omega.p}, which, using \Cref{prop:bnd.implies.L1} and \cite[Theorem A']{AGKS2_1},  implies \Cref{thm:Lie}}.
\end{remark}

\section{Alternative versions of \Cref{thm:main}}\label{sec:AltMain}
Denote:
\begin{itemize}
    \item $\ug$ - the Lie algebra of $\bfG$
    \item $\uc$ - the affine space of degree $n$ monic polynomials. 
    \item $p_0:\ug\to \uc$ - the Chevalley map.
    \item $\ug_i:=\gl_n^{\times i}
    :=\underbrace{\gl_n\times_{\uc} \dotsc\times_{\uc} \gl_n}_{i \text{ times}}$ considered as an algebraic variety over $\F_\ell$.
\end{itemize}
One can replace the assumption of \Cref{conj:2} in \Cref{thm:main} (and the versions of Theorems \ref{thm:intro.Omega} and \ref{thm:Lie}) with any of the following more precise conditions:
\begin{enumerate}
    \item \label{sec:alt.form:0} For any $i\in \N$, the variety $\ug_i$ admits a strong resolution of singularities.  
        \item\label{sec:alt.form:1} For any $i\in \N$, the defining ideal of $\ug_i$ inside $\ug^{\times i}$ has monomial principalization (see \cite[Definition 12.0.1]{AGKS2}).
    \item \label{sec:alt.form:2}
    \begin{enumerate}
    \item 
     The defining ideal of the nilpotent cone  inside $\gl_n$ has monomial principalization, and
    \item  For any $i\in \N$, the variety, $\ug_i$ has a resolution of singularities (not necessarily a strong one).
    \end{enumerate}
    \item \label{sec:alt.form:3} $\Upsilon$ is geometrically integrable.
    \item \label{sec:alt.form:5} $p$ is almost analytically FRS (see \cite[Definition 1.3.5(3)]{AGKS2}).
\end{enumerate}
Indeed,
\begin{itemize}
    \item The fact that one can replace \Cref{conj:2} with condition \eqref{sec:alt.form:0} follows from the actual formulation of \cite[Theorem D]{AGKS2}.
    \item The fact that one can replace \Cref{conj:2} with any of the conditions (\ref{sec:alt.form:1},\ref{sec:alt.form:2}) follows from the alternative formulations of  \cite[Theorem D]{AGKS2} given in \cite[\S 12]{AGKS2}.
    \item The fact that one can replace \Cref{conj:2} with condition \eqref{sec:alt.form:3} follows from Theorems \ref{thm:bnd2} and \ref{thm.IT}.
        \item The fact that one can replace \Cref{conj:2} with condition \eqref{sec:alt.form:5} follows from the proofs of \Cref{thm:main} and \Cref{thm:minor}.
\end{itemize}
\begin{remark}
$ $
    \begin{itemize}
        \item Note that unlike conditions (\ref{sec:alt.form:0}-\ref{sec:alt.form:2}), condition \eqref{sec:alt.form:3} is not a  special case of \Cref{conj:2} (or its version). However, given an explicit resolution of singularities of $\Upsilon$, it should be easy to check whether condition \eqref{sec:alt.form:3} holds.
        \item In conditions (\ref{sec:alt.form:0}-\ref{sec:alt.form:2}) one can replace the requirement for any $i$, to the value $i=\RamiA{2}^{n^2+\RamiA{3}}$. This follows from \Cref{prop:concrete lower bound} and from the proof of \Cref{thm:main}.
         Indeed, if in the proof of  \Cref{thm:bnd1} we use \Cref{prop:concrete lower bound} instead of \Cref{thm.IT} then we get that in \Cref{thm:bnd1} one can take  
         $\eps=\left(1+((n-1)n +n^2-n)2^{n^2-n}\right)^{-1}$.
         Thus 
in the proof of \Cref{thm:main} it is enough to require that $p_*$ maps any $C_c^\infty$ measure to a measure with  density in $L^N(C)$, where $$N:=\frac{1}{1-\frac{1}{1+\frac{\eps}{2}}}=\frac{2}{\eps}+1<\RamiA{2}^{n^2+\RamiA{3}},$$ in order to get \eqref{eq:bnd.before.pull}.
Now, we need to use \Cref{lem: duality of pull-push}
 for $\tfrac{\eps}{2}$. It is easy to see that in this case we can replace (in \Cref{lem: duality of pull-push}) $L^{<\infty}$ with $L^N$. 
So, we need to show that our weaker assumption still implies the assertion of  \Cref{thm:alm.an.frs} with $L^{<\infty}$ replaced by $L^N$.
This follows from \cite[Theorem D]{AGKS2} and \cite[\S 12]{AGKS2}.
    \end{itemize}
\end{remark}
\appendix
\section{Integrability of pushforward measures in positive characteristic}\label{app:IY}
\begin{center}
\textit{by Itay Glazer and Yotam I.~Hendel}
\end{center}
Let $F$ be a non-Archimedean local field of arbitrary characteristic,
with ring of integers $\mathcal{O}_{F}$ and absolute value $|\cdot|_{F}$,
and let $X$ be an $F$-analytic manifold of dimension $n$. Let $(U_{\alpha}\subset X,\psi_{\alpha}:U_{\alpha}\to F^{n})_{\alpha\in\mathcal{I}}$
be an atlas, and fix a Haar measure $\mu_{F^{n}}$ on $F^{n}$, with
$\mu_{F^{n}}(\mathcal{O}_{F}^{n})=1$. We consider the following spaces
(whose definition is independent of the choice of atlas). 
\begin{enumerate}
\item Let $C^{\infty}(X)$ be the space of smooth (i.e.~locally constant)
complex-valued functions on $X$, and let $C_{c}^{\infty}(X)$ be the subspace of
smooth compactly supported functions. 
\item Let $\mathcal{M}^{\infty}(X)$ be the space of smooth measures on
$X$, i.e.~measures such that each $(\psi_{\alpha})_{*}(\mu|_{U_{\alpha}})$
has a locally constant density with respect to the Haar measure on
$F^{n}$. Let $\mathcal{M}_{c}^{\infty}(X)$ be the subspace of compactly
supported smooth measures. 
\item For $1\leq q\leq\infty$, let $\mathcal{M}_{c,q}(X)$ be the space of
compactly supported Radon measures $\mu$ on $X$ such that for every
$\alpha\in\mathcal{I}$ the measure $(\psi_{\alpha})_{*}(\mu|_{U_{\alpha}})$
is absolutely continuous, and with density in $L^{q}(F^{n})$. 
\end{enumerate}
Given $\mu\in\mathcal{M}_{c,1}(X)$, we define the \emph{integrability exponent}
\[
\epsilon_{\star}(\mu):=\sup\left\{ \epsilon\geq0:\mu\in\mathcal{M}_{c,1+\epsilon}(X)\right\} .
\]

\begin{defn}
Let $\psi:X\to Y$ be an $F$-analytic map between $F$-analytic manifolds
$X,Y$. We say that $\psi$ is \emph{generically submersive} if there
exists an open dense subset $U$ in $X$ such that the differential of $\psi$ at each $x\in U$ is surjective. 
\end{defn}

If $\psi$ is generically submersive, then $\psi_{*}\mu\in\mathcal{M}_{c,1}(Y)$
whenever $\mu\in\mathcal{M}_{c,1}(X)$. In particular, it makes sense
to consider $\epsilon_{\star}(\psi_{*}\mu)$. This leads us to define
the following invariant. 
\begin{defn}
\label{def:epsilon}Let $\psi:X\to Y$ be a generically submersive
$F$-analytic map between $F$-analytic varieties. For each $x_{0}\in X$,
we define the integrability exponent of $\psi$ at $x_{0}$ by 
\begin{equation}
\epsilon_{\star}(\psi;x_{0}):=\sup_{U\ni x_{0}}\inf_{\mu\in\mathcal{M}_{c}^{\infty}(U)}\epsilon_{\star}(\psi_{*}\mu),\label{eq:defeps-phi}
\end{equation}
where the supremum is taken over all open neighborhoods $U$ of $x_{0}$.
We also set 
\begin{equation}
\epsilon_{\star}(\psi):=\inf_{\mu\in\mathcal{M}_{c}^{\infty}(X)}\epsilon_{\star}(\psi_{*}\mu)=\inf_{x\in X}\epsilon_{\star}(\psi;x).\label{eq:def-eps-star}
\end{equation}
\end{defn}

The invariant $\epsilon_{\star}(\psi;x_{0})$ was introduced and explored in \cite{GH21,GHS} in the characteristic zero case\footnote{These works also treat the integrability exponent over Archimedean
local fields.}, where it was shown that $\epsilon_{\star}(\psi;x_{0})$ is a positive
number that can be bounded from below effectively. This was used in
\cite{GGH} to study integrability of Harish-Chandra characters of
representations of reductive groups over local fields of characteristic
zero.

The aim of this appendix
is to establish a similar bound on $\epsilon_{\star}(\psi;x_{0})$
over local fields of positive characteristic. We start our discussion
by noting that when $\mathrm{char}(F)\neq0$, non-\RamiA{constant} analytic maps
$f:F^{n}\rightarrow F$ need not be generically submersive. 
\begin{example}
\label{exa:pathology}Let $p$ be a prime and let $f(x)=x^{p}$. Then
$d_{x}f=px^{p-1}=0$ for every $x\in\mathbb{F}_{p}[[t]]$, so $f:\mathbb{F}_{p}[[t]]\rightarrow\mathbb{F}_{p}[[t]]$
is not generically submersive. Moreover, if we take $\mu=\mu_{\mathbb{F}_{p}[[t]]}$,
then $f_{*}\mu_{\mathbb{F}_{p}[[t]]}$ is supported on the set of
$p$-th powers $\left\{ \sum_{i=0}^{\infty}a_{i}t^{pi}:a_{i}\in\mathbb{F}_{p}\right\} \subseteq\mathbb{F}_{p}[[t]]$,
and thus $f_{*}\mu_{\mathbb{F}_{p}[[t]]}$ is not absolutely continuous
with respect to $\mu_{\mathbb{F}_{p}[[t]]}$.
\end{example}

We recall the following notion from \cite{GH_lct_pos}. 
\begin{defn}[{\cite[Definition 1.1]{GH_lct_pos}}]
\label{def:lct}Let $X$ be an $F$-analytic manifold, let $x_{0}\in X$
and let $f_{1},\ldots,f_{r}:X\to F$ be $F$-analytic functions generating
a non-zero ideal $J$ (in the ring of $F$-analytic functions on $X$).
We define the $F$\textit{-analytic log-canonical threshold} of $J$
at $x_{0}$ by 
\[
\lct_{F}(J;x_{0}):=\sup\left\{ s>0:\exists U\ni x_{0}\text{\,\,\,s.t.\,\,\,}\forall\mu\in\mathcal{M}_{c}^{\infty}(U),\int_{X}\min\limits_{1\le i\le r}|f_{i}(x)|_{F}^{-s}\mu(x)<\infty\right\} ,
\]
where $U$ in the definition above is an open neighborhood of $x_{0}$.
\end{defn}

\begin{defn}
\label{def:Jacobian ideal}Given a generically submersive map $\psi:X\to Y$
between $F$-analytic manifolds, we write $\mathcal{J}_{\psi}$ for
the \emph{Jacobian ideal sheaf} of $\psi$. Locally, if $X\subseteq F^{n}$
and $Y\subseteq F^{m}$ are open subsets, $\mathcal{J}_{\psi}$ is
the ideal in the algebra of analytic functions on $X$ generated by
the $m\times m$-minors of $d_{x}\psi$. This construction is invariant
under analytic coordinate changes and defines an ideal sheaf
on $X$.
\end{defn}

The following are the main results of this appendix.
\begin{thm}
\label{thm:main theorem}Let $\psi:X\rightarrow Y$ be an $F$-analytic
map between $F$-analytic manifolds. Suppose that $\psi$ is generically
submersive. Then for every $x_{0}\in X$, there exists $\epsilon_{x_{0}}>0$
such that 
\[
\epsilon_{\star}(\psi;x_{0})\ge\lct_{F}(\mathcal{J}_{\psi};x_{0})\geq\epsilon_{x_{0}}.
\]
\end{thm}

Given a generically smooth morphism $\varphi:X\to Y$ of smooth algebraic
$F$-varieties, we get an $F$-analytic map $\varphi_{F}:X(F)\to Y(F)$,
which is generically submersive. In this setting, we have a uniform
lower bound on $\epsilon_{\star}(\varphi_{F};x_{0})$. 
\begin{thm}
\label{thm:main theorem for alg maps}Let $\varphi:X\rightarrow Y$
be a generically smooth morphism between smooth algebraic $F$-varieties.
Then there exists $\epsilon>0$ depending 
only on the complexity class\footnote{For a precise definition of complexity, we refer to \cite[Definition 7.7]{GH19}.}
 of $\varphi:X\rightarrow Y$ such that
\[
\epsilon_{\star}(\varphi_{F})>\epsilon.
\]
\end{thm}

The following proposition gives a concrete lower bound on $\epsilon_{\star}(\varphi_{F})$ using the data defining $\varphi$.
\begin{prop}
\label{prop:concrete lower bound}
Let $X,Y$ and $\varphi$ be as in  Theorem \ref{thm:main theorem for alg maps}. 
Suppose that:
\begin{enumerate}
\item $X' \subseteq \mathbb{A}_{F}^{n_{1}+m_{1}}$ is a closed (possibly singular) subvariety 
of dimension $n_1$ cut by polynomials $g_{1}=\ldots=g_{r_{1}}=0$ of degree at most $ d_{1}$, and $X \subseteq X'$ an open affine subvariety.
\item $Y \subseteq \mathbb{A}_{F}^{n_{2}+m_{2}}$ is a closed subvariety, admitting an \'etale map $\pi:Y\rightarrow\mathbb{A}_{F}^{n_{2}}$ where
$\pi_{1},\ldots,\pi_{n_{2}}$ are polynomials of degree at most $ d_{2}$
(locally it is the case, since $Y$ is smooth). 
\item We have $\varphi=\Phi|_{X}$, 
where  $\Phi:\mathbb{A}_{F}^{n_{1}+m_{1}}\rightarrow \mathbb{A}_{F}^{n_{2}+m_{2}}$ is a polynomial map of degree $d$.
\end{enumerate}
Then: 
\[
\epsilon_{\star}(\varphi_{F})\geq\frac{1}{\left((d\cdot d_{2}-1)\cdot n_{2}+(d_{1}-1)m_{1}\right)\cdot d_{1}^{m_{1}}}.
\]
\end{prop}

Theorems \ref{thm:main theorem} and \ref{thm:main theorem for alg maps}
work over all local fields, where the new aspect is the proof for
local fields of positive characteristic. The inequality $\epsilon_{\star}(\psi;x_{0})\ge\lct_{F}(\mathcal{J}_{\psi};x_{0})$
follows similarly to \cite[Theorem 1.1]{GHS}. The inequality $\lct_{F}(\mathcal{J}_{\psi};x_{0})\geq\epsilon_{x_{0}}$
follows from \cite{GH_lct_pos}, where new methods are required
to deal with local fields of positive characteristic. These results
complement \cite[Theorem 1.1]{GHS}, which was proven in the characteristic
zero case.

Finally, as the next example shows, we note that in the setting of Theorem \ref{thm:main theorem},
$\epsilon_{\star}(\psi)$ might not be strictly positive without an
additional assumption. 
\begin{example}
\label{exa:pathological example}
Fix a prime $p$ and set $X=Y=F=\mathbb{F}_{p}((t))$. For $k\ge1$, set $d_{k}=p^{k}+1$ and $U_{k}:=\{x\in F:\left|x-t^{-k}\right|_{F}\leq1\}$.
Then the subsets $\left\{ U_{k}\right\} _{k=1}^{\infty}$ are disjoint.
Define $\psi:X\to Y$ by
\[
\psi(x)=\begin{cases}
x & \text{if }x\notin\bigcup_{k=1}^{\infty}U_{k},\\
\left(x-t^{-k}\right)^{d_{k}} & \text{if }x\in U_{k}. 
\end{cases}
\]
Then $\psi$ is generically submersive, and by Proposition \ref{prop:reduction to local statement}
we have $\epsilon_{\star}(\psi|_{U_{k}})=\frac{1}{d_{k}-1}=p^{-k}$.
In particular, $\epsilon_{\star}(\psi)=0$. 
\end{example}

\subsection*{Acknowledgement}
I.G.~was supported by ISF grant 3422/24.

\subsection{Proof of the main theorems}
\begin{lem}
\label{lem:Jensen}Let $\psi:X\rightarrow Y$ be a submersion of $F$-analytic
manifolds. Then
\[
\epsilon_{\star}(\psi_{*}\mu)\geq\epsilon_{\star}(\mu)
\]
for every $\mu\in\mathcal{M}_{c,1}(X)$, with equality if $\psi$
is a local diffeomorphism. 
\end{lem}

\begin{proof}
It is clear that $\epsilon_{\star}(\psi_{*}\mu)=\epsilon_{\star}(\mu)$
if $\psi$ is a local diffeomorphism. Since $\mu$ is compactly supported,
by working locally using the local submersion theorem (see e.g.~\cite[III, p.85]{Ser92}),
we may assume that $\psi:F^{n}\rightarrow F^{m}$ is the projection
to the last $m$ coordinates, with $n\geq m$. For simplicity write
$x=(x_{1},\ldots,x_{n-m})$, $y=(x_{n-m+1},\ldots,x_{n})$, so that $\psi(x,y)=y$.
Write $\mu=f(x,y)\mu_{F^{n}}$ and $\psi_{*}\mu=h(y)\mu_{F^{m}}$.
Let $B\subseteq F^{n-m}$ be a ball which contains the projection
of $\mathrm{supp}(\mu)$ to the last $n-m$ coordinates $F^{n-m}$.
Let $C:=\mu_{F^{n-m}}(B)$. Then by Jensen's inequality, for every
$s>0$, we have: 
\begin{align*}
\int_{F^{m}}h(y)^{1+s}dy & =\int_{F^{m}}\left(\int_{F^{n-m}}f(x,y)dx\right)^{1+s}dy\\
&=\int_{F^{m}}\mu_{F^{n-m}}(B)^{1+s}\left(\frac{1}{\mu_{F^{n-m}}(B)}\int_{B}f(x,y)dx\right)^{1+s}dy\\
 & \leq C^{s}\int_{F^{m}}\int_{F^{n-m}}f(x,y)^{1+s}dxdy=C^{s}\int_{F^{n}}f(x,y)^{1+s}\mu_{F}^{n}.
\end{align*}
This concludes the proof. 
\end{proof}
We next reduce Theorem \ref{thm:main theorem} to Proposition \ref{prop:reduction to local statement}
below.  
Recall that a power series $f(x_{1},\ldots,x_{n}):=\sum_{I\in\Z_{\geq0}^{n}}a_{I}x^{I}\in F\langle x_{1},\ldots,x_{n}\rangle$
is called \emph{strictly convergent} if $a_{I}\underset{\left|I\right|\rightarrow\infty}{\longrightarrow}0$
(see \cite[Definition 2.1(2)]{GH_lct_pos}). 
\begin{prop}
\label{prop:reduction to local statement}Let $\psi:X\rightarrow F^{m}$
be a generically submersive $F$-analytic map, where $X\subseteq\mathcal{O}_{F}^{n}$
is an open compact neighborhood of $0$, and such that $\psi=(\psi_{1},\ldots,\psi_{m})$,
where $\psi_{i}:X\rightarrow F$ is given by strictly
convergent power series centered at $0$. Then
\[
\epsilon_{\star}(\psi;0)\ge\lct_{F}(\mathcal{J}_{\psi};0)>0,
\]
with equality if $m=n$.
\end{prop}

\begin{proof}[Proposition \ref{prop:reduction to local statement} implies Theorem
\ref{thm:main theorem}]
Let $\psi:X\to Y$ be a generically submersive map. Note that if
$\phi_{1}:X'\underset{\simeq}{\longrightarrow}U$ and $\phi_{2}:V\underset{\simeq}{\longrightarrow}Y'$
are diffeomorphisms, for open neighborhoods $x_{0}\in U\subseteq X$
and $\psi(x_{0})\in V\subseteq Y$, then 
\[
\epsilon_{\star}(\psi;x_{0})=\epsilon_{\star}(\phi_{2}\circ\psi\circ\phi_{1};\phi_{1}^{-1}(x_{0})).
\]
Hence, by analytic change of coordinates, we may assume that $x_{0}=0$,
$X\subseteq\mathcal{O}_{F}^{n}$ is an open compact neighborhood of
$0$, and that $Y=F^{m}$, with $n\ge m$. Since $\psi$ is analytic
near $0$, by shrinking $X$, we may assume that $\psi=(\psi_{1},\ldots,\psi_{m})$,
where each $\psi_{i}:X\subseteq F^{n}\rightarrow F$ is given by a
converging power series centered at $0$. 
Let $\varpi_{F}$ be a uniformizer of $F$. 
By altering $X$ as follows, we may assume that each $\psi_{i}$ converges on $\mathcal{O}_{F}^{n}$,
and therefore each $\psi_{i}$ is strictly convergent (see e.g.~\cite[Section 5.1.4, Proposition 1]{BGR84}). First, we further shrink $X$ such that $\varpi_{F}^{-k}X\subseteq\mathcal{O}_{F}^{n}$.
Then, we may apply a change of coordinates of the form $(x_{1},\ldots,x_{n})\mapsto (\varpi_{F}^{k}x_{1},\ldots,\varpi_{F}^{k}x_{n})$
for $k\in\N$, and replace $\psi$ with $\widetilde{\psi}(x_{1},\ldots,x_{n}):=\psi(\varpi_{F}^{k}x_{1},\ldots,\varpi_{F}^{k}x_{n})$. 
Thus, we have reduced Theorem \ref{thm:main theorem} precisely to
the setting of Proposition \ref{prop:reduction to local statement}.
\end{proof}
\begin{lem}
\label{lem:reduction to lct}In the setting of Proposition \ref{prop:reduction to local statement},
with $m=n$, we have: 
\[
\epsilon_{\star}(\psi;0)=\operatorname{lct}_{F}(\operatorname{Jac}_{x}(\psi);0),
\]
where $\operatorname{Jac}_{x}(\psi):=\det(d_{x}(\psi))$ is the Jacobian
determinant at $x$. 
\end{lem}

\begin{proof}
Since $\psi$ is generically submersive, there is an open dense subset
$U\subseteq X$, where $\operatorname{Jac}_{x}(\psi)\neq0$, for every
$x\in U$. By the inverse mapping theorem \cite[p.\,73]{Ser92}, $\psi|_{U}:U\rightarrow F^{m}$
is a local diffeomorphism. By \cite[Theorem 1]{Lip84}, since $\psi_{i}$
is strictly convergent for $1\leq i\leq m$, there exists $L\in\N$
such that $\#\left\{ \psi^{-1}(\psi(x))\right\} \leq L$ for every
$x\in U$. From here, the proof of the lemma is identical to the proof
of \cite[Proposition 4.1]{GHS}. In particular, for every $\mu\in\mathcal{M}_{c,\infty}(X)$,
if $\psi_{*}\mu=g(y)\cdot\mu_{F^{n}}$, we get 
\begin{equation}
\int_{X}\frac{1}{\left|\operatorname{Jac}_{x}(\psi)\right|_{F}^{s}}\mu(x)\leq\int_{Y}g(y)^{1+s}dy\leq L^{s}\int_{X}\frac{1}{\left|\operatorname{Jac}_{x}(\psi)\right|_{F}^{s}}\mu(x).\qedhere\label{eq:concrete estimate}
\end{equation}
 
\end{proof}
We can now prove Proposition \ref{prop:reduction to local statement}
and deduce Theorem \ref{thm:main theorem}. 
\begin{proof}[Proof of Proposition \ref{prop:reduction to local statement}]
Let $\psi:X\to F^{m}$ be as in Proposition \ref{prop:reduction to local statement}.
The inequality $\lct_{F}(\mathcal{J}_{\psi};0)>0$ follows from \cite[Theorem 1.2]{GH_lct_pos}. It is left to prove that $\epsilon_{\star}(\psi;0)\ge\lct_{F}(\mathcal{J}_{\psi};0)$. 

Since $\psi$ is generically submersive, $U:=\left\{ x\in X:\mathrm{rk}(d_{x}\psi)=m\right\}$
is open and dense in $X$. 
Denote by $\mathcal{A}_m$ the set of subsets $I=\{i_{1},\ldots,i_{m}\}$ of $\{1,\ldots,n\}$. For each $I \in \mathcal{A}_m$ 
let $M_{I}$ be the corresponding $m\times m$-minor of $d_{x}\psi$.
Fix $s<\lct_{F}(\mathcal{J}_{\psi};0)$. By Definition \ref{def:lct},
there exists an open compact subset $0\in U'\subseteq X$ such that
\begin{equation}
\forall\mu'\in\mathcal{M}_{c}^{\infty}(U'),\int_{X}\min\limits_{I\in \mathcal{A}_m}|M_{I}(x)|_{F}^{-s}\mu'(x)<\infty.\label{eq:definition of lct of Jac}
\end{equation}
For each $I \in \mathcal{A}_m$, set
\[
U_{I}:=\left\{ x\in U'\cap U:\max_{I'\in \mathcal{A}_m}\left|M_{I'}(x)\right|_{F}=\left|M_{I}(x)\right|_{F}\right\}.
\]
We may refine the cover $\bigcup_{I\in \mathcal{A}_m}U_{I}$ into
a disjoint cover $\bigcup_{I\in \mathcal{A}_m}V_{I}$, where $V_{I}\subseteq U_{I}$ is a measurable subset. Set $J=\{j_{1},\ldots.,j_{n-m}\}:=\{1,\ldots,n\}\backslash I$
and consider the map $\psi_{I}:V_{I}\to F^{n}$ given by $\psi_{I}(x):=(\psi(x),x_{j_{1}},\ldots,x_{j_{n-m}})$. 

Let $\mu\in\mathcal{M}_{c}^{\infty}(U')$ and denote $\mu_{I}:=1_{V_{I}}\cdot\mu$.
Since $\bigcup_{I\in \mathcal{A}_m}V_{I}$ is of full measure in
$U'$, we can write $\mu=\sum_{I}\mu_{I}$. We can further write:
\[
\psi_{*}\mu=g(y)\cdot\mu_{F^{m}},\,\,\,\psi_{*}\mu_{I}=g_{I}(y)\cdot\mu_{F^{m}}\text{\,\, and \,\,}\left(\psi_{I}\right)_{*}\mu_{I}=\widetilde{g}_{I}(z)\cdot\mu_{F^{n}}
\]
where
\begin{equation}
\widetilde{g}_{I}(z)=\sum_{x\in\psi_{I}^{-1}(z)}\left|\operatorname{Jac}_{x}(\psi_{I})\right|_{F}^{-1}=\sum_{x\in\psi_{I}^{-1}(z)}\left|M_{I}(x)\right|_{F}^{-1}.\label{eq:formula for equidimension}
\end{equation}
It is enough to show that $\int_{F^{m}}g(y)^{1+s}\mu_{F^{m}}<\infty$
for each $0<s<\lct_{F}(\mathcal{J}_{\psi};0)$ as above.

By Jensen's inequality, there exists $C_{1}(s)>0$ such that:
\begin{equation}
\int_{F^{m}}g(y)^{1+s}dy=\int_{F^{m}}\left(\sum_{I\in \mathcal{A}_m}g_{I}(y)\right)^{1+s}dy\leq C_{1}(s)\sum_{I}\int_{F^{m}}g_{I}(y)^{1+s}dy.\label{eq:0.6}
\end{equation}
Let $q :F^n \to F^m$ be the projection to the first $m$ coordinates. 
Since $\psi|_{V_{I}}=q\circ\psi_{I}$, we have: 
\[
g_{I}(y)=\int_{F^{n-m}}\widetilde{g}_{I}(y,z_{m+1},\ldots,z_{n})dz.
\]
Using Jensen's inequality as in the proof of Lemma \ref{lem:Jensen}, there exists $C_{2}(s)>0$ (depending
on $\psi(U')$) such that 
\begin{equation}
\int_{F^{m}}g_{I}(y)^{1+s}dy\leq C_{2}(s)\int_{F^{n}}\widetilde{g}_{I}(z)^{1+s}dz.\label{eq:0.7}
\end{equation}
Taking $L\in\N$ such that $\#\left\{ \psi_{I}^{-1}(\psi_{I}(x))\right\} \leq L$
for every $x\in V_{I}$ and every $I$, and using (\ref{eq:formula for equidimension}),
similarly to (\ref{eq:concrete estimate}), we get: 
\begin{equation}
\int_{F^{n}}\widetilde{g}_{I}(z)^{1+s}dz\leq L^{s}\int_{U'}\left|M_{I}(x)\right|_{F}^{-s}\mu_{I}\leq L^{s}\int_{U'}\min_{I\in \mathcal{A}_m}\left[\left|M_{I}(x)\right|_{F}^{-s}\right]\mu<\infty.\label{eq:0.8}
\end{equation}
Combining (\ref{eq:0.6}),(\ref{eq:0.7}) and (\ref{eq:0.8}), we get
\begin{align*}
\int_{F^{m}}g(y)^{1+s}dy & \leq C_{1}(s)C_{2}(s)\sum_{I}\int_{F^{n}}\widetilde{g}_{I}(z)^{1+s}dz\\
 & \leq C_{1}(s)C_{2}(s)L^{s}{n \choose m}\int_{U'}\min_{I\in \mathcal{A}_m}\left[\left|M_{I}(x)\right|_{F}^{-s}\right]\mu<\infty.\qedhere
\end{align*}
\end{proof}
\begin{proof}[Proof of Theorem \ref{thm:main theorem for alg maps}]
By \cite[Theorem 1.3]{GH_lct_pos}, there exists $\epsilon>0$
depending only on $\varphi$, such that for every $x_{0}\in X(F)$,
\[
\lct_{F}(\mathcal{J}_{\varphi_{F}};x_{0})>\epsilon.
\]
By Theorem \ref{thm:main theorem}, we get that $\epsilon_{\star}(\varphi_{F})\geq\epsilon>0$.
\end{proof}
We finish with a proof of Proposition \ref{prop:concrete lower bound}.
\begin{proof}[Proof of Proposition \ref{prop:concrete lower bound}] 
Fix $x_0 \in X(F)$. 
Since $\pi:Y\rightarrow\mathbb{A}_{F}^{n_{2}}$ is \'etale, the map
$\pi_{F}:Y(F)\rightarrow F^{n_{2}}$ is a local diffeomorphism, and
hence $\epsilon_{\star}(\varphi_{F};x_{0})=\epsilon_{\star}((\pi\circ\varphi)_{F};x_{0})$.
By our assumption, the morphism $\widetilde{\varphi}:=\pi\circ\varphi:X\rightarrow\mathbb{A}_{F}^{n_{2}}$
is a restriction of a polynomial map $(\widetilde{\varphi}_{1},\ldots,\widetilde{\varphi}_{n_{2}}):\mathbb{A}_{F}^{n_{1}+m_{1}}\rightarrow\mathbb{A}_{F}^{n_{2}}$,
where each $\widetilde{\varphi}_{i}$ is of degree $\leq d\cdot d_{2}$. 
Since $\widetilde{\varphi}$ is generically smooth, there exists $I'=\{i'_{1},\ldots,i'_{n_{1}-n_{2}}\}\subseteq\{1,\ldots,n_{1}+m_{1}\}$
such that the map $\eta:X\rightarrow\mathbb{A}_{F}^{n_{1}}$ given
by 
\[
\eta(x_{1},\ldots,x_{n_{1}+m_{1}}):=(\widetilde{\varphi}(x),x_{i'_{1}},\ldots,x_{i'_{n_{1}-n_{2}}}),
\]
is generically \'etale. 
Let $q:\mathbb{A}_{F}^{n_{1}}\rightarrow\mathbb{A}_{F}^{n_{2}}$
be the projection to the first $n_{2}$ coordinates.
Note that $\widetilde{\varphi}=q\circ\eta$.
  By Lemma \ref{lem:Jensen}
and Proposition \ref{prop:reduction to local statement}, we have,
\[
\epsilon_{\star}(\widetilde{\varphi}_{F};x_{0})\geq\epsilon_{\star}(\eta_{F};x_{0})=\operatorname{lct}_{F}(\operatorname{Jac}_{x}(\eta_{F});x_{0}).
\]
Since $X$ is a smooth open subvariety of $X' \subseteq \mathbb{A}_{F}^{n_{1}+m_{1}}$ of dimension $n_1$, and $X'$ is 
cut by $g_{1}=\ldots=g_{r_{1}}=0$, the tangent space $T_{x_{0}}X$ 
is given by $d_{x_{0}}g=0$,  where $d_{x_{0}}g$
is a matrix of size $(n_{1}+m_{1})\times r_{1}$ of (maximal) rank $m_{1}$. 
Therefore, we may choose $m_1$ polynomials out of $\{ g_1,\ldots,g_{r_1}\}$  
such that their common zero locus $\widetilde{X}\supseteq X$ is of dimension $n_1$, and where $x_0 \in \widetilde{X}(F)$ is a smooth point.  
Without loss of generality, we may take these polynomials to be $g_1,\ldots,g_{m_1}$. 
Since $\widetilde{X}$ is smooth at $x_0$, it is locally irreducible there. 
Thus, there exist $I=\{i_{1},\ldots,i_{m_{1}}\}\subseteq\{1,\ldots,n_{1}+m_{1}\}$ 
and a Zariski open set $x_0 \in U_{I} \subseteq X$ on which the $I\times \{1,\ldots,m_1 \}$-minor of $d_{x}g$ is non-vanishing, and such that $dx_{j_{1}}\wedge\ldots\wedge dx_{j_{n_{1}}}$ is a non-vanishing top form on $U_I$, where $J=\{1,\ldots,n_{1}+m_{1}\}\setminus I$.
We get that 
\begin{equation}
\operatorname{Jac}_{x}(\eta)=\frac{d\widetilde{\varphi}_{1}\wedge\ldots\wedge d\widetilde{\varphi}_{n_{2}}\wedge dx_{i'_{1}}\wedge\ldots\wedge dx_{i'_{n_{1}-n_{2}}}}{dx_{j_{1}}\wedge\ldots\wedge dx_{j_{n_{1}}}}.\label{eq:effective Jacobian}
\end{equation}
Multiplying the $n_{1}$-forms at the numerator and denominator of
(\ref{eq:effective Jacobian}) by $dg_{1}\wedge\ldots\wedge dg_{m_{1}}$,
we get:
\[
\operatorname{Jac}_{x}(\eta)=\frac{d\widetilde{\varphi}_{1}\wedge\ldots\wedge d\widetilde{\varphi}_{n_{2}}\wedge dx_{i'_{1}}\wedge\ldots\wedge dx_{i'_{n_{1}-n_{2}}}\wedge dg_{1}\wedge\ldots\wedge dg_{m_{1}}}{dx_{j_{1}}\wedge\ldots\wedge dx_{j_{n_{1}}}\wedge dg_{1}\wedge\ldots\wedge dg_{m_{1}}}.
\]
Set $\psi:U_{I}\rightarrow\mathbb{A}_{F}^{1}$ by
\[
\psi(x):=\frac{d\widetilde{\varphi}_{1}\wedge\ldots\wedge d\widetilde{\varphi}_{n_{2}}\wedge dx_{i'_{1}}\wedge\ldots\wedge dx_{i'_{n_{1}-n_{2}}}\wedge dg_{1}\wedge\ldots\wedge dg_{m_{1}}}{dx_{1}\wedge\ldots\wedge dx_{n_{1}+m_{1}}}.
\]
Since by our construction, $dx_{j_{1}}\wedge\ldots\wedge dx_{j_{n_{1}}}\wedge dg_{1}\wedge\ldots\wedge dg_{m_{1}}$
is a non-vanishing top form of $\mathbb{A}_{F}^{n_{1}+m_{1}}$ near
$x_{0}$ it follows that 
\[
\operatorname{lct}_{F}(\operatorname{Jac}_{x}(\eta_{F});x_{0})=\operatorname{lct}_{F}(\psi_{F}(x);x_{0}).
\]
Since $\psi(x_{1},...,x_{n+m})$ is a polynomial of degree $\leq(d\cdot d_{2}-1)\cdot n_{2}+(d_{1}-1)m_{1}$
it follows by \cite[Theorem 1.4]{GH_lct_pos} that: 
\[
\epsilon_{\star}(\varphi_{F};x_{0})\geq\operatorname{lct}_{F}(\operatorname{Jac}_{x}(\eta_{F});x_{0})=\operatorname{lct}_{F}(\psi_{F}(x);x_{0})\geq\frac{1}{\left((d\cdot d_{2}-1)\cdot n_{2}+(d_{1}-1)m_{1}\right)\cdot d_{1}^{m_{1}}}.\qedhere
\]
\end{proof}

\section{Explanation of the mistake in \cite{Le4}}\label{app:lem}

  The arguments in \cite{Le4} and its sequels were based on a construction of a certain submersion that replaces the Luna slice for closed orbits which are not semi-simple, see \cite[\S2.2]{Le4}. A key  property of this submersion is described in \cite[Lemma 2.3.2]{Le4}. The formulation of this Lemma is inconsistent.  Namely, a certain set (denoted there by $U_b'\cap U_c'$) is discussed in \cite[Lemma 2.3.2(2)]{Le4}. It is implicitly assumed that this set is open both in  $U_b'$ and $U_c'$ (as a function in $C_c^\infty(U_b'\cap U_c')$ is considered both as a function on $U_b'$ and $U_c'$) which is wrong in general.

  A version of \cite[Lemma 2.3.2]{Le4} with a consistent formulation is \cite[Lemma 5.4.2]{Le5}. However this lemma is false as stated.  
\section{Diagrams}\label{sec:diag}
For the convenience of the reader, we present here several diagrams of objects frequently used in the paper.
\subsection{The main varieties in the paper}
$$
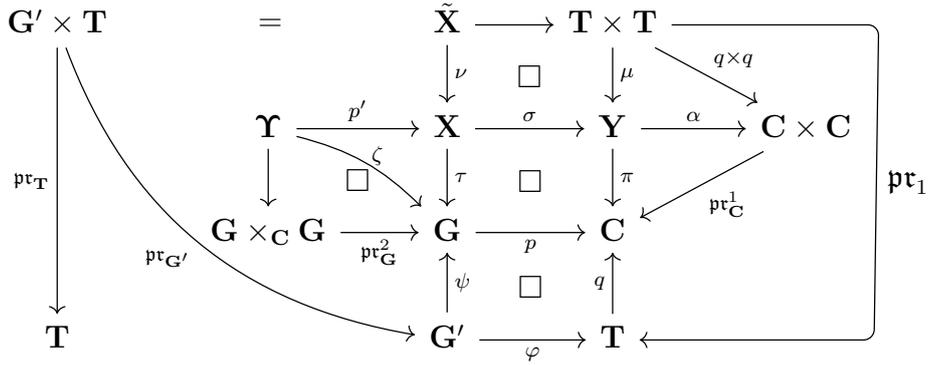
\begin{tikzcd}
    \bfG'\times \bfT\arrow[swap,"\mathfrak{pr}_{\bfT}"]{ddd}\arrow[swap,bend right=30, "\mathfrak{pr}_{\bfG'}"]{dddrr}\arrow[phantom, "="]{rr}
    &
    &\tilde \bfX \drar[phantom, "\square"]\rar[""]\dar["\nu"] & \bfT\times \bfT\dar["\mu"]\drar["q\times q"] 
    \arrow[rounded corners, to path={ -- ([xshift=15ex]\tikztostart.east) -- 
    node[right]{$\mathfrak{pr}_1$}
    ([xshift=17ex]\tikztotarget.east) -- (\tikztotarget)}]{ddd}&  &
    \\      
  &{\bf \Upsilon}\rar["p'"] \dar[] \drar[phantom, "\square"] \drar[bend left=15, "\zeta"{yshift=-4}]
  &\bfX \drar[phantom, "\square"]\rar["\sigma"]\dar["\tau"] & \bfY\rar["\alpha"]\dar["\pi"] & \bfC\times \bfC  \dlar["\fp\fr_\bfC^1"] & 
  \\  
&\bfG\times_\bfC\bfG\rar[swap, "\fp\fr_\bfG^2"]&\bfG\drar[phantom, "\square"]\rar[swap, "p"] & \bfC &    & 
\\\bfT 
&&\bfG'\uar[swap,"\psi"]\rar[swap, "\varphi"] & \bfT\uar["q"] &   & 
\end{tikzcd}
$$



\subsection{Open subsets inside the varieties (mainly used in \S\S\ref{sec:basic}-\ref{sec:kap.form})}

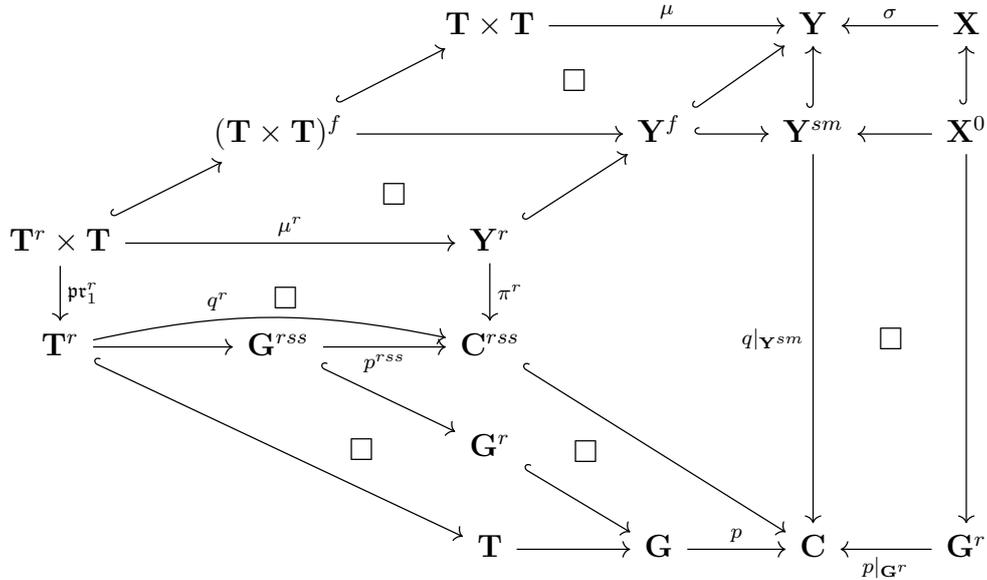
$$
\begin{tikzcd}
& & \bfT \times \bfT \arrow[rr, "\mu"]\arrow[dr,phantom,"\square"] & & \bfY &\bfX \arrow[l,"\sigma"']\\
& (\bfT \times \bfT)^f\arrow[dr,phantom,"\square"] \arrow[ur, hook] \arrow[rr] &  & \bfY^f \arrow[ur, hook]\arrow[r, hook] & \bfY^{sm} \arrow[u, hook]\arrow[dddd, "q|_{\bfY^{sm}}"']& \bfX^0\arrow[u,hook]\arrow[dddd]\arrow[l]\arrow[ddddl,phantom, "\square"]\\
\bfT^r \times \bfT\arrow[drr,phantom,"\square"] \arrow[d, "\fp\fr_1^r"] \arrow[ur, hook] \arrow[rr, "\mu^r"]  & & \bfY^r \arrow[ur, hook] \arrow[d, "\pi^r"] & &&\\
\bfT^r \arrow[ddrrr,phantom,"\square"]\arrow[rr, bend left=12, pos=0.4,"q^r"] \arrow[r] \arrow[ddrr,hook] & \bfG^{{rss}}\arrow[ddrrr,phantom,"\quad\quad\square"] \arrow[r,"p^{rss}"'] \arrow[dr,hook] & \bfC^{{rss}} \arrow[ddrr, hook] & &&\\
 & & \bfG^r \arrow[dr, hook] & 
& &\\
& & \bfT \arrow[r] & \bfG \arrow[r, "p"] & \bfC & \bfG^{r}\arrow[l, "p|_{\bfG^{r}}"]
\end{tikzcd}
$$

\subsection{The sets $\mathcal{A}$ and $\mathcal{B}$
(mainly used in \S\S\ref{sec:kap.form},\ref{sec:Pfbnd234})}

$$
\begin{tikzcd}
    \mathcal{A}\arrow[d,hook] \arrow[dr, phantom, "\square"]\arrow[r]&\mathcal{B} \arrow[d,hook] \arrow[r]
      \arrow[dr, phantom, "\square"]
    & C \times \mathbf{C}(O_F) \arrow[d,hook] \\
    X\arrow[r,"\sigma"]\arrow[d,"\tau"] \arrow[dr, phantom, "\square"]&Y\arrow[d,"\pi"]  \arrow[r, "\alpha"] & C \times C
    \\
    G\arrow[r]&C
\end{tikzcd}
$$

\begingroup
  \let\clearpage\relax
  \let\cleardoublepage\relax 
  \printindex
\endgroup
\bibliographystyle{alpha}
\bibliography{Ramibib}

\end{document}